\documentclass[12pt,final]{amsart}
\usepackage{pdfsync}
\usepackage{amssymb,euscript,nicefrac}
\usepackage{amscd,txfonts,yfonts}
\usepackage{latexsym,todonotes}
\usepackage{epsfig,pxfonts,marvosym}
\usepackage{amsfonts}
\usepackage{xypic}
\usepackage[automake,sort=use,style=super3col,postdot,docdef=true]{glossaries-extra}
 \makeglossaries

\usepackage{bbm,ifpdf}

\usepackage[margin=1in,footskip=25pt,headheight=21pt]{geometry}
\usepackage{fancyhdr,mathrsfs}
\newtheorem{theorem}{Theorem}[section]
\newtheorem{lemma}[theorem]{Lemma}
\newtheorem{proposition}[theorem]{Proposition}

{
\theoremstyle{plain}
\newtheorem{definition}[theorem]{Definition}

\newtheorem{remark}[theorem]{Remark}
}

\newcommand{\nc}{\newcommand}
\newcounter{subeqn}
\renewcommand{\thesubeqn}{\theequation\alph{subeqn}}
\newcommand{\subeqn}{%
  \refstepcounter{subeqn}
  \tag{\thesubeqn}
}
\makeatletter
\@addtoreset{subeqn}{equation}
\newcommand{\newseq}{%
  \refstepcounter{equation}
}
\nc{\cat}{\mathcal{V}}
\nc{\func}{\EuScript{T}}
\nc{\res}{\operatorname{res}}
\nc{\wellsep}{\mathscr{V}}
\newcommand{\Spec}{\operatorname{Spec}}

\renewcommand{\Re}{\operatorname{Re}}
\nc{\iwedge}[1]{\bigwedge\nolimits^{\! #1}}

\nc{\cod}{\mathbbm{d}}

\renewcommand{\dim}{\operatorname{dim}}

\newcommand{\Sym}{\operatorname{Sym}}

\newcommand{\Bi}{\mathbf{i}}

\nc{\ck}{g}
\newcommand{\Bk}{\mathbf{k}}
\newcommand{\Bz}{\mathbf{z}}

\newcommand{\pq}{\mathsf{q}}

\nc{\slehat}{\mathfrak{\widehat{sl}}_e}
\nc{\sllhat}{\mathfrak{\widehat{sl}}_\ell}
\nc{\glehat}{\mathfrak{\widehat{gl}}_e}
\nc{\slnhat}{\mathfrak{\widehat{sl}}_n}
\nc{\glnhat}{\mathfrak{\widehat{gl}}_n}
\nc{\eE}{\EuScript{E}}
\newcommand{\arxiv}[1]{\href{http://arxiv.org/abs/#1}{\tt arXiv:\nolinkurl{#1}}}

\nc{\eF}{\EuScript{F}}
\nc{\fF}{\mathfrak{F}}
\nc{\fE}{\mathfrak{E}}

\newcommand{\bone}{\mathbbm{1}}
\newcommand{\hone}{\text{\textswab{1}}}

\newcommand{\K}{\mathbbm{k}}

\newcommand{\Z}{\mathbb{Z}}
\newcommand{\Q}{\mathbb{Q}}
\nc{\Qlb}{\mathbb{\bar \Q}_\ell}
\nc{\Fq}{\mathbb{F}_q}
\nc{\Fqb}{\mathbb{\bar F}_q}

\nc{\walg}{W}

\newcommand{\R}{\mathbb{R}}

\newcommand{\C}{\mathbb{C}}

\newcommand{\la}{\leftarrow}

\nc{\PQ}{\mathsf{Q}}

\nc{\Bv}{\mathbf{v}}
  \nc{\Bw}{\mathbf{w}}
  \nc{\Bb}{\mathbf{b}}
\nc{\tU}{\mathcal{U}}
\nc{\Bu}{\mathbf{u}}
 \nc{\Fl}{\mathscr{F}\!\ell}
\nc{\Tr}{\operatorname{Tr}}
\nc{\sheafK}{\EuScript{K}}
\nc{\bmu}{\boldsymbol{\mu}}
\nc{\bpi}{\boldsymbol{\pi}}
\nc{\dwalg}{\mathbb{W}}
\nc{\dalg}{\mathbb{T}}
\nc{\aalg}{\mathscr{A}}
\nc{\tilt}{\rotatebox[origin=c]{180}{T}}

\renewcommand{\la}{\lambda}

\newcommand{\al}{\alpha}

\newcommand{\Hom}{\operatorname{Hom}}

\newcommand{\cP}{\mathcal{P}}

\newcommand{\cO}{\mathcal{O}}

\newcommand{\cS}{\gls{cS}}

\newcommand{\waha}{\gls{waha}}

\newcommand{\excise}[1]{}

\newcommand{\End}{\operatorname{End}}

\nc{\mH}{\gls{mH}}
\nc{\hmH}{\gls{hmH}}
\nc{\hR}{\gls{hR}}
\nc{\hcP}{\widehat{\cP}}
\nc{\hP}{\widehat{P}}
\nc{\PC}{\mathfrak{C}}

\newcommand{\mmod}{\operatorname{-mod}}

\newcommand{\bla}{{\underline{\boldsymbol{\la}}}}

\newcommand{\thetitle}{On graded presentations of Hecke algebras\\ and their generalizations}
\newcommand{\theheading}{On graded presentations of Hecke algebras and
  their generalizations}
\numberwithin{equation}{section}

\begin{document}

\newglossaryentry{mH} {
  text={\ensuremath{\mathcal{\widetilde{H}}}},
  name={\ensuremath{\mathcal{\widetilde{H}}(\pq)}},
   description={The affine Hecke algebra of $S_n$ with parameter $\pq=qe^h$.}
 }
 \newglossaryentry{mHfin} {
  name={\ensuremath{\mathcal{H}(\pq)}},
   description={The affine Hecke algebra of $S_n$ with parameter $\pq=qe^h$.}
 }
 \newglossaryentry{cmH} {
  text={\ensuremath{\mathcal{H}}},
  name={\ensuremath{\mathcal{H}(\pq,\PQ_\bullet)}},
   description={The cyclotomic quotient of $\gls{mH}(\pq)$ with Definition \ref{def:cmH}.}
 }
  \newglossaryentry{hmH} {
  text={\ensuremath{\widehat{\widetilde{\mathcal{H}}}}},
  name={\ensuremath{\widehat{\widetilde{\mathcal{H}}}(\pq)}},
   description={The completion of affine Hecke algebra $\gls{mH}(\pq)$
   with respect to the topology induced by $U$.}
 }
  \newglossaryentry{R} {
  text={\ensuremath{{R}}},
  name={\ensuremath{{R}(h)}},
   description={The Khovanov-Lauda-Rouquier algebra for the Dynkin
  diagram $U$ with the $h$-deformed relations.}
 }
   \newglossaryentry{hR} {
  text={\ensuremath{\widehat{R}}},
  name={\ensuremath{\widehat{R}(h)}},
   description={The Khovanov-Lauda-Rouquier algebra completed with
  respect to its grading.}
 }
   \newglossaryentry{cR} {
  text={\ensuremath{{R}}},
  name={\ensuremath{{R}^{Q_\bullet}(h,\Bz)}},
   description={The cyclotomic quotient of \gls{R} with Definition \ref{def:cR}.}
 }
   \newglossaryentry{hcP} {
  text={\ensuremath{\widehat{\cP}}},
  name={\ensuremath{\widehat{\cP}^\pm}},
   description={The (signed) polynomial representation of $\hmH$.}
 }
    \newglossaryentry{hP} {
  text={\ensuremath{\widehat{P}}},
  name={\ensuremath{\widehat{P}}},
   description={The polynomial representation of $\gls{hR}$.}
 }
     \newglossaryentry{waha} {
  tex={\ensuremath{\EuScript{W}}},
  name={\ensuremath{\EuScript{W}(\pq)}},
   description={The weighted affine Hecke algebra defined in
  Definition \ref{def:waha}.}
 }
     \newglossaryentry{hwaha} {
  text={\ensuremath{\widehat{\EuScript{W}}}},
  name={\ensuremath{\widehat{\EuScript{W}}(\pq)}},
  description={The completion of the weighted affine Hecke algebra defined in
  Definition \ref{def:waha}.}
 }
     \newglossaryentry{cS} {
  text={\ensuremath{\mathcal{S}}},
  name={\ensuremath{\mathcal{S}(\pq,n,m)}},
   description={The affine $q$-Schur algebra as defined in Definition \ref{def:cS}.}
 }
      \newglossaryentry{W} {
  name={\ensuremath{{W}}},
   description={The weighted KLR algebra defined in
  Definition \ref{def:wKLR}.}
 }
      \newglossaryentry{hW} {
  name={\ensuremath{\widehat{W}}},
   description={The completed weighted KLR algebra defined in
  Definition \ref{def:hW}.}
 }
 \newglossaryentry{tF} {
  name={\ensuremath{{\tilde{\EuScript{F}}^\vartheta (\pq,\PQ_{\bullet})}}},
   description={The type F affine Hecke algebra defined in
  Definition \ref{def:tF}.}
 }
      \newglossaryentry{htF} {
  name={\ensuremath{\widehat{\tilde{\EuScript{F}}^\vartheta}
  (\pq,\PQ_\bullet)}},
   description={The completed type F affine Hecke algebra.}
 }  \newglossaryentry{cF} {
  name={\ensuremath{\EuScript{F}^\vartheta
  (\pq,\PQ_\bullet)}},
   description={The type F Hecke algebra defined in
  Definition \ref{def:cF}.}
 }
  \newglossaryentry{tT} {
  name={{\ensuremath{\tilde{T}^{\boldsymbol{\la}}(h,\Bz)}}},
   description={The affine Stendhal algebra defined in
  Definition \ref{def:tT}.}
 }
      \newglossaryentry{htT} {
  name={\ensuremath{\widehat{\tilde{T}^{\boldsymbol{{\la}}}}(h,\Bz)}},
   description={The completion of the  affine Stendhal algebra.}
 }
 \newglossaryentry{cT} {
  name={\ensuremath{T^{\boldsymbol{\la}} (h,\Bz)}},
   description={The Stendhal algebra defined in
  Definition \ref{def:tT}}
  }
   \newglossaryentry{tWF} {
  name={\ensuremath{{\EuScript{\widetilde{WF}} (\pq,\PQ_{\bullet})}}},
  text={\ensuremath{{\EuScript{\widetilde{WF}}}}},
   description={The type WF affine Hecke algebra defined in
  Definition \ref{def:tWF}.}
 }
      \newglossaryentry{htWF} {
  name={\ensuremath{\widehat{\tilde{\EuScript{WF}}}}},
   description={The completed type F affine Hecke algebra.}
 }
 \newglossaryentry{cWF} {
  name={\ensuremath{\EuScript{{WF}} (\pq,\PQ_{\bullet})}},
  text={\ensuremath{\EuScript{{WF}}}},
   description={The type WF Hecke algebra defined in
  Definition \ref{def:cWF}.}
 }
   \newglossaryentry{tdalg} {
  name={{\ensuremath{\tilde{\dalg}^\vartheta (h,\Bz)}}},
  text={{\ensuremath{\tilde{\dalg}}}},
   description={The WF KLR algebra (reduced weighted KLR algebra) defined in
  Definition \ref{def:WF-KLR}.}
 }
 \newglossaryentry{cdalg} {
  name={\ensuremath{\dalg ^\vartheta (h,\Bz)}},
  text={\ensuremath{\dalg}},
   description={The steadied quotient of the WF KLR algebra defined in
  Definition \ref{def:WF-KLR}.}
  }
 \newglossaryentry{eprime} {
  name={\ensuremath{e'}},
  text={\ensuremath{e'}},
   description={The idempotent $e'$ in $\waha(\pq)$ which corresponds
   to all $m$-part compositions of $n$.}
  }
  \newglossaryentry{Cset} {
  name={\ensuremath{\mathscr{C}}},
  text={\ensuremath{\mathscr{C}}},
   description={The collection of sets $C_\mu$ attached to $m$-part
   composition of $n$.}
  }
    \newglossaryentry{cycSchur} {
  name={\ensuremath{\mathscr{S}^\pm(\Lambda)}},
  text={\ensuremath{\mathscr{S}(\Lambda)}},
   description={The cyclotomic $q$-Schur algebra $\mathscr{S}(\Lambda)$ of rank $n$  attached to the
data $(\pq,\PQ^\bullet)$ defined by Dipper, James and Mathas
\cite[6.1]{DJM}.}
  }
    \newglossaryentry{U} {
  name={\ensuremath{U}},
  text={\ensuremath{U}},
   description={A fixed finite subset of $\K\setminus\{0\}$, which we
   endow with a quiver structure by connecting $u\to qu$ whenever
   $u,qu\in U$.}
  }

\usetikzlibrary{decorations.pathreplacing,backgrounds,decorations.markings,calc}
\tikzset{wei/.style={draw=red,double=red!40!white,double distance=1.5pt,thin}}

\begin{center}
\noindent {\large  \bf \thetitle}
\medskip

\noindent {\sc Ben Webster}\footnote{Supported by the NSF under Grant
  DMS-1151473 and an NSERC Discovery Grant. This research was supported in part by Perimeter Institute for Theoretical Physics. Research at Perimeter Institute is supported by the Government of Canada through the Department of Innovation, Science and Economic Development Canada and by the Province of Ontario through the Ministry of Research, Innovation and Science.
}\\  
Department of Pure Mathematics\\ University of Waterloo \&\\
Perimeter Institute for Mathematical Physics\\
Waterloo, ON, Canada\\
Email: {\tt ben.webster@uwaterloo.ca}
\end{center}
\bigskip

\begin{quote}
\noindent {\em Abstract.}  In this paper, we define a number of
closely related isomorphisms.  On one side of these isomorphisms sit a
number of of algebras generalizing the Hecke and affine Hecke
algebras, which we call the ``Hecke family''; on the other, we find
generalizations of KLR algebras in finite and affine type A, the ``KLR family.''

We show that these algebras have compatible isomorphisms generalizing
those between Hecke and KLR algebras given by Brundan and Kleshchev.  This allows
us to organize a long list of algebras and categories into a single
system, including (affine/cyclotomic) Hecke algebras, (affine/cyclotomic) $q$-Schur
algebras, (weighted) KLR algebras, category $\cO$ for $\mathfrak{gl}_N$ and for the Cherednik
algebras for the groups $\Z/e\Z\wr S_n$, and give graded presentations
of all of these objects.
\end{quote}

\section{Introduction}
\label{sec:introduction}

Fix a field $\K$ and an element $q\neq 1,0\in \K$.  Let
$e$ be the multiplicative order of $q$.  In this paper, we
discuss isomorphisms between two different families of algebras
constructed from this data.

One of these families is
ultimately descended from Erich Hecke, though it is a rather distant
descent.  It's not clear he would recognize these particular progeny.
The other family is of a more recent vintage.  While the first hint of
its existence was the nilHecke algebra acting on the cohomology of the
complete flag variety, it was not written in full generality until the
past decade in work of Khovanov, Lauda and Rouquier \cite{KLI,KLII,Rou2KM}.    

In the spirit of
other families in representation theory, one can think of Hecke family
as being {\it trigonometric} and the KLR family as {\it rational}. However, a common phenomenon
in mathematics is the existence of an isomorphism between
trigonometric and rational versions of an object after suitable completion; the
``ur-isomorphism'' of this type is between the associated graded of
the K-theory of a manifold and its cohomology.  Such an isomorphism
has been given for completions of non-degenerate and degenerate affine
Hecke algebras by Lusztig in \cite{Lusgraded}.  Another similar
isomorphism is given in \cite{GTL} for Yangians and quantum affine
algebras.  In this paper, we will define isomorphisms with a similar
flavor between the algebras in the Hecke and KLR families.  These
isomorphisms are between certain special completions; before
discussing the specific examples, we cover some generalities on this
type of completion in Section \ref{sec:polyn-style-repr}.

In both cases, these families have somewhat complicated family
trees. Every one depends on a choice of a rank, which we will denote $n$
throughout.  In the diagrammatics, this will always correspond to a
number of strands. On the Hecke side, we will always have a dependence
on a parameter $q$, which we will sometimes want to deform to $qe^h$
with a $h$ a formal parameter.  On the KLR side, we will not see an
explicit family of algebras as we vary $q$, but the underlying
Dynkin diagram used in the definition of these algebras will depend on $h$.  

Like blood types, there are two complementary ways that they can
become complicated.  The simplest case, our analogue of blood type O,
is the affine Hecke algebra (on the Hecke side) and the KLR algebra of
the Dynkin diagrams $\hat{A}_e/A_\infty$ (on the KLR side).  The two complications we can
add are like the type A and type B antigens on our red blood cells.
Since ``type A/B'' already have established connotations in
mathematics, we will instead call these types W and F:
\begin{itemize} 
\item algebras with the type W complication are ``weighted''\footnote{The referee has suggested that
    ``wraith'' in reference to the ghost strands which appear might be
  more appropriate. You might very well think that; the author couldn't possibly comment.} : these
  include affine $q$-Schur algebras \cite{GreenSchur} (on the Hecke side) and weighted KLR
  algebras \cite{WebwKLR}.
\item algebras with the type F complication are 
  ``framed'': these include cyclotomic Hecke algebras \cite{AK} and the
  $\hat{A}_e/A_\infty$ tensor product categorifications from
  \cite{Webmerged}.  These are analogs of the passage from Lusztig to
  Nakajima quiver varieties.
\item finally, both of these complications can be present
  simultaneously, giving type WF.  The natural object which appears in
  the Hecke family is the category $\cO$ of a Cherednik algebra $\Z/e\Z\wr
  S_n$ \cite{GGOR}, though in a guise in which has not been seen previously.  On
  the KLR side, the result is a steadied quotient of a weighted KLR
  algebra for the Crawley-Boevey quiver of a dominant weight of type
  $\hat{A}_e/A_\infty$ (see Definition \ref{def:WF-KLR} and \cite[\S 3.1]{WebwKLR}).
\end{itemize}

Our main theorem is that in each type, there are completions of these
Hecke- and KLR-type algebras that are isomorphic.

Since a great number of different algebras of representation theoretic
interest appear in this picture, it can be quite difficult to keep
them all straight.  For the convenience of the reader, we give a table
in Figure \ref{table-of-algebras}, placing all the algebras and categories which appear in this picture in their
appropriate type.  Note that many of the items listed below (such as
Ariki-Koike algebras, or cyclotomic $q$-Schur and quiver Schur
algebras) are not the most general family members of that type, but
rather special cases. We'll ultimately focus on the category of
representations of a given algebra, so we have not
distinguished between Morita equivalent algebras.   

\begin{figure}\label{table-of-algebras}
  \centerline{\begin{tabular}{c|l|l|} Type& KLR side & Hecke side \\
                \hline O&\parbox[c][1.5em][c]{0.39\textwidth}{KLR
                          algebra \gls{R} \cite{KLI,Rou2KM}
                          }& affine Hecke algebra \gls{mH} of type A (Thm. \ref{type-O-Hecke})\\    \hline
                W& \parbox[c][3em][c]{0.39\textwidth}{weighted KLR
                   algebra \cite{WebwKLR}, \\ quiver Schur algebra 
                \cite{SWschur}} & affine
                                  $q$-Schur algebra $\gls{cS}(n,m)$ (Thm. \ref{waha-Schur})\\    \hline
                F& \parbox[c][4.5em][c]{0.45\textwidth}{cyclotomic KLR
                   algebras \cite{KLI},\\ algebras $\gls{cT}$
                categorifying tensor products for type
                $A/\hat{A}$ \cite{Webmerged}} &
                                                        \parbox[c][2.8em][c]{0.49\textwidth}{
                                                        cyclotomic
                                                        Hecke
                                                        (Ariki-Koike)
                                                        algebras
                                                        (Prop.~\ref{prop:cyclo-Hecke}),
                \\ category $\cO$
                for $\mathfrak{gl}_N$ ($e=\infty$) (\cite[9.11]{Webmerged})}\\    \hline
                WF& \parbox[c][4.5em][c]{0.47\textwidth}{reduced
                    steadied quotients
                    $\gls{cdalg}^\vartheta$ categorifying Uglov Fock spaces \cite{WebRou}, 
                cyclotomic quiver Schur algebras \cite{SWschur}}&
                 \parbox[c][2.8em][c]{0.49\textwidth}{category $\cO$ for a
                     Cherednik  algebra with  $\Z/e\Z\wr S_n$
                                                                  (\cite[Thm. A]{WebRou}),
                                                                  cyclotomic
                                                                  $q$-Schur
                                                                  algebras
                                                                  (Prop. \ref{cqs-morita})}  \\  \hline
              \end{tabular}}
          \caption{The algebras of interest}\end{figure}
On the KLR side, the diagrammatic formulation we give matches the
original definition of these algebras (with the exception of quiver
Schur algebras, which are shown to be Morita equivalent to certain
reduced steadied quotients in \cite[Th. 3.9]{WebwKLR}).  For
the Hecke side, typically our description is a bit different from the
definitions readers will be used to, and we have
listed the result in this paper or another which gives the relation.

\begin{remark}
  All of the algebras on the Hecke side of this list have degenerate
  analogues, and we could have written this paper, like \cite{BKKL}
  with parallel sets of formulas in the degenerate and non-degenerate
  cases.  We avoided doing this because of length, because the correspondence between
  degenerate and non-degenerate formulas is easy to work out (just replace
  multiplication by $q^{\pm 1}$ by addition of $\pm 1$),  and our
  ultimate goal is to apply our results to the Cherednik category
  $\cO$ in \cite{WebRou},
  which only uses the non-degenerate case. 
\end{remark}

\begin{remark}
  Very closely related (and in many cases, Morita equivalent) algebras
  were introduced by Maksimau and Stroppel \cite{MS18}; they use the
  terms ``Hecke family'' and ``KLR family'' exactly as above.  The
  main difference between the approaches in these papers is that this
  paper emphasizes not Schur algebras as those working in the field
  understand them, but certain Morita equivalent algebras we find more
  convenient to work with,  whereas \cite{MS18} work more directly
  with the Schur algebra.   
\end{remark}

\subsubsection*{Type O}  We'll first consider the simplest case of
this isomorphism.  In essence, this is just a rewriting of the
approach in 
\cite[\S 3.2]{Rou2KM}, but for applications in \cite{WebRou}, we require
a small generalization of those results, and it will serve to
illustrate our techniques for the sections on other types.  The two algebras we consider are:
\begin{itemize}
\item the affine Hecke algebra $\gls{mH}(\pq)$ of
$S_n$ with parameter $qe^h$, considered as a $\K[[h]]$-algebra. Note
that this deformation only makes sense if $\K$ has characteristic 0.
\item the KLR algebra $\gls{R}(h)$ of rank $n$ for
$\slehat$ attached to the polynomials $Q_{i+1,i}(u,v)=u-v+h$, also
considered over $\K[[h]]$.
\end{itemize}
These algebras are defined in \cite[(4.1--5)]{BKKL} and
\cite[(1.6--15)]{BKKL} respectively; here we consider them with
the addition of an $h$-adic deformation.  This deformation is very
important since it allows us to compare affine Hecke algebras with $q$
at a root of unity with those for generic $q$.  
For KLR algebras, this
corresponds to comparing KLR algebras for $\widehat{A}_e$ and
$A_\infty$ (as in \cite[Ex. 2.25]{WebwKLR}).
By the usual idempotent lifting arguments (see, for example,
\cite[Lem. 2.2]{Kbook}), the Grothendieck group of
projective modules for $\gls{R}(h)$ is the same as that for the usual KLR
algebra $R$ with $h$ set to $0$; thus, $\gls{R}(h)$-modules categorify the
algebra $U^+(\mathfrak{sl}_\infty)$ or
$U^+(\mathfrak{\widehat{sl}}_e)$ by \cite[Thm. 8]{KLII}.

\begin{theorem} \label{thm:O-main}
  There is a $\K[[h]]$-algebra isomorphism $\hmH(\pq)\cong \hR(h)$.  
\end{theorem}
The characteristic 0 assumption may look peculiar to experts in the
field; the Hecke algebra over a field of characteristic $p$ has
similar deformations coming from deforming the parameter $q$ (though
$e^h$ does not make sense here), but it's not clear how to match other
deformations of the Hecke algebra with the simplest deformations of
the KLR algebra.  
A different deformation of the KLR algebra defined by Hu and Mathas \cite{HMsemi} is
compatible with more general deformations of the Hecke algebra, in particular with the
deformation of $\mathbb{F}_p[S_n]$ to $\mathbb{Z}_p[S_n]$.  
Since our primary applications will be to Hecke algebras and
related structures of characteristic 0, this hypothesis is no problem
for us.  In general, we'll prove our results in parallel with the undeformed Hecke
algebra (and related structures) in arbitrary characteristic, and with
the exponentially deformed Hecke algebra in characteristic 0.

One isomorphism between type O completions was implicitly constructed by Brundan and
Kleshchev in \cite{BKKL} and for a related localization by Rouquier in
\cite[\S 3.2.5]{Rou2KM} for $h=0$. Unfortunately, it is not clear how
to extend these isomorphisms to the deformed case, so instead we
construct an isomorphism which is different even after the
specialization $h=0$.  This isomorphism still has a similar flavor to those previously
defined; in brief, we use a general power series of the form
$1+y+\cdots$ (in particular $e^y$) where Brundan and Kleshchev or
Rouquier use $1+y$.  

We will also generalize this theorem in a small
but useful way:    in fact there is a natural class of completions of the Hecke
algebra that correspond with the KLR algebra for a larger Lie
algebra $\mathfrak{G}_{\gls{U}}$.  Here we consider an arbitrary finite subset $\gls{U}\subset \K\setminus\{0\}$, given a graph
structure connecting $u$ and $u'$ if $qu=u'$, and let $\mathfrak{G}_{\gls{U}}$
be the associated Kac-Moody algebra.  

This definition is the same as
the ``type A graphs'' in \cite[\S 3.2.5]{Rou2KM}, but we do not impose
a connectedness assumption.    
The most important case is when $\gls{U}$ is the $e$th roots of unity, so
$\gls{U}$ is an $e$-cycle, but having a more general statement will be
useful in an analysis of the category $\cO$ for a cyclotomic rational
Cherednik algebra given in \cite{WebRou}.  A more direct proof of this
equivalence using the Dunkl-Opdam subalgebra is now given in \cite{Webalt}.  The generalization of
Theorem \ref{thm:O-main} to this case (Proposition \ref{O-isomorphism}) gives an
alternate approach (and graded version) of the theorem of Dipper and
Mathas \cite{DMmorita} that Ariki-Koike algebras for arbitrary parameters are Morita
equivalent to a tensor product of such algebras with $q$-connected parameters.

The technique we use for this isomorphism and all others considered in
this paper is a variation on that
used by Rouquier in \cite[\S 3.26]{Rou2KM}. We construct an isomorphism between
  completions of the polynomial representations of $\gls{mH}(\pq)$ and $\gls{R}(h)$,
  and then match the operators given by these algebras.  This
  requires considerably less calculation than confirming the relations
  of the algebras themselves.  It also has the considerable advantage of
  easily generalizing to other types.

  In Maksimau and Stroppel's framework \cite{MS18}, these are the cases which are
  ``no level, not Schur.''

\subsubsection*{Type W} 
The first variation we introduce is ``weightedness.''  This is a
similar change of framework in both the Hecke and KLR families, though
it is not easy to see from the usual perspective on the Hecke algebra.  This algebra
can be considered as the span of strand diagrams with number of strands equal to
the rank of the algebra, and a crossing corresponding to $T_i+1$ or
$T_i-q$, depending on conventions.  In this framework, we can
introduce a generalization of the Hecke algebra which allows ``action
at a distance'' where certain interactions between strands occur at a
fixed distance from each other rather than when they cross.  To see
the difference between these, compare the local relations
(\ref{Hecke-1}--\ref{Hecke-triple}) with
(\ref{nilHecke-2}--\ref{eq:triple-point-2}).  We have already
introduced this concept in the KLR family as {\bf weighted KLR
  algebras} \cite{WebwKLR}, but the idea of incorporating it into the
Hecke algebra seems to be new.  Note that wKLR algebras are defined
for any Cartan datum, but as usual, we will only consider those
attached to the quiver structures on sets $\gls{U}$ (which are always unions
of finite and affine type A).

 The main result in this case is
that we obtain a graded KLR type algebra Morita equivalent to the
affine Schur algebra after completion; after this preprint had appearedon the arXiv,
Miemietz and Stroppel \cite{MiSt} showed a direct isomorphism of the completed affine Schur
algebra with a quiver Schur algebra from \cite{SWschur}.
When $e=\infty$, these algebras are Morita equivalent to the type O
algebras, and thus they still categorify the algebra
$U^+(\mathfrak{sl}_\infty)$.  When $e <\infty$, the category of
representations is larger, and corresponds to the passage from
$U^+(\mathfrak{\widehat{sl}}_e)$ to
$U^+(\mathfrak{\widehat{gl}}_e)$.

  Thus, in Maksimau and Stroppel's framework \cite{MS18}, these are the cases which are
  ``no level, Schur'' (though again, we should emphasize that our
  algebra only match theirs up to Morita equivalence in the Schur
  cases).  

\subsubsection*{Type F}
The second variation we'll consider is ``framing.''  This is also a
fundamentally graphical operation, accomplished by including red lines, which
then interplay with those representing our original Hecke algebra.
This case is closely related to the extension from Hecke algebras to
cyclotomic Hecke algebras and parabolic category $\cO$ of type A.

These algebras lead to categorifications of tensor products of simple
representations.  
In the KLR family, these are precisely the tensor
product algebras introduced in \cite[Def. 4.7]{Webmerged}; in the Hecke family,
these algebras do not seem to have appeared in precisely this form
before, though they appear naturally as endomorphisms of modules
over cyclotomic Hecke algebras.

In particular, we show that our isomorphism and deformation are also compatible with
  deformations of cyclotomic quotients.  
For a fixed multiset $\{Q_1,\dots, Q_\ell\}$ of elements of $\gls{U}$, there are cyclotomic quotients of both
  $\gls{mH}(q)$ and $\gls{R}(q)$ (the specializations at $h=0$), which Brundan and Kleshchev construct an
  isomorphism between.  We can deform this cyclotomic quotient with
  respect to variables $\mathbf{z}=\{z_{j}\}$.

For $\gls{mH}(\pq)$,  consider the
deformed cyclotomic quotient attached to the 
polynomial 
\mbox{$C(A)=  \prod_{i=1}^\ell
(A-Q_ie^{-z_i}).$}  
\begin{definition}\label{def:cmH}
  The deformed cyclotomic quotient $\gls{cmH}(\pq,\PQ_\bullet)$ is the quotient
  of the base extension $\gls{mH}(\pq)\otimes_\K\K[[\mathbf{z}]]$ by the
  2-sided ideal generated by $C(X_1)$.
\end{definition}
This is precisely the Ariki-Koike algebra of \cite[Def. 3.1]{AK} for
$G(\ell,1,n)$ with the parameters $u_i=Q_ie^{-z_i}$ (where we use $u_i$ as in the reference
of \cite{AK}).

For $\gls{R}(h)$, the corresponding quotient is given by an additive
deformation of the roots.  For each $u\in \gls{U}$, we have a
polynomial $c_u(a)=\prod_{Q_j=u}(a-z_{j})$.
\begin{definition}\label{def:cR}
  The deformed cyclotomic quotient $\gls{cR}^{Q_\bullet}(h,\Bz)$ is a quotient of
  the base extension $\gls{R}(h)\otimes_\K\K[[\mathbf{z}]]$ by the ideal
  generated by $c_{u_1}(y_1)e_\Bu$ for every length $n$ sequence
  $\Bu\in \gls{U}^n$.
\end{definition}
For the usual indexing of cyclotomic quotients by dominant weights,
this is a deformation of the cyclotomic KLR algebra $\gls{cR}^\la$ attached in \cite{KLI} to a
dominant weight $\la$ of $\mathfrak{G}_{\gls{U}}$ satisfying
\[\al_u^\vee(\la)=\#\{i\in [1,\ell]\mid Q_i=u\}.\]

\begin{theorem}
  The isomorphism $\hmH(\pq)\cong \hR(h)$ induces an isomorphism of
  $\K[[h,\Bz]]$-algebras $\gls{cmH}(\pq,\PQ_\bullet)\cong \gls{cR}^{\la}(h,\Bz)$.
\end{theorem}

In Maksimau and Stroppel's framework \cite{MS18}, these are the cases which are
  ``higher level, not Schur.''

\subsubsection*{Type WF} 

Our final goal, the algebras incorporating both these modifications,
is the least likely to be familiar to readers.  The category of
representations over these algebras is equivalent to the category
$\cO$ for a rational Cherednik algebra for $\Z/\ell\Z \wr S_n $, as
we show in \cite{WebRou}.  In certain cases, these algebras are also
Morita equivalent to cyclotomic $q$-Schur algebras.

The isomorphism between the two families in
this case will prove key in the results of  \cite{WebRou}, proving the
conjecture of Rouquier identifying decomposition numbers in this
category $\cO$ with parabolic Kazhdan-Lusztig polynomials.  This
construction is also of some independent interest as a
categorification of Uglov's higher level Fock space, introduced in
\cite{Uglov}.  In \cite{WebRou}, we will show that several natural,
but hard-to-motivate structures on the Fock space arise from these
algebras.

In Maksimau and Stroppel's framework \cite{MS18}, these include the cases which are
  ``higher level, Schur'' (as before, up to Morita equivalence).    We should however, note that the algebras
  we consider are more general, since they depend on the ratios of
  parameters corresponding to the weightedness and the framing; the
  higher level Schur case only captures situations where this ratio is
  small.  This more general context is used in
  \cite{WebRou,Webalt} to compare with category $\cO$ over Cherednik
  algebras \cite{GGOR}.

\section*{Acknowledgements}
\label{sec:acknowledgements}

Many thanks to Stephen Griffeth for pointing out several small errors
in an earlier version of the paper, to the referees for several
thoughtful suggestions and to Peng Shan, Michaela
Varagnolo, \'Eric Vasserot, Liron Speyer, Christopher Bowman-Scargill,
Ruslan Maksimau and Catharina Stroppel for useful discussions.

\section{Polynomial-style representations}
\label{sec:polyn-style-repr}
First, we will discuss some generalities about completions of algebras
and their representations.  There are a few facts about these
completions we will want to use many times, so it is more convenient
to have a general framework from which they follow.  

 Let $A$ be a $\mathbb{K}$-algebra for $\mathbb{K}$ a commutative
  ring.  Let $B$ a Noetherian commutative $\mathbb{K}$-algebra such
  that $\Spec B$ is a smooth curve over $\Spec \mathbb{K}$.  We'll
  primarily be interested in the case where $B=\mathbb{K}[X,X^{-1}]$
  or $B=\mathbb{K}[y]$, that is the affine line or punctured affine
  line.  The $n$-fold tensor power $B^{\otimes
    n}=B\otimes_{\mathbb{K}} B\otimes_{\mathbb{K}} \cdots
  \otimes_{\mathbb{K}} B$ is thus the functions on the $n$-fold fiber
  product of $\Spec B$ with its usual induced action of $S_n$, and the
  algebra $Z=(B^{\otimes n})^{S_n}$ has smooth spectrum $\Spec
  Z=\Sym^n_{\Spec\K}(\Spec B)$.  As usual, $B^{\otimes n}$ is projective of
  rank $n!$ over $Z$, and free if $\Spec B$ is the punctured or
  unpunctured affine line.
\begin{definition}
 Consider a $\mathbb{K}$-algebra homomorphism $\psi\colon B^{\otimes
    n}\to A$ and an  $A$-module $P$.  We say that the data $(A,B,\psi,P)$ is a
  {\bf polynomial-style representation} of rank $p$ if  
  \begin{enumerate}
  \item $A$ is finite rank and free over $B^{\otimes n}$.
  \item $Z=(B^{\otimes n})^{S_n}$ is central in $A$.
 \item $P$ is faithful  and free
  over $B^{\otimes n} $ of some rank $p$.
  \end{enumerate}
  We call this a {\bf a graded  polynomial-style representation} if in
  addition $A, B$ are graded $\mathbb{K}$-algebras (for some grading
  on $\mathbb{K}$), with $B$ graded local with unique
  graded maximal ideal given by $B_{>0}$,  $P$ is a graded module, and $\psi$ a graded
  homomorphism.  
\end{definition}

We'll want to consider representations of such algebras where some fixed
ideal $I\subset B$ acts nilpotently under every inclusion
$\psi(B\otimes \cdots \otimes B\otimes I \otimes B\otimes \cdots \otimes B)$.  We can
express this as a topological condition.

Consider $B$ as a
topological ring with the $I$-adic topology, and the obvious induced
topologies on $B^{\otimes n}$ and $Z$.  Let $\widehat{B^{\otimes n}}$
and $\widehat{Z}$ be the corresponding completions of these algebras.
The former topology is just
the $I^{(n)}$-adic topology for $I^{(n)}$ the sum of all ideals of the
form $B\otimes \cdots \otimes I\otimes \cdots \otimes B$.
\begin{lemma}
  The subspace topology on $Z$ agrees with the 
  $I'=Z\cap I^{(n)}$-adic topology on this ring.  Alternatively, the
  $I^{(n)}$-adic topology on $B^{\otimes n}$ is the coarsest
  topological ring structure
  such that the inclusion of $Z$, with the $I'$-adic topology, is continuous.
\end{lemma}
\begin{proof}
  Obviously $ (I')^m\subset (I^{(n)})^k$, so the $I'$-adic topology is
  finer than the subspace topology.  In order to show the opposite, we need only show that for any fixed $m$, we have
  $(I^{(n)})^k\cap Z\subset (I')^m$ for all $k\gg 0$.  This will
  follow if $(I^{(n)})^k\subset B^{\otimes n}\cdot  I'$ for some $k$
  since $Z\cap(B^{\otimes n}\cdot  (I')^m)=(I')^m$  as a simple
  calculation with projection to invariants (i.e. the Reynolds
  operator) shows.  This will
 will follow if these ideals have the same radical.

  Since $B^{\otimes n}$ is integral over $Z$, every generator of $
  I^{(n)}$ has a minimal polynomial over $Z$, whose coefficients, of
  course, lie in $I'$.  Thus, a power of this generator lies in $I'$,
  which establishes the desired equality of radicals.   
\end{proof}
 Now we wish to endow $A$ with the coarsest topology compatible with
 this topology on $B^{\otimes n}$, or equivalently on $Z$.  This is induced by the bases $J_m=A (I^{(n)})^m
 A$ or $J_m'=A (I')^mA$, which give equivalent topologies by the
 equality $\sqrt{I^{(n)}}=\sqrt{B^{\otimes n}\cdot  I'}\subset
 B^{\otimes n}$.

 If $(A,B,\psi,P)$ is graded, and $I\subset B$ is the unique graded
 maximal ideal, then there is another description of this topology:
 \begin{lemma}\label{lem:grading-equivalent}
    If $(A,B,\psi,P)$ is graded, and $I=B_{>0}\subset B$ is the unique graded
 maximal ideal, then the topology on $A$ is equivalent to usual
 topology induced by the grading, i.e. the span $G_k$ of the elements of
 degree $\geq k$ is a neighborhood of 0, and these form a basis of
 such neighborhoods.
\end{lemma}
\begin{proof}
The algebra $A$ is finitely generated as a
$Z$-module and thus there is some integer
$M\geq 0$ such that the generators of $ A$ as a $Z$-module have
degrees in the interval $[-M,M]$.  Since the unique graded maximal
ideal of $Z$ is $Z_{>0}$, 
this shows that $G_kG_{m}\subset G_{k+m-M}$.  In particular, since $(I')^m\subset G_m$, we have $J_m'\subset G_{m-2M}$
for all $m$.

Since $Z$ is Noetherian,
$Z_{>0}/Z_{>0}^2$ is a finite dimensional graded vector space over the field $Z/Z_{>0}$, and
we can also assume that all the degrees appearing are $\leq M$ (by
increasing $M$ if necessary).  Note
that this means that all elements of degree $>kM$ lie in $Z_{>0}^k$.  

We know that elements of degree $\geq (k+1)M$
elements of $A$ are spanned by the products of generators with
elements of $Z$ of degree $\geq kM$.  As have observed, these elements
of $Z$ must lie in $Z_{>0}^k$.  Thus we have that $G_{(k+1)M}\subset J_k'$.
Thus, these
topologies are equivalent.  
\end{proof}
 \begin{definition}
   Let $\widehat{A}$ be the completion of $A$ with respect to this
   topology, and $\widehat{P}=\widehat{A}\otimes_AP$.  
 \end{definition}

 \begin{lemma}\label{lem:faithful-completion}
   The completion $\widehat{P}$ is a faithful representation of
   $\widehat{A}$, and is free over $\widehat{B^{\otimes n}}$ of the
   same rank as $P$ over $B^{\otimes n}$.  
 \end{lemma}
 \begin{proof}
Note that we have an injective map
$A\to \End_{Z}(P)$.  The projectivity of $P$
over $Z$ implies that $ \End_{Z}(P)\cong \Hom_Z(P,Z)\otimes_Z P$ is also
projective over $Z$.
Thus, the induced map
\[\widehat{A}\to \End_{\widehat{Z}}(\widehat{P})\cong
  \widehat{{Z}}\otimes_{Z}\End_{{Z}}(P).
\]
giving the action of $\widehat{A}$ agrees with the base change by
$\widehat{Z}$ of the original action map.  This remains injective by
the flatness of $A$ over $Z$.  
 \end{proof}

\section{Type O}

\subsection{Hecke algebras}
We will follow the conventions of \cite{BKKL} concerning Hecke
algebras. 
Our basic object is $\gls{mH}(\pq)$, the {\bf affine Hecke algebra}.  
Let us fix our assumptions on base fields and parameters:
\begin{itemize}
\item [$(*)$] Let $\K$ be a field of any characteristic.  Fix a element
$q\in \K\setminus\{0,1\}$; let $e$ the multiplicative order of $q$
(which may be $\infty$).  Let  $d(\pq)=1+d_1h+\cdots$ be a formal power series in
$\K[[z]]$, which satisfies $d(h_1+h_2)=d(h_1)d(h_2)$, and let $\pq=qd(\pq)$.  
\end{itemize}
Differentiating, we see that this is only possible if $d(h)=e^{d_1h}$; in
particular, if $\K$ has positive characteristic, we must have
$d_1=0$, whereas if $K$ has characteristic 0, this makes sense for any
$d_1$.


The algebra $\gls{mH}(\pq)$ is generated by $\{X_1^{\pm 1},\dots,X_n^{\pm 1}
\}\cup \{T_1,\dots, T_{n-1}\}$ with the relations:
\begin{align*}
  X_r^{\pm 1} X_s^{\pm 1}&= X_s^{\pm 1} X_r^{\pm 1}& T_r^2&=(\pq-1)T_r
+\pq\\
T_rX_rT_r&=\pq X_{r+1} & T_rT_{r+1}T_r&=T_{r+1}T_r T_{r+1}\\
T_rX_s&= X_sT_r\qquad  (r\neq s,s+1)&T_rT_s&=T_sT_r\qquad (r  \neq
s\pm 1)
 \end{align*}
The subalgebra generated by the $T_i$'s alone is a copy of the {\bf
  (finite) Hecke algebra} $\gls{mHfin}$, and the subalgebra generated by the
$X_i^{\pm 1}$ is a copy of the Laurent polynomial ring
$\mathcal{C}=\K[[h]][X_1^{\pm 1},\dots,X_n^{\pm 1}]$.
 
In this paper, we'll rely heavily on a diagrammatic visualization of
this algebra.
\begin{definition}
  Let a rank $n$ {\bf type O diagram} be a collection of $n$ curves in
  $\R\times [0,1]$ with each curve mapping diffeomorphically to
  $[0,1]$ via the projection to the $y$-axis.  Each curve is allowed
  to carry any number of squares or the formal inverse of a square.
  We assume that these curves have no triple points or tangencies, no
  squares lie on crossings and consider these up to isotopies that
  preserve these conditions.
\end{definition}
An example of such a rank 5 diagram is given below:
\[
\begin{tikzpicture}[baseline,very thick,green!50!black, xscale=2,yscale=1.3]
  \draw (-.5,-1) to[out=90,in=-90] (-1,0) to[out=90,in=-90](.5,1);
  \draw (.5,-1) to[out=90,in=-90] (1,1);
  \draw  (1,-1) to[out=90,in=-90] 
  node[midway,fill=green!50!black,inner sep=3pt]{} (0,1);
  \draw (-1, -1) to[out=90,in=-90] (-.5,0) to[out=90,in=-90] (-1,1);
  \draw (0,-1) to[out=90,in=-90] (-.5,1);
\end{tikzpicture}\]
As usual, we can compose these by taking $ab$ to be the diagram where
we place $a$ on top of $b$ and attempt to match up the bottom of $a$
and top of $b$.  If the number of strands is the same, the result is
unique up to isotopy, and if it is different, we formally declare the
result to be $0$.

The rank $n$ {\bf type O affine Hecke algebra} is the quotient of the
span of these diagrams over $\K[[h]]$ by the local relations:

\newseq


\begin{equation*}\subeqn\label{Hecke-1}
    \begin{tikzpicture}[scale=.7,baseline]
      \draw[very thick,green!50!black](-3,0) +(-1,-1) -- +(1,1); \draw[very thick,green!50!black](-3,0) +(1,-1) --
      node[pos=.8,fill=green!50!black,inner sep=3pt]{} +(-1,1) ;
      \node at (-1.5,0){$-$}; \draw[very thick,green!50!black](0,0) +(-1,-1) -- +(1,1); \draw[very thick,green!50!black](0,0) +(1,-1) --  node[pos=.2,fill=green!50!black,inner sep=3pt]{}
      +(-1,1); 
    \end{tikzpicture}\hspace{4mm}=\hspace{4mm}
    \begin{tikzpicture}[scale=.7,baseline]
      \draw[very thick,green!50!black](-3,0) +(-1,-1) --  node[pos=.2,fill=green!50!black,inner sep=3pt]{}+(1,1); \draw[very thick,green!50!black](-3,0) +(1,-1) -- +(-1,1); 
      \node at (-1.5,0){$-$}; \draw[very thick,green!50!black](0,0) +(-1,-1) --
      node[pos=.8,fill=green!50!black,inner sep=3pt]{} +(1,1); \draw[very thick,green!50!black](0,0) +(1,-1) -- +(-1,1)
     ; \node at (2,0){$=$}; \draw[very
      thick,green!50!black](4,0) +(-1,-1) -- node[midway,fill=green!50!black,inner
      sep=3pt]{}+(-1,1); \draw[very
      thick,green!50!black](4,0) +(0,-1) -- +(0,1); \node
      at (5,0){$-\,\,\pq$}; \draw[very thick,green!50!black](7,0) +(-1,-1) -- +(-1,1)
    ; \draw[very thick,green!50!black](7,0) +(0,-1) --
      node[midway,fill=green!50!black,inner sep=3pt]{} +(0,1);
    \end{tikzpicture}
  \end{equation*}
  \begin{equation*}\subeqn\label{Hecke-2}
    \begin{tikzpicture}[very thick,scale=.9,baseline,green!50!black]
      \draw(-2.8,0) +(0,-1) .. controls (-1.2,0) ..  +(0,1)
      ; \draw (-1.2,0) +(0,-1) .. controls
      (-2.8,0) ..  +(0,1) ; 
    \end{tikzpicture}\hspace{4mm}
= (1+\pq)\hspace{5mm}
    \begin{tikzpicture}[very thick,scale=.9,baseline,green!50!black]
      \draw (-2.8,-1)--(-1.2,1);  \draw (-2.8,1)--(-1.2,-1); 
    \end{tikzpicture}
  \end{equation*}
 \begin{equation*}\subeqn\label{Hecke-triple}
    \begin{tikzpicture}[very thick,scale=.9,baseline,green!50!black]
      \draw (-3,0) +(1,-1) -- +(-1,1); \draw
      (-3,0) +(-1,-1) -- +(1,1) ; \draw
      (-3,0) +(0,-1) .. controls (-4,0) ..  +(0,1); 
    \end{tikzpicture}\hspace{4mm}-\hspace{4mm}
\begin{tikzpicture}[very thick,scale=.9,baseline,green!50!black]
\draw (1,0) +(1,-1) -- +(-1,1)
     ; \draw (1,0) +(-1,-1) -- +(1,1)
      ; \draw (1,0) +(0,-1) .. controls
      (2,0) ..  +(0,1); 
    \end{tikzpicture}\hspace{4mm}
=   \pq \hspace{4mm} \begin{tikzpicture}[very thick,scale=.9,baseline,green!50!black]
      \draw (-3,0) +(1,-1) -- +(1,1); \draw
      (-3,0) +(-1,-1) -- +(0,1) ; \draw
      (-3,0) +(0,-1) --  +(-1,1);
    \end{tikzpicture}\hspace{4mm}-\pq  \hspace{4mm} \begin{tikzpicture}[very thick,scale=.9,baseline,green!50!black]
      \draw (-3,0) +(-1,-1) -- +(-1,1); \draw
      (-3,0) +(1,-1) -- +(0,1) ; \draw
      (-3,0) +(0,-1) --  +(1,1);
    \end{tikzpicture}
  \end{equation*}
  \begin{remark}
    We want to make sure that the reader notices the distinction here
    between ``relations'' and ``local relations.''  Here ``relations''
    has the usual algebraic meaning: generators of the kernel of the
    homomorphism to an algebra of the free associative algebra on the
    generators.  However, ``local relations'' means something a bit
    more subtle: whenever we have two diagrams which are identical
    outside a small region, and match the two sides of the equation,
    then they are set equal.  Of course, the effect this has depends
    on what is allowed in the rest of the diagram.

    So a relation like
    \ref{Hecke-1} can be applied in type O diagrams (as in this
    section) or in type W diagrams, which we'll introduce later.
    While the pictures on the page are the same, the induced relations
    are different, since we have different rules for how the rest of
    the diagram is constructed.  Thus in later sections, we will
    refer back to these local relations, but apply them in a different
    diagrammatic framework.
  \end{remark}
  
  \begin{remark}
    In the degenerate case, we can write a similar set of local relations,
    replacing (\ref{Hecke-1}--\ref{Hecke-triple}) with the local relations:
\begin{equation*}\subeqn\label{deg-Hecke-1}
    \begin{tikzpicture}[scale=.6,baseline]
      \draw[very thick,green!50!black](-3,0) +(-1,-1) -- +(1,1); \draw[very thick,green!50!black](-3,0) +(1,-1) --
      node[pos=.8,fill=green!50!black,inner sep=3pt]{} +(-1,1) ;
      \node at (-1.5,0){$-$}; \draw[very thick,green!50!black](0,0) +(-1,-1) -- +(1,1); \draw[very thick,green!50!black](0,0) +(1,-1) --  node[pos=.2,fill=green!50!black,inner sep=3pt]{}
      +(-1,1); 
    \end{tikzpicture}\hspace{3mm}=\hspace{3mm}
    \begin{tikzpicture}[scale=.6,baseline]
      \draw[very thick,green!50!black](-3,0) +(-1,-1) --  node[pos=.2,fill=green!50!black,inner sep=3pt]{}+(1,1); \draw[very thick,green!50!black](-3,0) +(1,-1) -- +(-1,1); 
      \node at (-1.5,0){$-$}; \draw[very thick,green!50!black](0,0) +(-1,-1) --
      node[pos=.8,fill=green!50!black,inner sep=3pt]{} +(1,1); \draw[very thick,green!50!black](0,0) +(1,-1) -- +(-1,1)
     ; \node at (2,0){$=$}; \draw[very
      thick,green!50!black](4,0) +(-1,-1) -- node[midway,fill=green!50!black,inner
      sep=3pt]{}+(-1,1); \draw[very
      thick,green!50!black](4,0) +(0,-1) -- +(0,1); \node
      at (5,0){$-$}; \draw[very thick,green!50!black](7,0) +(-1,-1) -- +(-1,1)
    ; \draw[very thick,green!50!black](7,0) +(0,-1) --
      node[midway,fill=green!50!black,inner sep=3pt]{} +(0,1); \node
      at (8,0){$-$};
\draw[very thick,green!50!black](10,0) +(-1,-1) -- +(-1,1)
    ; \draw[very thick,green!50!black](10,0) +(0,-1) --+(0,1);
    \end{tikzpicture}
  \end{equation*}
  \begin{equation*}\subeqn\label{deg-Hecke-2}
    \begin{tikzpicture}[very thick,scale=.9,baseline,green!50!black]
      \draw(-2.8,0) +(0,-1) .. controls (-1.2,0) ..  +(0,1)
      ; \draw (-1.2,0) +(0,-1) .. controls
      (-2.8,0) ..  +(0,1) ; 
    \end{tikzpicture}\hspace{4mm}
= \,\,2\hspace{5mm}
    \begin{tikzpicture}[very thick,scale=.9,baseline,green!50!black]
      \draw (-2.8,-1)--(-1.2,1);  \draw (-2.8,1)--(-1.2,-1); 
    \end{tikzpicture}
  \end{equation*}
 \begin{equation*}\subeqn\label{deg-Hecke-triple}
    \begin{tikzpicture}[very thick,scale=.9,baseline,green!50!black]
      \draw (-3,0) +(1,-1) -- +(-1,1); \draw
      (-3,0) +(-1,-1) -- +(1,1) ; \draw
      (-3,0) +(0,-1) .. controls (-4,0) ..  +(0,1); 
    \end{tikzpicture}\hspace{4mm}-\hspace{4mm}
\begin{tikzpicture}[very thick,scale=.9,baseline,green!50!black]
\draw (1,0) +(1,-1) -- +(-1,1)
     ; \draw (1,0) +(-1,-1) -- +(1,1)
      ; \draw (1,0) +(0,-1) .. controls
      (2,0) ..  +(0,1); 
    \end{tikzpicture}\hspace{4mm}
=   \hspace{4mm} \begin{tikzpicture}[very thick,scale=.9,baseline,green!50!black]
      \draw (-3,0) +(1,-1) -- +(1,1); \draw
      (-3,0) +(-1,-1) -- +(0,1) ; \draw
      (-3,0) +(0,-1) --  +(-1,1);
    \end{tikzpicture}\hspace{4mm}-  \hspace{4mm} \begin{tikzpicture}[very thick,scale=.9,baseline,green!50!black]
      \draw (-3,0) +(-1,-1) -- +(-1,1); \draw
      (-3,0) +(1,-1) -- +(0,1) ; \draw
      (-3,0) +(0,-1) --  +(1,1);
    \end{tikzpicture}
  \end{equation*}
  \end{remark}

It may not be immediately clear what the additional value of this graphical
presentations is.  However, this perspective
will lead us to generalizations of the affine Hecke algebra which we
call types W, F and WF.
  \begin{theorem}\label{type-O-Hecke}
The algebra $\gls{mH}(\pq)$ is isomorphic to the rank $n$ type O Hecke
algebra via the map sending $T_r+1$ to the crossing of the $r$th and
$r+1$st strands, and $X_r$ to the square on the $r$th strand, as shown below:
 \begin{equation}\label{Hecke-gens}
    \tikz[baseline]{
      \node[label=below:{$X_j$}] at (0,0){ 
        \tikz[very thick,xscale=1.2,green!50!black]{
          \draw (-.5,-.5)-- (-.5,.5);
          \draw (.5,-.5)-- (.5,.5) node [midway,fill=green!50!black,inner
          sep=2.5pt]{};
          \draw (1.5,-.5)-- (1.5,.5);
          \node at (1,0){$\cdots$};
          \node at (0,0){$\cdots$};
        }
      };
      \node[label=below:{$T_j+1$}] at (4.5,0){ 
        \tikz[very thick,xscale=1.2, green!50!black]{
          \draw (-.5,-.5)-- (-.5,.5);
          \draw (.1,-.5)-- (.9,.5);
          \draw (.9,-.5)-- (.1,.5);
          \draw (1.5,-.5)-- (1.5,.5);
          \node at (1,0){$\cdots$};
          \node at (0,0){$\cdots$};
        }
      };
    }\vspace{-2mm}
  \end{equation}
  \end{theorem}
  \begin{proof}
\newseq
    We'll use the relations given in  \cite[\S 4]{BKKL} without
    additional citation.  The
    equations (\ref{Hecke-1}--\ref{Hecke-triple}) become  the relations: 
    \begin{align*}
      X_r(T_r+1)-(T_r+1)X_{r+1}
      &=T_rX_{r+1}+(1-\pq)X_{r+1}+X_r-T_rX_{r+1}-X_{r+1} \\ 
      &=X_r-\pq X_{r+1}\\
      X_{r+1}(T_r+1)-(T_r+1)X_{r}
      &=T_rX_{r}+(\pq-1)X_{r+1}+X_{r+1}-T_rX_{r}-X_{r}\\
      &=qX_{r+1}-X_r\\
   (T_r+1)^2&=T_r^2+2T_r+1\\
&=(\pq-1)T_r+\pq+2T_r+1\\
&=(1+\pq)(T_r+1)
  \end{align*}
   \begin{align*}
     (T_r+1)(T_{r+1}+1)(T_r+1)-(T_{r+1}+1)(T_{r}+1)(T_{r+1}+1)&=
     T_r^2+T_r-T_{r+1}^2 -T_{r+1}\\
     &=\pq(T_r-T_{r+1})
   \end{align*}
Similarly, one can easily derive the relations of the affine Hecke
from the diagrammatic ones given above.  This shows that we have an
isomorphism.
  \end{proof}
Note that if we instead sent the element $T_i-\pq$ to the crossing, we
would obtain  local relations which are quite similar to
(\ref{Hecke-1}--\ref{Hecke-triple}), but have a few subtle
differences:
\begin{equation*}\subeqn\label{qHecke-1}
    \begin{tikzpicture}[scale=.7,baseline,green!50!black]
      \draw[very thick](-3,0) +(-1,-1) -- +(1,1); \draw[very thick](-3,0) +(1,-1) --
      node[pos=.8,fill=green!50!black,inner sep=3pt]{} +(-1,1) ;
      \node at (-1.5,0){$-$}; \draw[very thick](0,0) +(-1,-1) -- +(1,1); \draw[very thick](0,0) +(1,-1) --  node[pos=.2,fill=green!50!black,inner sep=3pt]{}
      +(-1,1); 
    \end{tikzpicture}\hspace{4mm}=\hspace{4mm}
    \begin{tikzpicture}[scale=.7,baseline,green!50!black]
      \draw[very thick](-3,0) +(-1,-1) --  node[pos=.2,fill=green!50!black,inner sep=3pt]{}+(1,1); \draw[very thick](-3,0) +(1,-1) -- +(-1,1); 
      \node at (-1.5,0){$-$}; \draw[very thick](0,0) +(-1,-1) --
      node[pos=.8,fill=green!50!black,inner sep=3pt]{} +(1,1); \draw[very thick](0,0) +(1,-1) -- +(-1,1)
     ; \node at (2,0){$=$}; \draw[very
      thick](4,0) +(-1,-1) -- +(-1,1); \draw[very
      thick](4,0) +(0,-1) -- node[midway,fill=green!50!black,inner
      sep=3pt]{}+(0,1); \node[black]
      at (5,0){$-\,\,\pq$}; \draw[very thick](7,0) +(-1,-1) --  node[midway,fill=green!50!black,inner sep=3pt]{}+(-1,1)
    ; \draw[very thick](7,0) +(0,-1) --
     +(0,1);
    \end{tikzpicture}
  \end{equation*}
  \begin{equation*}\subeqn\label{qHecke-2}
    \begin{tikzpicture}[very thick,scale=.9,baseline,green!50!black]
      \draw(-2.8,0) +(0,-1) .. controls (-1.2,0) ..  +(0,1)
      ; \draw (-1.2,0) +(0,-1) .. controls
      (-2.8,0) ..  +(0,1) ; 
    \end{tikzpicture}\hspace{4mm}
= -(1+\pq)\hspace{5mm}
    \begin{tikzpicture}[very thick,scale=.9,baseline,green!50!black]
      \draw (-2.8,-1)--(-1.2,1);  \draw (-2.8,1)--(-1.2,-1); 
    \end{tikzpicture}
  \end{equation*}
 \begin{equation*}\subeqn\label{qHecke-triple}
    \begin{tikzpicture}[very thick,scale=.9,baseline,green!50!black]
      \draw (-3,0) +(1,-1) -- +(-1,1); \draw
      (-3,0) +(-1,-1) -- +(1,1) ; \draw
      (-3,0) +(0,-1) .. controls (-4,0) ..  +(0,1); 
    \end{tikzpicture}\hspace{4mm}-\hspace{4mm}
\begin{tikzpicture}[very thick,scale=.9,baseline,green!50!black]
\draw (1,0) +(1,-1) -- +(-1,1)
     ; \draw (1,0) +(-1,-1) -- +(1,1)
      ; \draw (1,0) +(0,-1) .. controls
      (2,0) ..  +(0,1); 
    \end{tikzpicture}\hspace{4mm}
=   \pq \hspace{4mm} \begin{tikzpicture}[very thick,scale=.9,baseline,green!50!black]
      \draw (-3,0) +(1,-1) -- +(1,1); \draw
      (-3,0) +(-1,-1) -- +(0,1) ; \draw
      (-3,0) +(0,-1) --  +(-1,1);
    \end{tikzpicture}\hspace{4mm}-\pq  \hspace{4mm} \begin{tikzpicture}[very thick,scale=.9,baseline,green!50!black]
      \draw (-3,0) +(-1,-1) -- +(-1,1); \draw
      (-3,0) +(1,-1) -- +(0,1) ; \draw
      (-3,0) +(0,-1) --  +(1,1);
    \end{tikzpicture}
  \end{equation*}
Our first task is to describe the completions that are of interest to
us.

Consider a finite subset
$\gls{U}\subset \K\setminus\{0\}$; as before, we endow this with a graph structure by adding an edge
from $u$ to $u'$ if $u'=qu$.  Note that for $\gls{U}$
chosen generically there will simply be no edges, and that under this
graph structure $\gls{U}$ will always be a union of segments and cycles with $e$
nodes (if $e<\infty$).

We will apply the results of Section \ref{sec:polyn-style-repr} in
this context with
\begin{equation}
\mathbb{K}=\K[[h]]\qquad A=\mH(\pq)\qquad
  B=\K[[h]][X^{\pm}]\qquad I=Bh+B \prod_{u\in \gls{U}}(X-u).\label{eq:Hecke-PR}
\end{equation}

One natural construction of modules over $\gls{mH}(\pq)$ is given by
induction from $\gls{mHfin} $, as discussed in \cite[\S
4.3]{MacHecke}; as discussed there, the result is free as a
$\mathcal{C}$-module if the original module is free over $\K[[h]]$,
with ranks matching.  In particular, applying this to the two 1-dimensional
representations of $\gls{mHfin}$, where this algebra acts by the
characters $\chi^\pm$ where
\[ \chi^+(T_i)=\pq\qquad \chi^-(T_i)=-1\] gives natural (signed) polynomial
representation \[\cP^\pm=\gls{mH}(\pq)\otimes_{\gls{mHfin}}\K[[h]];\] where
$\gls{mHfin}$ acts on $\K[[h]]$ via the homomorphism $\chi^\pm$.
\begin{lemma}
  The data of (\ref{eq:Hecke-PR}) defines a polynomial-style
  representation on $P=\cP^{\pm}$.  
\end{lemma}
\begin{proof}\mbox{}
  \begin{enumerate}
  \item  This freeness is clear from the basis in
    \cite[4.3]{MacBourbaki}; in fact, as
    $B^{\otimes n}=\mathcal{C}$-module, we have
    $\gls{mH}(\pq)\cong \mathcal{C}\otimes_{\K[[h]]}\gls{mHfin} $. 
    \item The center of the affine Hecke algebra $\gls{mH}(\pq)$ is precisely
    the symmetric Laurent polynomials $Z=\mathcal{C}^{S_n}$ by
    \cite[4.5]{MacBourbaki}.
    \item The module $\cP^{\pm}$ is faithful by
      \cite[(4.3.10)]{MacHecke}, and its freeness over $\mathcal{C}$
      has already been discussed.\qedhere
  \end{enumerate}
\end{proof}

Thus, as in Section \ref{sec:polyn-style-repr}, we have an induced
topology on $\mH(\pq)$ with completion $\hmH(\pq)$.  We could also
define $\hmH(\pq)$ as the completion of
$\gls{mH}(\pq)$ in the directed system of all quotients where the spectrum of
each $X_i$ lies in $\gls{U}$.  

We can
identify $\Spec(\mathcal{C})$ with $(\mathbb{A}^1\setminus
\{0\})^n\times  \widehat{\mathbb{A}}^1$ where the last factor has coordinate
$h$ and is completed at $0$. 
 Let $\mathcal{U}=\gls{U}^n\times \{0\}\subset \Spec(\mathcal{C})$.   This
 is the vanishing set of $I^{(n)}$ as defined in Section
 \ref{sec:polyn-style-repr}.  Thus, the closure of $\mathcal{C}$ in
 $\hmH(\pq)$ is the completion of $\mathcal{C}$ at this subscheme.  In
 particular, the identity in $\mathcal{C}$, and thus in $\hmH(\pq)$
 decomposes as a sum of idempotents $1=\sum_{\Bu\in \gls{U}^n}e_{\Bu}$.  These have the
 property that on any topological $\hmH(\pq)$-module $M$, we have that  
\[e_{\Bu}M=\{ g\in M \mid\lim_{N\to \infty}
  (X_j-{u_j})^Ng=0\},\]
and for any module, we have $M=\oplus e_{\Bu}M$.  


In particular,
we have that $\hmH(\pq)=\bigoplus_{\Bu\in \gls{U}^n}e_{\Bu}\hmH(\pq)=\bigoplus_{\Bu,\Bu'\in \gls{U}^n}e_{\Bu}\hmH(\pq) e_{\Bu'}$.  

\subsubsection{Formulas for the polynomial representation}
\label{sec:form-polyn-repr}

Now, let us study the action of $\mH(\pq)$ on its polynomial
representation $\cP^{\pm}$.  Denote the action of $S_n$ on $\gls{U}^n$ by $\Bu\mapsto \Bu^s$ for $s\in
S_n$; as usual, we let $s_i=(i, i+1)$. For any Laurent polynomial $F$,
we let  $F^{s_r}(X_1,\dots,X_n)=F(X_1,\dots,X_{r+1},X_r,\dots, X_n)$.  

For notational clarity, we denote $\hone=1\otimes 1\in \cP^-$, so this
representation is generated by this vector, subject to the relation
$T_i \hone=-\hone$.  As in \cite[(4.3.3)]{MacHecke}, one can calculate the action of $T_i$ on
$F \hone$ 
for any Laurent polynomial $F$; this is easiest to see if we expand
$F=F_0+(X_i-X_{i+1})F_1$ where $F_0$ and $F_1$ are $s_i$-invariant
Laurent polynomials.  Thus, we have that 
\begin{align*}
  T_i F\hone&= T_i F_0\hone+T_i(X_i-X_{i+1})F_1\hone\\
  &=-F_0\hone+(X_i-X_{i+1})F_1\hone+2(1-\pq)
X_{i+1}F_1\hone\\
&=  -F^{s_i}\hone+ (1-\pq)
X_{i+1}\frac{F^{s_i}-F}{X_{i+1}-X_i}\hone.
\end{align*}
Thus, we have that 
\[(T_i+1) F\hone =\Big(F-F^{s_i}+ (1-\pq)
X_{i+1}\frac{F^{s_i}-F}{X_{i+1}-X_i}\Big)\hone+\frac{X_i-\pq X_{i+1}}{X_{i+1}-X_i} (F^{s_i}-F)\hone.\]

The Hecke algebra acts faithfully on this
representation by \cite[(4.3.10)]{MacHecke}, so we can identify the affine Hecke algebra with a
subalgebra of operators on $\cP^\pm$.

Similarly, the representation $\cP^+$ is
generated by an element $\hone^+$ satisfying $T_i\hone^+=\pq \hone^+$.
The action of $\gls{mH}(\pq)$ in this case is given by the formula
\[T_iF\hone^+=\pq F^{s_i}\hone^++(1-\pq)X_{i+1}\frac{F^{s_i}-F}{X_{i+1}-X_i}\hone^+\]
so we have that 
\[(T_i-\pq)\hone^+=\frac{X_{i+1}-\pq X_i}{X_{i+1}-X_i}( F^{s_i}-F).\]

Consider the $\hmH(\pq)$-module $\hcP^\pm :=\hmH(\pq)
\otimes_{\gls{mH}(\pq)}\cP^\pm$.  It follows from Lemma
\ref{lem:faithful-completion} that:
\begin{lemma}\label{lem:Hecke-P}
   The module $\hcP^\pm$ is a rank 1 free module over the completion
   of $\mathcal{C} $  at the set $\mathcal{U}$, and this
   representation remains faithful.  The space $e_{\Bu}\hcP^\pm $ is isomorphic to
$\K[[(X_1-{u_1}),\dots, (X_n-{u_n}),h]]$ via the action map on
$e_{\Bu}\hone$.
 \end{lemma}

\subsection{KLR algebras}

We wish to define a similar completion of the KLR algebra $\gls{R}(h)$ for the
graph $\gls{U}$.  We use the conventions of Brundan and Kleshchev, but we
record the relations we need here for the sake of completeness and to
match our slightly more general context.  The rank\footnote{Note here
  that the ``rank'' $n$ has no relationship to the size of the set $\gls{U}$
(typically called the rank of the corresponding Kac-Moody algebra);
for any fixed $\gls{U}$, we get a different algebra for each positive
integer $n$.} $n$ KLR algebra $\gls{R}(h)$ attached to the Dynkin
diagram $\gls{U}$ is
generated over $\K[h]$ 
by elements $\{e(\Bu)\}_{\Bu\in \gls{U}^n}\cup \{y_1,\dots, y_n\}\cup \{\psi_1,\dots, \psi_{n-1}\}$
subject to the relations:
\begin{align*}
e(\Bu) e(\Bv) &= \delta_{\Bu,\Bv} e(\Bu);
\hspace{53mm}{\sum_{\Bu \in I^\alpha}} e(\Bu) = 1;\\
y_r e(\Bu) &= e(\Bu) y_r;
\hspace{53mm}\psi_r e(\Bu) = e(\Bu^{s_r}) \psi_r;\\
y_r y_s &= y_s y_r;\\
\psi_r y_s  &= y_s \psi_r\hspace{61.4mm}\text{if $s \neq r,r+1$};\\
\psi_r \psi_s &= \psi_s \psi_r\hspace{60.8mm}\text{if $s\neq r\pm 1$};\\
\psi_r y_{r+1} e(\Bu) 
&= 
\begin{cases}
  (y_r\psi_r+1)e(\Bu) &\hbox{if $u_r=u_{r+1}$},\\
  y_r\psi_r e(\Bu) \hspace{48mm}&\hbox{if $u_r\neq u_{r+1}$};
\end{cases}\\
y_{r+1} \psi_re(\Bu) &=
\begin{cases}
  (\psi_r y_r+1) e(\Bu)
  &\hbox{if $u_r=u_{r+1}$},\\
  \psi_r y_r e(\Bu) \hspace{48mm}&\hbox{if $u_r\neq u_{r+1}$};\\
\end{cases}\\
\psi_r^2e(\Bu) &=
\begin{cases}
0 \hspace{61mm}&\text{if $u_r = u_{r+1}$},\\
e(\Bu)&\text{if $u_r \neq q^{\pm 1}u_{r+1},u_{r+1}$},\\
(y_{r+1}-y_r+d_1h)e(\Bu)&\text{if $u_r = q^{-1}u_{r+1}, q\neq -1$},\\
(y_r - y_{r+1}+d_1h)e(\Bu)&\text{if $u_r = qu_{r+1}, q\neq -1$},\\
(y_r - y_{r+1}+d_1h)(y_{r+1} - y_{r}+d_1h) e(\Bu)&\text{if $u_r=-u_{r+1}, q=-1$};
\end{cases}
\end{align*}\begin{align*}
\psi_{r}\psi_{r+1} \psi_{r} e(\Bu)
&=
\begin{cases}
  (\psi_{r+1} \psi_{r} \psi_{r+1} +1)e(\Bu)&\text{if $u_r = u_{r+2}=q^{-1}u_{r+1}, q\neq -1$},\\
  (\psi_{r+1} \psi_{r} \psi_{r+1} -1)e(\Bu)&\text{if $u_r = u_{r+2}=qu_{r+1}, q\neq -1$},\\
  \big(\psi_{r+1} \psi_{r} \psi_{r+1} -2y_{r+1}  +y_r+y_{r+2}\big)e(\Bu)
&\text{if $u_r = u_{r+2}=-u_{r+1}, q= -1$},\\
  \psi_{r+1} \psi_{r} \psi_{r+1} e(\Bu)&\text{otherwise}.
\end{cases}
\end{align*}

\newseq
Just as in the Hecke case, there is a graphical
presentation for the KLR algebra.  Since this is covered in \cite{KLI}
and numerous other sources, we'll just record an example of an
appropriate KLR diagram here for comparison purposes:
\[
\begin{tikzpicture}[baseline,very thick,xscale=2,yscale=1.3]
  \draw (-.5,-1) to[out=90,in=-90]  node[below,at start]{$u_2$} (-1,0) to[out=90,in=-90](.5,1);
  \draw (.5,-1) to[out=90,in=-90] node[below,at start]{$u_4$} (1,1);
  \draw  (1,-1) to[out=90,in=-90] node[below,at start]{$u_5$}
  node[circle, midway,fill=black,inner sep=3pt]{} (0,1);
  \draw (-1, -1) to[out=90,in=-90]node[below,at start]{$u_1$} (-.5,0) to[out=90,in=-90] (-1,1);
  \draw (0,-1) to[out=90,in=-90] node[below,at start]{$u_3$} (-.5,1);
\end{tikzpicture}\] and write out the local relations here for convenience:
\begin{equation*}\subeqn\label{first-QH}
    \begin{tikzpicture}[scale=.8,baseline]
      \draw[very thick](-4,0) +(-1,-1) -- +(1,1) node[below,at start]
      {$u$}; \draw[very thick](-4,0) +(1,-1) -- +(-1,1) node[below,at
      start] {$v$}; \fill (-4.5,.5) circle (3pt);
      \node at (-2,0){=}; \draw[very thick](0,0) +(-1,-1) -- +(1,1)
      node[below,at start] {$u$}; \draw[very thick](0,0) +(1,-1) --
      +(-1,1) node[below,at start] {$v$}; \fill (.5,-.5) circle (3pt);
      \node at (4,0){unless $u=v$};
    \end{tikzpicture}
  \end{equation*}
\begin{equation*}\subeqn
    \begin{tikzpicture}[scale=.8,baseline]
      \draw[very thick](-4,0) +(-1,-1) -- +(1,1) node[below,at start]
      {$u$}; \draw[very thick](-4,0) +(1,-1) -- +(-1,1) node[below,at
      start] {$v$}; \fill (-3.5,.5) circle (3pt);
      \node at (-2,0){=}; \draw[very thick](0,0) +(-1,-1) -- +(1,1)
      node[below,at start] {$u$}; \draw[very thick](0,0) +(1,-1) --
      +(-1,1) node[below,at start] {$v$}; \fill (-.5,-.5) circle (3pt);
      \node at (4,0){unless $u=v$};
    \end{tikzpicture}
  \end{equation*}
\begin{equation*}\subeqn\label{nilHecke-1}
    \begin{tikzpicture}[scale=.8,baseline]
      \draw[very thick](-4,0) +(-1,-1) -- +(1,1) node[below,at start]
      {$u$}; \draw[very thick](-4,0) +(1,-1) -- +(-1,1) node[below,at
      start] {$u$}; \fill (-4.5,.5) circle (3pt);
      \node at (-2,0){$-$}; \draw[very thick](0,0) +(-1,-1) -- +(1,1)
      node[below,at start] {$u$}; \draw[very thick](0,0) +(1,-1) --
      +(-1,1) node[below,at start] {$u$}; \fill (.5,-.5) circle (3pt);
      \node at (2,0){$=$}; 
    \end{tikzpicture}
    \begin{tikzpicture}[scale=.8,baseline]
      \draw[very thick](-4,0) +(-1,-1) -- +(1,1) node[below,at start]
      {$u$}; \draw[very thick](-4,0) +(1,-1) -- +(-1,1) node[below,at
      start] {$u$}; \fill (-4.5,-.5) circle (3pt);
      \node at (-2,0){$-$}; \draw[very thick](0,0) +(-1,-1) -- +(1,1)
      node[below,at start] {$u$}; \draw[very thick](0,0) +(1,-1) --
      +(-1,1) node[below,at start] {$u$}; \fill (.5,.5) circle (3pt);
      \node at (2,0){$=$}; \draw[very thick](4,0) +(-1,-1) -- +(-1,1)
      node[below,at start] {$u$}; \draw[very thick](4,0) +(0,-1) --
      +(0,1) node[below,at start] {$u$};
    \end{tikzpicture}
  \end{equation*}
  \begin{equation*}\subeqn\label{black-bigon}
    \begin{tikzpicture}[very thick,scale=.8,baseline]

      \draw (-2.8,0) +(0,-1) .. controls (-1.2,0) ..  +(0,1)
      node[below,at start]{$u$}; \draw (-1.2,0) +(0,-1) .. controls
      (-2.8,0) ..  +(0,1) node[below,at start]{$v$}; 
   \end{tikzpicture}=\quad
   \begin{cases}
0 & u=v\\
     \begin{tikzpicture}[very thick,scale=.6,baseline=-3pt]
       \draw (2,0) +(0,-1) -- +(0,1) node[below,at start]{$v$};
       \draw (1,0) +(0,-1) -- +(0,1) node[below,at start]{$u$};
     \end{tikzpicture} & u\notin \{v,qv,q^{-1}v\}\\
   \begin{tikzpicture}[very thick,scale=.6,baseline=-3pt]
       \draw (2,0) +(0,-1) -- +(0,1) node[below,at start]{$v$};
       \draw (1,0) +(0,-1) -- +(0,1) node[below,at start]{$u$};\fill (2,0) circle (4pt);
     \end{tikzpicture}-\begin{tikzpicture}[very thick,scale=.6,baseline=-3pt]
       \draw (2,0) +(0,-1) -- +(0,1) node[below,at start]{$v$};
       \draw (1,0) +(0,-1) -- +(0,1) node[below,at start]{$u$};\fill (1,0) circle (4pt);
     \end{tikzpicture}+d_1h \begin{tikzpicture}[very thick,scale=.6,baseline=-3pt]
       \draw (2,0) +(0,-1) -- +(0,1) node[below,at start]{$v$};
       \draw (1,0) +(0,-1) -- +(0,1) node[below,at start]{$u$};
     \end{tikzpicture}& u=q^{-1}v,q\neq -1\\
  \begin{tikzpicture}[very thick,baseline=-3pt,scale=.6]
       \draw (2,0) +(0,-1) -- +(0,1) node[below,at start]{$v$};
       \draw (1,0) +(0,-1) -- +(0,1) node[below,at start]{$u$};\fill (1,0) circle (4pt);
     \end{tikzpicture}-\begin{tikzpicture}[very thick,scale=.6,baseline=-3pt]
       \draw (2,0) +(0,-1) -- +(0,1) node[below,at start]{$v$};
       \draw (1,0) +(0,-1) -- +(0,1) node[below,at start]{$u$};\fill (2,0) circle (4pt);
     \end{tikzpicture}+d_1h \begin{tikzpicture}[very thick,scale=.6,baseline=-3pt]
       \draw (2,0) +(0,-1) -- +(0,1) node[below,at start]{$v$};
       \draw (1,0) +(0,-1) -- +(0,1) node[below,at start]{$u$};
     \end{tikzpicture}& u=qv,q\neq -1\\
 - \Bigg(\begin{tikzpicture}[very thick,scale=.6,baseline=-3pt]
       \draw (2,0) +(0,-1) -- +(0,1) node[below,at start]{$v$};
       \draw (1,0) +(0,-1) -- +(0,1) node[below,at start]{$u$};\fill
       (1,0) circle (4pt); \node at (.5,0) {$2$};
     \end{tikzpicture}\Bigg)+2 \Bigg(\begin{tikzpicture}[very thick,scale=.6,baseline=-3pt]
       \draw (2,0) +(0,-1) -- +(0,1) node[below,at start]{$v$};
       \draw (1,0) +(0,-1) -- +(0,1) node[below,at start]{$u$};\fill (2,0) circle (4pt);\fill (1,0) circle (4pt);
     \end{tikzpicture}\Bigg)- \Bigg(\begin{tikzpicture}[very thick,scale=.6,baseline=-3pt]
       \draw (2,0) +(0,-1) -- +(0,1) node[below,at start]{$v$};
       \draw (1,0) +(0,-1) -- +(0,1) node[below,at start]{$u$};\fill
       (2,0) circle (4pt); \node at (2.5,0) {$2$}; \end{tikzpicture}\Bigg)+d_1^2h^2 \Bigg(\begin{tikzpicture}[very thick,scale=.6,baseline=-3pt]
       \draw (2,0) +(0,-1) -- +(0,1) node[below,at start]{$v$};
       \draw (1,0) +(0,-1) -- +(0,1) node[below,at start]{$u$};
     \end{tikzpicture}\Bigg)& u=-v,q=-1
   \end{cases}
  \end{equation*}
 \begin{equation*}\subeqn\label{triple-dumb}
    \begin{tikzpicture}[very thick,scale=.8,baseline=-3pt]
      \draw (-2,0) +(1,-1) -- +(-1,1) node[below,at start]{$w$}; \draw
      (-2,0) +(-1,-1) -- +(1,1) node[below,at start]{$u$}; \draw
      (-2,0) +(0,-1) .. controls (-3,0) ..  +(0,1) node[below,at
      start]{$v$}; \node at (-.5,0) {$-$}; \draw (1,0) +(1,-1) -- +(-1,1)
      node[below,at start]{$w$}; \draw (1,0) +(-1,-1) -- +(1,1)
      node[below,at start]{$u$}; \draw (1,0) +(0,-1) .. controls
      (2,0) ..  +(0,1) node[below,at start]{$v$}; \end{tikzpicture}=\quad
      \begin{cases} 
    \begin{tikzpicture}[very thick,scale=.6,baseline=-3pt]
     \draw (6.2,0)
      +(1,-1) -- +(1,1) node[below,at start]{$w$}; \draw (6.2,0)
      +(-1,-1) -- +(-1,1) node[below,at start]{$u$}; \draw (6.2,0)
      +(0,-1) -- +(0,1) node[below,at
      start]{$v$};     \end{tikzpicture}& u=w=qv,q\neq -1\\
    -\begin{tikzpicture}[very thick,scale=.6,baseline]
     \draw (6.2,0)
      +(1,-1) -- +(1,1) node[below,at start]{$w$}; \draw (6.2,0)
      +(-1,-1) -- +(-1,1) node[below,at start]{$u$}; \draw (6.2,0)
      +(0,-1) -- +(0,1) node[below,at
      start]{$v$};     \end{tikzpicture}& u=w=q^{-1}v,q\neq -1\\
 -\begin{tikzpicture}[very thick,scale=.6,baseline=-3pt]
     \draw (0,0)
      +(1,-1) -- +(1,1) node[below,at start]{$w$}; \draw (0,0)
      +(-1,-1) -- +(-1,1) node[below,at start]{$u$}; \draw (0,0)
      +(0,-1) -- +(0,1) node[below,at
      start]{$v$};  \fill
       (-1,0) circle (4pt);   \end{tikzpicture}-\begin{tikzpicture}[very thick,scale=.6,baseline]
     \draw (0,0)
      +(1,-1) -- +(1,1) node[below,at start]{$w$}; \draw (0,0)
      +(-1,-1) -- +(-1,1) node[below,at start]{$u$}; \draw (0,0)
      +(0,-1) -- +(0,1) node[below,at
      start]{$v$};  \fill
       (1,0) circle (4pt);   \end{tikzpicture}& u=w=-v,q= -1\\
      \end{cases}
  \end{equation*}

We will again apply the results of Section \ref{sec:polyn-style-repr},
now with 
\begin{equation}
\mathbb{K}=\K[h]\qquad A=\gls{R}(h)\qquad
  B=\K[h,y]\qquad I=Bh+By.\label{eq:KLR-PR}
\end{equation}
We let $C=B^{\otimes n}=\K[y_1,\dots,y_n,h]=\subset \gls{R}(h)$.
The algebra $\gls{R}(h)$ also has a natural polynomial representation $P$, defined
by Rouquier
in \cite[\S 3.2]{Rou2KM} and Khovanov and Lauda in \cite[\S 2.3]{KLI}.
 This representation $P$ is generated
by a single element $\bone$, with the relations
\[ \psi_ke_{\Bu}\bone=
\begin{cases}
  0 & u_k=u_{k+1}\\
  (y_{k+1}-y_k+d_1h)e_{\Bu^{s_k}}\bone & u_k=qu_{k+1}\\
  e_{\Bu^{s_k}}\bone & u_k\neq u_{k+1},qu_{k+1}.
\end{cases}
\]
This can be written as a sum of the images of $e_{\Bu}$, and we always
have that $e_{\Bu}P$ is a rank 1 free module over
$C$.  

Just as in the Hecke algebra, the action of $\psi_k$ on arbitrary
polynomials can be written in terms of Demazure operators.  For a
polynomial $f\in \K[[h]][y_1,\dots, y_n]$, we can describe the action
as 
\begin{equation}
 \psi_k f e_{\Bu}\bone=
\begin{cases}
 \displaystyle \frac{f^{s_k}-f}{y_{k+1}-y_k} e_{\Bu}\bone& u_k=u_{k+1}\\
  (y_{k+1}-y_k+d_1h)f^{s_k}e_{\Bu^{s_k}}\bone & u_k=qu_{k+1}\\
 f^{s_k} e_{\Bu^{s_k}}\bone & u_k\neq u_{k+1},qu_{k+1}.
\end{cases}\label{eq:KLR-action}
\end{equation}

\begin{lemma}
  The data of (\ref{eq:KLR-PR}) defines a graded polynomial-style
  representation on $P$.  
\end{lemma}
\begin{proof}\mbox{}
  \begin{enumerate}
  \item The algebra $\hR(h)$ is finitely generated free as a
$C$-module by  \cite[Cor. 2.10]{KLI}.  
  \item  By
    \cite[Thm. 2.9]{KLI}, $Z=C^{S_n} $ is central.
    \item The representation $P$ is faithful by \cite[Cor. 2.6]{KLI},
      and free over $C$ of rank $(\# \gls{U})^n$.  
    \end{enumerate}
    The graded property is clear from the definitions; $\K[h,y]$ is
    graded local as required because $h$ and $y$ have positive degree.
    \end{proof}
Let  $\hR(h)$ be the completion of $\gls{R}(h)$
respect to the induced topology and \[\hP\cong
\hR(h)\otimes_{\gls{R}(h)}P\cong \hat{Z}\otimes_{Z} P\]
be the completion of this polynomial representation.   By Lemma
\ref{lem:grading-equivalent}, this is the same as completing these
graded abelian groups with respect to their grading.  We can easily
deduce from Lemma \ref{lem:faithful-completion} that:
\begin{lemma}\label{lem:KLR-P}
The module $\hP$ over $\hR(h)$ is faithful, and the action of $C$ induces an isomorphism $e_{\Bu}\hP\cong \widehat{C}$, the completion of
this ring with respect to its grading topology.  
\end{lemma}

\subsection{Isomorphisms}
Let $b(h)\in 1+h+h^2\K[[h]]$ be a formal power series; if $d_1\neq 0$, we
assume that $b(h)=e^h$. Thus, we must have that
\begin{equation}
b(h_1)d(h_2)=b(h_1+d_1h_2).  \label{eq:b-d}
\end{equation}
Our approach will match Brundan and Kleshchev's if we choose
$b(h)=1+h$.  

\begin{lemma}\label{lem:gammap}
  There is a unique 
vector space isomorphism $\gamma_p\colon \hcP^-\to \hP$ defined by the formula
\begin{equation}
 \gamma_p((u_1^{-1}X_1)^{a_i}\cdots (u_n^{-1}X_n)^{a_n}e_{\Bu})=
\prod_{i=1}^nb(y_1)^{a_1}\cdots b(y_n)^{a_n}e_{\Bu}.\label{eq:poly-map}
\end{equation}
  In particular, under this map, the operator of multiplication by $X_i$ on
    $e_{\Bu} \hcP^-$ is sent to multiplication by $u_ib(y_i)$.  
\end{lemma}Here the subscript $p$ is not a parameter, but
  distinguishes this map from an isomorphism of algebras we'll define
  later.
  \begin{proof}
    By Lemma \ref{lem:Hecke-P}, the elements
    $(u_1^{-1}X_1)^{a_i}\cdots (u_n^{-1}X_n)^{a_n}e_{\Bu}$ are a basis
    of $\hcP^-$, so this map is well-defined.  We will check that it
    is an isomorphism on the image of each idempotent $e_{\Bu}$.  On
    this image, this map is induced by the ring homomorphism
    $\K[[h,(X_1-u_1),\dots, (X_n-u_n)]]\to \K[[h,y_1,\dots,y_n]]$
    sending $X_i-u_i\mapsto u_i(b(y_i)-1)$.  The induced map modulo
    the square of the maximal ideal sends $X_i-u_i\mapsto
    u_iy_i+\cdots$, and so defines an isomorphism of these completed
    polynomial rings.
    By Lemma \ref{lem:KLR-P}, this shows that
    the map is an isomorphism.
  \end{proof}
  Just as in Brundan and Kleshchev, it will be convenient for us to
  use different generators for $\hmH(\pq)$.  Let \[\Phi_r:=
  T_r+\sum_{\Bu\text{ s.t. }u_r\neq
    u_{r+1}}\frac{1-\pq}{1-X_rX_{r+1}^{-1}}e_{\Bu}+\sum_{\Bu\text{ s.t. }u_r=
    u_{r+1}}e_{\Bu}\]
We will freely use the relations involving these given in
\cite[Lem. 4.1]{BKKL}, the most important of which is 
\begin{equation}
\Phi_r e_{\Bu}=e_{\Bu^{s_r}}\Phi_r\label{eq:Phi-e}
\end{equation}
Let  \[\varphi_r(y_r,y_{r+1})=\frac{u_rb(y_r)-\pq
  u_{r+1}b(y_{r+1})}{u_{r+1}b(y_{r+1}) -u_rb(y_r)}=\frac{u_rb(y_r)-q
  u_{r+1}b(y_{r+1}+d_1h)}{u_{r+1}b(y_{r+1}) -u_rb(y_r)}  \]
where the second equality holds by \eqref{eq:b-d}.   Also, let
$\beta(w,z)=\frac{b(w)-b(z)}{w-z};$ note that this is an invertible
element of $\K[[w,z]]$.  
Thus, we have that 
\begin{itemize}
\item $\varphi_r(y_r,y_{r+1})$ is an invertible element of 
  $\K[[h,y_r,y_{r+1}]]$ if and only if $u_r\neq qu_{r+1},u_{r+1}$.  
\item If
  $u_r= qu_{r+1}$, then we have \[\varphi_r(y_r,y_{r+1})=(y_r-y_{r+1})\frac{\beta(y_r,y_{r+1}+d_1h)}{q^{-1} -1+q^{-1}b(y_{r+1}) -b(y_r)}.\]
  This fraction is an invertible power series, since both the numerator
  and denominator have non-zero constant terms.
\item if $u_r=u_{r+1}$,
  then \[\varphi_r(y_r,y_{r+1})=\frac{b(y_r)-\pq b(y_r)}{b(y_{r+1})
    -b(y_r)}=\frac{1}{y_{r+1}
    -y_r}\frac{1-qd(h)+b(y_r)-qd(h)b(y_r)}{\beta(y_{r+1} ,y_r)}\]
  which is also invertible.  
\end{itemize}
Thus we can define an invertible power series by \[A^{\Bu}_{r}=
\begin{cases}
  \displaystyle
\varphi_r(y_r,y_{r+1})(y_{r+1} -y_r)
& u_r =u_{r+1}\\
   \displaystyle \frac{\varphi_r(y_r,y_{r+1})}{y_{r+1}-y_r+d_1h}&
  u_r=qu_{r+1}\\
 \displaystyle \varphi_r(y_r,y_{r+1})&
 u_r\neq u_{r+1}, qu_{r+1}.
\end{cases}
\]

\begin{theorem}\label{O-isomorphism}
  The isomorphism $\gamma_p$ induces an isomorphism $\gamma\colon \hmH(\pq)\cong
  \hR(h)$ such that \[\gamma(X_r)=\sum_{\Bu}u_rb(y_r)e_{\Bu} \qquad \gamma(\Phi_r)=\sum_{\Bu}
  A_r^{\Bu}\psi_re_{\Bu}\]
  which intertwines these two representations, if either $d(h)=1$ (and
  $b(h)$ is arbitrary) or 
  $d(h)=b(h)=e^h$.
\end{theorem}
\begin{proof} 
The match $\gamma(X_r)=\sum_{\Bu}u_rb(y_r)e_{\Bu}$ is clear from
the definition of the map \eqref{eq:poly-map}.  Thus, we turn to considering $\gamma(\Phi_r)$.
Using \eqref{eq:Phi-e} and the definition, one
can easily calculate that
\begin{align*}
  \Phi_r e_{\Bu} \hone&=
  \begin{cases}
    \displaystyle \frac{X_r-\pq X_{r+1}}{X_{r+1}-X_r}
    e_{\Bu^{s_r}} \hone& u_r\neq u_{r+1}\\
    0 & u_r= u_{r+1}
  \end{cases}\\
  \Phi_r (X_{r+1}-X_r) e_{\Bu}\hone&=
  \begin{cases}
   (\pq X_{r+1}-X_r) e_{\Bu^{s_r}} \hone & u_r\neq u_{r+1}\\2(\pq X_{r+1}-X_r)e_{\Bu^{s_r}} \hone& u_r= u_{r+1}
  \end{cases}
\end{align*}
Using the commutation of $\Phi_r$ with symmetric Laurent polynomials
in the $X_i^{\pm 1}$'s, we obtain a general form of action of this
operator on an
arbitrary Laurent polynomial $F\in \K[h,X_1^{\pm 1},\dots, X_n^{\pm
  1}]$.  
\begin{equation}
\Phi_r F(X_1,\dots,X_n) e_{\Bu} \hone =
\begin{cases}
 \displaystyle \frac{X_r-\pq X_{r+1}}{X_{r+1}-X_r}F^{s_r}
  e_{\Bu^{s_r}} \hone& u_r\neq u_{r+1}\\
\displaystyle\frac{X_{r}- \pq X_{r+1}}{X_{r+1}-X_r}(F^{s_r}-F) e_{\Bu}
\hone & u_r= u_{r+1}
\end{cases}\label{eq:1}
\end{equation}

 Now, consider how this operator acts if we intertwine with the
 isomorphism $\gamma_p$;  substituting into the formulas \eqref{eq:1},
 we obtain that for a power series $f\in \K[[h,y_1,\dots, y_n]]$, 
\[\gamma(\Phi_r) f(y_1,\dots, y_n)e_{\Bu}\mathbbm{1} =\begin{cases}
 \displaystyle \varphi_r(y_r,y_{r+1})f^{s_r}
  e_{\Bu^{s_r}} \mathbbm{1}& u_r\neq u_{r+1}\\
\displaystyle \varphi_r(y_r,y_{r+1})
  ( f^{s_r}-f )e_{\Bu}\mathbbm{1}& u_r= u_{r+1}
\end{cases}\]

Thus from \eqref{eq:KLR-action}, 
we immediately obtain that $A^{\Bu}_{r}\psi_r e(\Bu)=\Phi_r e(\Bu).$
Since $A^\Bu_r$ is invertible, this immediately shows that the image
of $\hR(h)$ lies in that of $\hmH(\pq)$ and {\it vice versa}.  Thus, we
obtain an induced isomorphism between these algebras.
\end{proof}

\section{Type W}
\label{sec:weight-gener}

\subsection{Type W Hecke algebras}
\label{sec:type-w-hecke}

The isomorphism of Theorem \ref{O-isomorphism} can be generalized a bit further to include not just KLR
algebras but also weighted KLR algebras, a generalization introduced
by the author in \cite{WebwKLR}.  

Fix a real number $\ck\neq 0$. 
\begin{definition}
A rank $n$ {\bf type W diagram} consists of strands $\R\times [0,1]$ like in a type O
diagram defined above, with addition that we draw a dashed line $\ck$
units to the right of each strand, (which we interpret as $-\ck$ units
left if $\ck<0$). We call this a {\bf ghost}, and require that there are no
triple points or tangencies involving any combination of strands or
ghosts.  We also only consider these diagrams equivalent if they are related by
an isotopy that avoids these tangencies and double points.
\end{definition}
An example of a rank 5 type W diagram with $g<0$ is given below:
\[
\begin{tikzpicture}[baseline,very thick,green!50!black, xscale=2.5,yscale=1.4]
  \draw (-.5,-1) to[out=90,in=-90] (-1,0) to[out=90,in=-90](.5,1);
  \draw[dashed] (-1.25,-1) to[out=90,in=-90] (-1.75,0) to[out=90,in=-90](-.25,1);
  \draw (.5,-1) to[out=90,in=-90] (1,1);
    \draw[dashed] (-.25,-1) to[out=90,in=-90] (.25,1);
  \draw  (1,-1) to[out=90,in=-90] 
  node[midway,fill=green!50!black,inner sep=3pt]{} (0,1);
    \draw[dashed]  (.25,-1) to[out=90,in=-90] (-.75,1);
  \draw (-1, -1) to[out=90,in=-90] (-.5,0) to[out=90,in=-90] (-1,1);
  \draw (0,-1) to[out=90,in=-90] (-.5,1);
  \draw[dashed] (-1.75, -1) to[out=90,in=-90] (-1.25,0) to[out=90,in=-90] (-1.75,1);
  \draw[dashed] (-.75,-1) to[out=90,in=-90] (-1.25,1);
\end{tikzpicture}\]

\begin{definition}\label{def:waha}
The rank $n$ {\bf type W affine Hecke algebra} (WAHA)  $\waha_{\mathscr{B}}(\pq)$ for some collection $\mathscr{B}$
of finite subsets $B_i\subset \R$ is the $\K[[h]]$-span of all rank $n$ type W Hecke diagrams such that endpoints of the strands on the lines
$y=0$ and $y=1$ form a subset in $\mathscr{B}$, modulo the local relations:
\newseq
\begin{equation*}\subeqn\label{nilHecke-2}
    \begin{tikzpicture}[scale=.7,baseline,green!50!black]
      \draw[very thick](-3,0) +(-1,-1) -- +(1,1); \draw[very thick](-3,0) +(1,-1) --
      node[pos=.8,fill=green!50!black,inner sep=3pt]{} +(-1,1) ;
      \node at (-1.5,0){$-$}; \draw[very thick](0,0) +(-1,-1) -- +(1,1); \draw[very thick](0,0) +(1,-1) --  node[pos=.2,fill=green!50!black,inner sep=3pt]{}
      +(-1,1); 
    \end{tikzpicture}\hspace{4mm}=\hspace{4mm}
    \begin{tikzpicture}[scale=.7,baseline]
      \draw[very thick,green!50!black](-3,0) +(-1,-1) --  node[pos=.2,fill=green!50!black,inner sep=3pt]{}+(1,1); \draw[very thick,green!50!black](-3,0) +(1,-1) -- +(-1,1); 
      \node at (-1.5,0){$-$}; \draw[very thick,green!50!black](0,0) +(-1,-1) --
      node[pos=.8,fill=green!50!black,inner sep=3pt]{} +(1,1); \draw[very thick,green!50!black](0,0) +(1,-1) -- +(-1,1)
     ; \node[black] at (2,0){$=$}; \draw[very
      thick,green!50!black](4,0) +(-1,-1) -- +(-1,1); \draw[very
      thick,green!50!black](4,0) +(0,-1) -- +(0,1); 
    \end{tikzpicture}
  \end{equation*}
  \begin{equation*}\subeqn\label{NilHecke3}
    \begin{tikzpicture}[very thick,scale=.9,baseline,green!50!black]
      \draw(-2.8,0) +(0,-1) .. controls (-1.2,0) ..  +(0,1)
      ; \draw (-1.2,0) +(0,-1) .. controls
      (-2.8,0) ..  +(0,1) ; 
    \end{tikzpicture}\hspace{4mm}
= 0\qquad \qquad 
    \begin{tikzpicture}[very thick,scale=.9,baseline,green!50!black]
      \draw (-3,0) +(1,-1) -- +(-1,1); \draw
      (-3,0) +(-1,-1) -- +(1,1) ; \draw
      (-3,0) +(0,-1) .. controls (-4,0) ..  +(0,1); 
    \end{tikzpicture}\hspace{4mm}=\hspace{4mm}
\begin{tikzpicture}[very thick,scale=.9,baseline,green!50!black]
\draw (1,0) +(1,-1) -- +(-1,1)
     ; \draw (1,0) +(-1,-1) -- +(1,1)
      ; \draw (1,0) +(0,-1) .. controls
      (2,0) ..  +(0,1); 
    \end{tikzpicture}\hspace{4mm}
  \end{equation*}
\[ \subeqn\label{green-ghost-bigon1}
\begin{tikzpicture}[very thick,xscale=1.3,baseline=25pt,green!50!black]
 \draw (1,0) to[in=-90,out=90]  (1.5,1) to[in=-90,out=90] (1,2)
;
  \draw[dashed] (1.5,0) to[in=-90,out=90] (1,1) to[in=-90,out=90] (1.5,2);
  \draw (2.5,0) to[in=-90,out=90]  (2,1) to[in=-90,out=90] (2.5,2);
\node[black] at (3,1) {$=$};
  \draw (3.7,0) --node[midway,fill,inner sep=3pt]{} (3.7,2) 
 ;
  \draw[dashed] (4.2,0) to (4.2,2);
  \draw (5.2,0) -- (5.2,2);
\node[black] at (5.6,1) {$-\pq$};
  \draw (6.2,0) -- (6.2,2);
  \draw[dashed] (6.7,0)-- (6.7,2);
  \draw (7.7,0) -- node[midway,fill,inner sep=3pt]{} (7.7,2);
\end{tikzpicture}
\]
\[ \subeqn\label{green-ghost-bigon2}
\begin{tikzpicture}[very thick,xscale=1.3,baseline=25pt,green!50!black]
 \draw[dashed]  (1,0) to[in=-90,out=90]  (1.5,1) to[in=-90,out=90] (1,2)
;
  \draw(1.5,0) to[in=-90,out=90] (1,1) to[in=-90,out=90] (1.5,2);
  \draw (2,0) to[in=-90,out=90]  (2.5,1) to[in=-90,out=90] (2,2);
\node[black] at (3,1) {$=$};
  \draw[dashed] (3.7,0) --(3.7,2) 
 ;
  \draw (4.2,0) to node[midway,fill,inner sep=3pt]{}  (4.2,2);
  \draw (4.7,0) -- (4.7,2);
\node[black] at (5.6,1) {$-\pq$};
  \draw[dashed] (6.2,0) -- (6.2,2);
  \draw (6.7,0)-- (6.7,2);
  \draw (7.2,0) -- node[midway,fill,inner sep=3pt]{} (7.2,2);
\end{tikzpicture}
\]
\begin{equation*}\label{eq:triple-point-1}\subeqn
    \begin{tikzpicture}[very thick,xscale=1.5,baseline,green!50!black]
      \draw[dashed] (-3,0) +(.4,-1) -- +(-.4,1);
 \draw[dashed]      (-3,0) +(-.4,-1) -- +(.4,1); 
    \draw (-2,0) +(.4,-1) -- +(-.4,1); \draw
      (-2,0) +(-.4,-1) -- +(.4,1); 
 \draw (-3,0) +(0,-1) .. controls (-3.5,0) ..  +(0,1);\node[black] at (-1,0) {=};  \draw[dashed] (0,0) +(.4,-1) -- +(-.4,1);
 \draw[dashed]      (0,0) +(-.4,-1) -- +(.4,1); 
    \draw (1,0) +(.4,-1) -- +(-.4,1); \draw
      (1,0) +(-.4,-1) -- +(.4,1); 
 \draw (0,0) +(0,-1) .. controls (.5,0) ..  +(0,1);
\node[black] at (2.1,0) {$-\pq$};
     \draw (4,0)
      +(.4,-1) -- +(.4,1); \draw (4,0)
      +(-.4,-1) -- +(-.4,1); 
 \draw[dashed] (3,0)
      +(.4,-1) -- +(.4,1); \draw[dashed] (3,0)
      +(-.4,-1) -- +(-.4,1); 
\draw (3,0)
      +(0,-1) -- +(0,1);
    \end{tikzpicture}
  \end{equation*}
\begin{equation*}\label{eq:triple-point-2}\subeqn
    \begin{tikzpicture}[very thick,xscale=1.5,baseline,green!50!black]
      \draw (-3,0) +(.4,-1) -- +(-.4,1);
 \draw     (-3,0) +(-.4,-1) -- +(.4,1); 
\draw (-2,0) +(0,-1) .. controls (-2.5,0) ..  +(0,1);
 \draw[dashed] (-3,0) +(0,-1) .. controls (-3.5,0) ..  +(0,1);\node[black] at (-1,0) {$=$};  \draw (0,0) +(.4,-1) -- +(-.4,1);
 \draw   (0,0) +(-.4,-1) -- +(.4,1); 
    \draw[dashed] (0,0) +(0,-1) .. controls (.5,0) ..  +(0,1);
 \draw (1,0) +(0,-1) .. controls (1.5,0) ..  +(0,1);
\node[black] at (2,0)
      {$+$};   
 \draw (3,0)
      +(.4,-1) -- +(.4,1); \draw (3,0)
      +(-.4,-1) -- +(-.4,1); 
 \draw[dashed] (3,0)
      +(0,-1) -- +(0,1); \draw (4,0)
      +(0,-1) -- +(0,1); 
    \end{tikzpicture}
  \end{equation*}

\end{definition}
\begin{remark}
  As in the type O case, this algebra also has a degenerate analogue,
  where we replace
  (\ref{green-ghost-bigon1}--\ref{eq:triple-point-1}) with the
  equations 
\[ \subeqn\label{deg-green-ghost-bigon1}
\begin{tikzpicture}[very thick,xscale=1.3,baseline=25pt,green!50!black]
 \draw (1,0) to[in=-90,out=90]  (1.5,1) to[in=-90,out=90] (1,2)
;
  \draw[dashed] (1.5,0) to[in=-90,out=90] (1,1) to[in=-90,out=90] (1.5,2);
  \draw (2.5,0) to[in=-90,out=90]  (2,1) to[in=-90,out=90] (2.5,2);
\node[black] at (3,1) {$=$};
  \draw (3.7,0) --node[midway,fill,inner sep=3pt]{} (3.7,2) 
 ;
  \draw[dashed] (4.2,0) to (4.2,2);
  \draw (5.2,0) -- (5.2,2);
\node[black] at (5.6,1) {$-$};
  \draw (6.2,0) -- (6.2,2);
  \draw[dashed] (6.7,0)-- (6.7,2);
  \draw (7.7,0) -- node[midway,fill,inner sep=3pt]{} (7.7,2);
\node[black] at (8.4,1) {$-$};
  \draw (9,0) -- (9,2);
  \draw[dashed] (9.5,0)-- (9.5,2);
  \draw (10.5,0) -- (10.5,2);
\end{tikzpicture}
\]
\[ \subeqn\label{deg-green-ghost-bigon2}
\begin{tikzpicture}[very thick,xscale=1.3,baseline=25pt,green!50!black]
 \draw[dashed]  (1,0) to[in=-90,out=90]  (1.5,1) to[in=-90,out=90] (1,2)
;
  \draw(1.5,0) to[in=-90,out=90] (1,1) to[in=-90,out=90] (1.5,2);
  \draw (2,0) to[in=-90,out=90]  (2.5,1) to[in=-90,out=90] (2,2);
\node[black] at (3,1) {$=$};
  \draw[dashed] (3.7,0) --(3.7,2) 
 ;
  \draw (4.2,0) to node[midway,fill,inner sep=3pt]{}  (4.2,2);
  \draw (4.7,0) -- (4.7,2);
\node[black] at (5.6,1) {$-$};
  \draw[dashed] (6.2,0) -- (6.2,2);
  \draw (6.7,0)-- (6.7,2);
  \draw (7.2,0) -- node[midway,fill,inner sep=3pt]{} (7.2,2);
\node[black] at (7.8,1) {$-$};
  \draw[dashed] (8.4,0) -- (8.4,2);
  \draw (8.9,0)-- (8.9,2);
  \draw (9.4,0) -- (9.4,2);
\end{tikzpicture}
\]
\begin{equation*}\label{eq:deg-triple-point-1}\subeqn
    \begin{tikzpicture}[very thick,xscale=1.5,baseline,green!50!black]
      \draw[dashed] (-3,0) +(.4,-1) -- +(-.4,1);
 \draw[dashed]      (-3,0) +(-.4,-1) -- +(.4,1); 
    \draw (-2,0) +(.4,-1) -- +(-.4,1); \draw
      (-2,0) +(-.4,-1) -- +(.4,1); 
 \draw (-3,0) +(0,-1) .. controls (-3.5,0) ..  +(0,1);\node[black] at (-1,0) {=};  \draw[dashed] (0,0) +(.4,-1) -- +(-.4,1);
 \draw[dashed]      (0,0) +(-.4,-1) -- +(.4,1); 
    \draw (1,0) +(.4,-1) -- +(-.4,1); \draw
      (1,0) +(-.4,-1) -- +(.4,1); 
 \draw (0,0) +(0,-1) .. controls (.5,0) ..  +(0,1);
\node[black] at (2.1,0) {$-$};
     \draw (4,0)
      +(.4,-1) -- +(.4,1); \draw (4,0)
      +(-.4,-1) -- +(-.4,1); 
 \draw[dashed] (3,0)
      +(.4,-1) -- +(.4,1); \draw[dashed] (3,0)
      +(-.4,-1) -- +(-.4,1); 
\draw (3,0)
      +(0,-1) -- +(0,1);
    \end{tikzpicture}
  \end{equation*}
\end{remark}

By convention, we'll let $e_{B}$ be the diagram with vertical lines at
$x=b$ for $b\in B$, and use $X_i$ to represent the square on the $i$th
strand from left.
\begin{proposition}\label{W-poly}
  The WAHA $\waha_{\mathscr{B}}(\pq)$ for a set $\mathscr{B}$ has a polynomial representation
  \[P_{\mathscr{B}} := \oplus_{B\in \mathscr{B}}\K[[h]][Y_1^{\pm 1},\dots,
  Y_{|B|}^{\pm 1}]\]
defined by the rule that 
\begin{itemize}
\item Each crossing of the $r$ and $r+1$st strands acts by the Demazure operator \[\partial_r(F)=
  \frac{F^{s_r}-F}{Y_{r+1}-Y_{r}}.\]
\item 
A crossing between the $r$th strand and a ghost of $s$th strand acts
by
\begin{itemize}
\item the identity if $\ck <0$ and the
  strand is NE/SW or $\ck >0$ and the strand is NW/SE,
\item the multiplication operator of  $Y_r-qY_s$ if $\ck <0$ and
  the strand is NW/SE or $\ck >0$ and the strand is NE/SW
\end{itemize}
\item A square on the $r$th strand acts by the multiplication operator
  $Y_r$.  
\end{itemize}
\end{proposition}
\begin{proof}
The equations (\ref{nilHecke-2}--\ref{NilHecke3})  are the usual
relations satisfied by multiplication and Demazure operators.  The 
equations   (\ref{green-ghost-bigon1}--\ref{green-ghost-bigon2}) are clear from
the definition of the operators for ghost/strand crossings.  Finally,
the relations (\ref{eq:triple-point-1}--\ref{eq:triple-point-2}) are
calculation with Demazure operators similar to that which is standard
for triple points in various KLR calculi.
For example, assuming $\ck<0$ for (\ref{eq:triple-point-1}), the
LHS is \[\partial_s\circ (Y_r-qY_s)=(Y_r-qY_{s+1})\circ\partial_s-q\]
using the usual twisted Leibnitz rule for Demazure operators; this is
the RHS, so we are done.  On the other hand,
(\ref{eq:triple-point-2}) follows in a similar way from the equation
\[(Y_r-qY_s)\circ\partial_r = \partial_r \circ (Y_{r+1}-qY_s)+1.\]  This
completes the proof.
\end{proof}

\begin{proposition}\label{prop:waha-basis}
  The rank $n$ type W Hecke algebra $\waha_{\mathscr{B}}(\pq)$ has a basis over $\K[[h]]$ given by the products
  $e_{B}D_wX_1^{a_1}\cdots X_n^{a_n}e_{B'}$ for $w\in S_n$ and $(a_1,\dots,
  a_n)\in \Z^n$; here $D_w$ is a arbitrarily chosen diagram which
  induces the permutation $w$ on the endpoints at $y=0$ when they are
  associated to the endpoint at the top of same strand, and no pair of
  strands or ghosts cross twice.

  The action of  $\waha_{\mathscr{B}}(\pq)$ on its polynomial
  representation is faithful.
\end{proposition}
\begin{proof}
  This proof follows many similar ones in KLR theory.  These elements
  are linearly independent because the elements $D_w$ span the action
  of $\K[S_n]$ after extending scalars to the fraction field of
  rational functions, since $D_w=f_w w+\sum_{v<w} f_v v$ for some
  rational functions $f_v$ with $f_w\neq 0$.  Thus our proposed basis
  is linearly independent over $\K$ in this scalar extension, so must
  have been linearly independent before.

  Note that this shows that the action of these elements on the
  polynomial representation is linearly independent.  Thus, if we show
  that they span, it will show that the representation is faithful.  

  Now we need only show that they span.  Using relation
  (\ref{nilHecke-2}), we can assume that all squares are at the bottom
  of the diagram.  

Furthermore, any two choices of the diagram $D_w$ differ via a series
of isotopies and triple points, so relations
(\ref{NilHecke3},\ref{eq:triple-point-1},\ref{eq:triple-point-2}) show
that these diagrams differ by diagrams with fewer crossings between
strands and ghosts.  Thus, we
need only show that any diagram with a bigon can be written as a sum
of diagrams with fewer crossings.  

Now, assume we have such a bigon.  We should assume that it has no
smaller bigons inside it.  In this case, we can shrink the bigon,
using the relations
(\ref{NilHecke3},\ref{eq:triple-point-1},\ref{eq:triple-point-2})
whenever we need to move a strand through the top and bottom of the
bigon or a crossing out through its side.  Thus, we can ultimately
assume that the bigon is empty, and apply the relations (\ref{NilHecke3}--\ref{green-ghost-bigon2}).
\end{proof}

We now have the results we need to apply the results of Section
\ref{sec:polyn-style-repr}, in the case of
\begin{equation}
\mathbb{K}=\K[[h]]\qquad A=\waha_{\mathscr{B}}(\pq)\qquad
  B=\K[[h]][X^{\pm}]\qquad I=Bh+B \prod_{u\in \gls{U}}(X-u)\qquad P=\mathcal{P}_{\mathscr{B}}.\label{eq:wHecke-PR}
\end{equation}
The requisite freeness and the faithfulness of the polynomial
representation follow from Proposition
\ref{prop:waha-basis}, so this defines a polynomial style
representation.  Thus, we have an induced completion
$\gls{hwaha}_{\mathscr{B}}$ with faithful completed polynomial
representation by Lemma \ref{lem:faithful-completion}.

\subsubsection{Comparison with Hecke and Schur algebras}
\label{sec:comp-with-hecke}

Choose $\wellsep=\{B_s=\{s,2s,3s,\dots, ns\}\}$ for $s$ some real
number with $s\gg |g|$.  For every type O diagram on $n$ strands,
we can choose an isotopy representative such that the endpoints of the
diagram are precisely $B_s$ at both $y=0$ and $y=1$.  Furthermore, we
can choose this representative so that if we think of it as a type W
diagram and add ghosts, no strand is between a crossing
of strands and the corresponding ghost crossing.  Obviously we can do
this for individual crossings, and any diagram can be factored into
these.
\begin{theorem}\label{wdHecke}
  This embedding induces an isomorphism between the WAHA $\waha_{\wellsep}(\pq)$ and
  the honest affine Hecke algebra
  $\gls{mH}(\pq)$:
  \begin{itemize}
  \item If $\ck <0$, this isomorphism sends a single crossing to
    $T_i+1$.  That is, the diagrams satisfy the local relations
    (\ref{Hecke-1}--\ref{Hecke-triple}).
\item If $\ck>0$, this
    isomorphism sends a single crossing to $T_i-q$.  That is, the diagrams satisfy the local relations
    (\ref{qHecke-1}--\ref{qHecke-triple}).
  \end{itemize}
The polynomial representation defined above
  is intertwined by this map with the polynomial representation of
  $\gls{mH}(\pq)$ if $\ck <0$ and the signed polynomial representation if $\ck>0$.  
\end{theorem}
This theorem shows that if we view type O diagrams as type W diagrams
where $|\ck|$ is sufficiently small that we cannot distinguish
between a strand and its ghost\footnote{Perhaps this will be easier if
  you take off your glasses.}, then the local relations
(\ref{Hecke-1}--\ref{Hecke-triple}) will be consequences of (\ref{nilHecke-2}--\ref{eq:triple-point-2}).
\begin{proof}
We'll consider the case where $\ck <0$.  We have that $T_i+1$ is sent to the diagram 
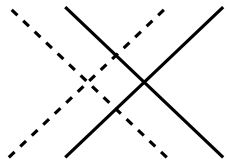
\begin{equation*}
    \begin{tikzpicture}[very thick,xscale=1.5,baseline,green!50!black]
      \draw (-2.5,0) +(.7,-1) -- +(-.7,1);
 \draw     (-2.5,0) +(-.7,-1) -- +(.7,1); 
      \draw[dashed] (-3,0) +(.7,-1) -- +(-.7,1);
 \draw [dashed]    (-3,0) +(-.7,-1) -- +(.7,1); 
    \end{tikzpicture}
  \end{equation*}
which sent by the polynomial representation of  the type W affine Hecke algebra
  representation to $(Y_r-\pq Y_{r+1})  \circ \partial_r$.  That is, we
  have $T_iF=-F^{s_r}+(1-\pq)Y_{r+1}\partial_r$.  Since
  $\waha_{\mathscr{B}}(\pq)$ acts faithfully on its polynomial representation,
  this shows that we have a map of the Hecke algebra to the WAHA;  the
  faithfulness of $\cP^-$ implies that this map is injective.  
  Since the diagram $D_w$ and the polynomials in the squares are in
  the image of this map, the map is surjective.  

The case $\ck>0$ follows similarly.
\end{proof}
Thus, the WAHA for any set containing $\wellsep$ is a ``larger'' algebra than the affine Hecke
algebra. The category of representations of affine Hecke algebras are
a quotient category of its representations via the functor $M\mapsto e_{\wellsep}M$, though in some cases, this
quotient will be an equivalence.  

For any composition $\Bk=(k_1,\dots, k_n)$ of $m$, we have an associated
quasi-idempotent $\epsilon_{\Bk}=\sum_{w\in S_{\Bk}}T_w$ symmetrizing for the associated Young
subgroup.  If $\Bk=(1,\dots, 1)$, then $\epsilon_{\Bk}=1$.

\begin{definition}\label{def:cS}
  The {\bf affine $q$-Schur algebra} $\cS(\pq,n,m)$, as defined in
  \cite[Def. 2.1.4]{GreenSchur}, is the algebra defined
  by
  \[\cS(\pq,n,m):=
  \operatorname{End}_{\gls{mH}(\pq)}\Big(\bigoplus_{|\Bk|=m}\epsilon_{\Bk}\gls{mH}(\pq)
  \Big)
  \] where the sum is over $n$-part compositions of $m$.
\end{definition}
Following \cite{GreenSchur}, we let $E(n,m)$ denote the $\cS(\pq,n,m)$-$\gls{mH}(\pq)$ bimodule $\bigoplus_{|\Bk|=m}\epsilon_{\Bk}\gls{mH}(\pq)$.

By a result of Jimbo \cite{Jimbo}, the affine Hecke algebra acts naturally on $M\otimes V^{\otimes n}$ for any finite
dimensional $U_{\pq}(\mathfrak{gl}_n)$-module $M$ and $V$ the defining
representation using universal R-matrices and Casimir operators; analogously, the
algebra $\cS(\pq,n,m)$ naturally acts on
\[\bigoplus _{|\Bk|=m}M\otimes \Sym^{ k_1}V\otimes \cdots\otimes
\Sym^{k_n}V\cong E(n,m)\otimes_{\gls{mH}(\pq)}M\otimes V^{\otimes n}.\]
Furthermore, the algebra $\cS(\pq,n,m)$ has a natural 
polynomial representation given by
\[\cP_{\cS}:=\bigoplus_{|\Bk|=m}
  \cP^{S_{\Bk}}\cong E(n,m)\otimes_{\gls{mH}(\pq)}\cP^-.\]
There is a more detailed exposition of this representation in \cite[\S 4]{MiSt}.
\begin{lemma}
  This representation is faithful.
\end{lemma}
\begin{proof}
  The algebra $\cS(\pq,n,m)$ have a basis $\phi^d_{\Bk,\Bk'}$ defined in
  \cite[Def. 2.2.3]{GreenSchur}.  This element is defined as a linear
  combination of left multiplications of elements of $\gls{mH}(\pq)$,
  restricted to $\epsilon_{\Bk}\gls{mH}(\pq)$.  Thus, any non-trivial linear
  combination of these elements has the same property.
  By the faithfulness of
  $\cP^-$, this implies that no non-trivial linear combination of
  $\phi^d_{\Bk,\Bk'}$ acts trivially.  That is, the action is
  faithful.  
\end{proof}

If we replace $\epsilon_{\Bk}$ by the anti-symmetrizing quasi-idempotent
$\epsilon_{\Bk}^-=\sum_{w\in S_{\Bk}} (-q)^{\ell(w)}T_w$, then we obtain the signed
$q$-Schur algebra $\cS_h^-(n,m)$, which instead acts on \[\bigoplus
_{|\Bk|=m}M\otimes \iwedge{ k_1}V\otimes \cdots\otimes
\iwedge{k_n}V.\]

The affine $q$-Schur algebra has a diagrammatic realization much like the
affine Hecke algebra.  For each composition $\mu=(\mu_1,\dots, \mu_n)$
of $m$, we let $C_\mu=\{i\epsilon+js\mid 0\leq i< \mu_j\}$ for some
fixed $0< \epsilon \ll g \ll s$, and let $\gls{Cset}$ be the
collection of these sets.  That is, we have groups of dots
corresponding to the parts of the composition, with sizes given by
$\mu_i$.  

In the type W affine Hecke algebra
$\waha_{\gls{Cset}}(\pq)$, we have an idempotent $e'_\mu$ which on each
group in $[js,js+\mu_j\epsilon]$ traces out the primitive idempotent
in the nilHecke algebra which acts as $\partial_{w_0} y^{\mu_j-1}_1\cdots
y_{\mu_j-1}$ in the polynomial representation.  For example, for
$\mu=(1,3,2)$, this idempotent is given by:
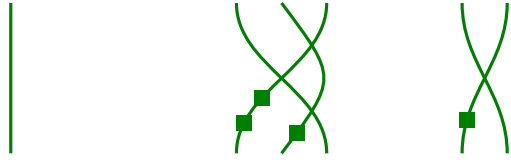
\[
\begin{tikzpicture}[very thick,xscale=1.5,baseline,green!50!black]
  \draw (0,-1) -- (0,1);
  \draw (4,-1) to[out=90,in=-90] node[pos=.2,fill=green!50!black,inner
     sep=3pt]{} (4.4, 1);
  \draw (4.4,-1) to[out=90,in=-90] (4, 1);
  \draw (2,-1) to[out=90,in=-90] node[pos=.35,fill=green!50!black,inner
     sep=3pt]{}node[pos=.175,fill=green!50!black,inner
     sep=3pt]{} (2.8, 1);
  \draw (2.8,-1) to[out=90,in=-90]  (2, 1);
  \draw (2.4,-1) .. controls (2.9,0) .. node[pos=.1,fill=green!50!black,inner
     sep=3pt]{}  (2.4, 1);
\end{tikzpicture}
\]
 Let
$\gls{eprime}=\sum_\mu e'_\mu$ be the sum of these idempotents over $m$-part
compositions of $n$.

\begin{theorem}\label{waha-Schur}
If $\ck<0$,   we have an isomorphism of algebras $\gls{eprime} \waha_{\gls{Cset}}(\pq) \gls{eprime}\cong \cS(\pq,n,m)$
  which induces an isomorphism of representations $\gls{eprime}
  P_{\gls{Cset}}\cong \cP_{\cS(\pq,n,m)}$.  Similarly, if $\ck>0$, we have an
  isomorphism of algebras $\gls{eprime} \waha_{\gls{Cset}}(\pq) \gls{eprime}\cong \cS_h^-(n,m)$.
\end{theorem}

  Setting $h=0$, we obtain an isomorphism between the WAHA $\gls{eprime} \waha_{\gls{Cset}}(\pq) \gls{eprime}$ (at $h=0$)
  with the usual affine Schur algebra for any field $\K$ and any
  $q\notin\{0,1\}$.  Since this isomorphism requires passing through a
  Morita equivalence, it is quite difficult to make it explicit.  A
  closely related isomorphism is shown in much greater detail by
  Miemietz and Stroppel in \cite{MiSt}, relating the affine Schur
  algebra and the quiver Schur algebra from \cite{SWschur}; presumably
  these results can ultimately be matched by tracing through the
  Morita equivalence of \cite[Th. 3.8]{WebwKLR}, but we will not trace
  through the details of doing so.  

\begin{proof}
First, consider the case $\ck<0$.  Consider the idempotent $e_{B_s}$ in $\gls{eprime} \waha_{\gls{Cset}}(\pq) \gls{eprime}$.  This
satisfies $e_{B_s}\waha_{\gls{Cset}}(\pq)e_{B_s}\cong \gls{mH}(\pq)$ by Theorem \ref{wdHecke}.

Thus, $\gls{eprime}e_{C_\mu}\waha_{\gls{Cset}}(\pq)e_{B_s}$ is naturally a right module
over $\gls{mH}(\pq)$.  We wish to show that it is isomorphic to $\epsilon_\mu
\gls{mH}(\pq)$.  Consider the diagram $e_{C_\mu}D_1e_{B_s}$.  Acting on the
right by $T_i+1$ with $(i,i+1)\in S_{\mu_1}\times \cdots \times S_{\mu_p}$ gives $e_{C_\mu}D_ee_{B_s}(T_i+1)=(q+1)
e_{C_\mu}D_1e_{B_s}$, since
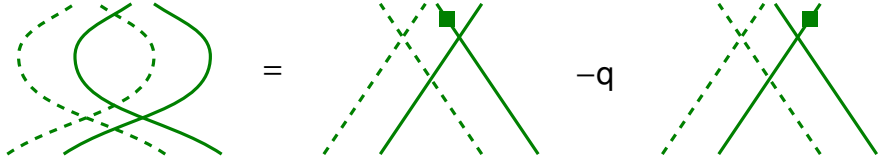
\begin{equation*}
    \begin{tikzpicture}[very thick,xscale=1.5,baseline,green!50!black]
      \draw (.7,-1) to [out=135,in=-90]  (-.6,.3) to
      [out=90,in=-135] (-.1,1) ;
 \draw     (-.7,-1)to [out=45,in=-90]  (.6,.3) to
      [out=90,in=-45] (.1,1); 
      \draw[dashed] (.2,-1) to [out=135,in=-90]  (-1.1,.3) to
      [out=90,in=-135] (-.6,1) ;
 \draw [dashed]    (-1.2,-1) to [out=45,in=-90]  (.1,.3) to
      [out=90,in=-45] (-.4,1) ;
    \end{tikzpicture}\hspace{4mm}=\hspace{4mm}
   \begin{tikzpicture}[xscale=1.5,baseline,green!50!black]
     \draw[very thick](-3,0) +(-.7,-1) -- +(.2,1); \draw[very thick](-3,0) +(.7,-1) -- node[pos=.9,fill=green!50!black,inner
     sep=3pt]{}+(-.2,1)
     ; \draw[very thick,dashed](-3.5,0) +(-.7,-1) -- +(.2,1);
     \draw[very thick,dashed](-3.5,0)+(.7,-1) -- +(-.2,1) ;
      \node[black] at (-1.8,0){$-\pq$}; \draw[very thick](0,0) +(-.7,-1) --   node[pos=.9,fill=green!50!black,inner sep=3pt]{}+(.2,1); \draw[very thick](0,0) +(.7,-1) --  
     +(-.2,1); 
\draw[very thick,dashed](-.5,0) +(-.7,-1) -- +(.2,1);
\draw[very thick,dashed] (-.5,0) +(.7,-1) -- +(-.2,1);
    \end{tikzpicture}
  \end{equation*}
Applying (\ref{nilHecke-2}), the RHS is equal to $1+q$ times the
identity, plus diagrams with a crossing at top, which are killed by $\gls{eprime}$.  This shows that
$\gls{eprime}e_{C_\mu}D_1e_{B_s}$ is invariant.  Thus, we have a map of
$\epsilon_\mu \gls{mH}(\pq)\to \gls{eprime}e_{C_\mu}\waha_{\gls{Cset}}(\pq)e_{B_s}$ sending
$\epsilon_\mu \mapsto \gls{eprime}e_{C_\mu}D_ee_{B_s}$.  This map must be
surjective, since every $ \gls{eprime}e_{C_\mu}D_we_{B_s}$ is in its image, and
comparing ranks over the fraction field $K=\K(X_1,\dots, X_n)$, we see that it
must be injective as well.  Thus, the action of $\gls{eprime} \waha_{\gls{Cset}}(\pq) \gls{eprime}$
on $\gls{eprime} \waha_{\gls{Cset}}(\pq) e_{B_s}$ defines a map $\gls{eprime} \waha_{\gls{Cset}}(\pq) \gls{eprime}\to
\cS_h$.  

Assume $a\neq 0$ is in the kernel of this map 
$\gls{eprime} \waha_{\gls{Cset}}(\pq) \gls{eprime}$; that is, $a$ acts trivially on $\gls{eprime}
\waha_{\gls{Cset}}(\pq) e_{B_s}$.  Note that
$\waha_{\gls{Cset}}(\pq)$ acts faithfully on the rational representation
$P^K_{\gls{Cset}}=P_{\gls{Cset}}\otimes_{\K[X_1^{\pm 1},\dots,X_n^{\pm 1}]}K$, and the element $D_1$ induces an isomorphism $
e_{C_{\mu} }P_{\gls{Cset}}^K\otimes F\to e_{B_s }P_{\gls{Cset}}^K\otimes
F$.  Thus, we must have that $aD_1$ then acts non-trivially in $ e_{B_s }P_{\gls{Cset}}$, and so $aD_1e_{B_s }\neq 0$, contradicting our assumption that
$a$ is in the kernel.
Thus, we can only have $a=0$, and the map to the Schur algebra is injective.

On the other hand, note that the element $\gls{eprime}e_{C_\mu}D_we_{C_{\mu'}}e$
for $w$ any shortest double coset representative is sent to the element
$\phi_w=\sum_{w'\in S_{\mu}w S_{\mu'}}T_{w'}$ plus elements in
$\K[[h]][X^{\pm 1}_1,\dots, X^{\pm 1}_n]\phi_v$ for $v$ shorter in Bruhat
order.  Since $\phi_w X^{a_1}_1,\dots, X^{a_n}_n$ give a basis of
$\cS_h$ by \cite[2.2.2]{GreenSchur}, the fact that these are in the image shows that this map is
surjective.

When $\ck >0$, the argument is quite similar, but with $T_i-q$
replacing $T_i+1$ and using the $\ck>0$ version of Theorem \ref{wdHecke}.
\end{proof}
We'll prove in Corollary \ref{affine-schur-morita} that the idempotent
$\gls{eprime}$ induces a Morita equivalence between $\waha_{\gls{Cset}}$ and $\gls{cS}_h(n,m)$.
Thus, from the perspective of the Hecke side, introducing the type W
relations is an alternate way of understanding the affine Schur
algebra.

\subsection{Weighted KLR algebras}
\label{sec:weight-klr-algebr}

On the other hand, the author has incorporated similar ideas
into the theory of KLR algebras, by introducing {\bf weighted KLR
  algebras} \cite{WebwKLR}.
\begin{definition}\label{def:wKLR}
  Let $\gls{W}(\pq)$ be the rank $n$ weighted KLR algebra attached to the graph $\gls{U}$; that
  is, $\gls{W}(\pq)$ is the quotient of $\K[h]$-span of weighted KLR
  diagrams with $n$ strands
  (as defined in \cite[Def. 2.3]{WebwKLR}) by the local relations
  (note that these relations are drawn with $\ck<0$): \newseq
  \begin{equation*}\subeqn\label{dots-1}
    \begin{tikzpicture}[scale=.45,baseline]
      \draw[very thick](-4,0) +(-1,-1) -- +(1,1) node[below,at start]
      {$u$}; \draw[very thick](-4,0) +(1,-1) -- +(-1,1) node[below,at
      start] {$v$}; \fill (-4.5,.5) circle (5pt);
      \node at (-2,0){=}; \draw[very thick](0,0) +(-1,-1) -- +(1,1)
      node[below,at start] {$u$}; \draw[very thick](0,0) +(1,-1) --
      +(-1,1) node[below,at start] {$v$}; \fill (.5,-.5) circle (5pt);
      \node at (4,0){for $i\neq j$};
    \end{tikzpicture}\end{equation*}
  \begin{equation*}\label{dots-2}\subeqn
    \begin{tikzpicture}[scale=.45,baseline]
      \draw[very thick](-4,0) +(-1,-1) -- +(1,1) node[below,at start]
      {$u$}; \draw[very thick](-4,0) +(1,-1) -- +(-1,1) node[below,at
      start] {$u$}; \fill (-4.5,.5) circle (5pt);
      \node at (-2,0){=}; \draw[very thick](0,0) +(-1,-1) -- +(1,1)
      node[below,at start] {$u$}; \draw[very thick](0,0) +(1,-1) --
      +(-1,1) node[below,at start] {$u$}; \fill (.5,-.5) circle (5pt);
      \node at (2,0){+}; \draw[very thick](4,0) +(-1,-1) -- +(-1,1)
      node[below,at start] {$u$}; \draw[very thick](4,0) +(0,-1) --
      +(0,1) node[below,at start] {$u$};
    \end{tikzpicture}\qquad
    \begin{tikzpicture}[scale=.45,baseline]
      \draw[very thick](-4,0) +(-1,-1) -- +(1,1) node[below,at start]
      {$u$}; \draw[very thick](-4,0) +(1,-1) -- +(-1,1) node[below,at
      start] {$u$}; \fill (-4.5,-.5) circle (5pt);
      \node at (-2,0){=}; \draw[very thick](0,0) +(-1,-1) -- +(1,1)
      node[below,at start] {$u$}; \draw[very thick](0,0) +(1,-1) --
      +(-1,1) node[below,at start] {$u$}; \fill (.5,.5) circle (5pt);
      \node at (2,0){+}; \draw[very thick](4,0) +(-1,-1) -- +(-1,1)
      node[below,at start] {$u$}; \draw[very thick](4,0) +(0,-1) --
      +(0,1) node[below,at start] {$u$};
    \end{tikzpicture}
  \end{equation*}
  \begin{equation*}\label{strand-bigon}\subeqn
    \begin{tikzpicture}[very thick,scale=.8,baseline]
      \draw (-2.8,0) +(0,-1) .. controls (-1.2,0) ..  +(0,1)
      node[below,at start]{$u$}; \draw (-1.2,0) +(0,-1) .. controls
      (-2.8,0) ..  +(0,1) node[below,at start]{$u$}; \node at (-.5,0)
      {=}; \node at (0.4,0) {$0$}; \node at (1.5,.05) {and};
    \end{tikzpicture}
    \hspace{.4cm}
    \begin{tikzpicture}[very thick,scale=.8 ,baseline]

      \draw (-2.8,0) +(0,-1) .. controls (-1.2,0) ..  +(0,1)
      node[below,at start]{$u$}; \draw (-1.2,0) +(0,-1) .. controls
      (-2.8,0) ..  +(0,1) node[below,at start]{$v$}; \node at (-.5,0)
      {=};

      \draw (1.8,0) +(0,-1) -- +(0,1) node[below,at start]{$v$}; \draw
      (1,0) +(0,-1) -- +(0,1) node[below,at start]{$u$};\node at (3.3,0){for $u\neq v$};
    \end{tikzpicture}
  \end{equation*} 
  \begin{equation*}\label{ghost-bigon1}\subeqn
    \begin{tikzpicture}[very thick,xscale=1.25 ,yscale=.8,baseline]
      \draw (1,-1) to[in=-90,out=90] node[below, at start]{$u$}
      (1.5,0) to[in=-90,out=90] (1,1) ; \draw[dashed] (1.5,-1)
      to[in=-90,out=90] (1,0) to[in=-90,out=90] (1.5,1); \draw
      (2.5,-1) to[in=-90,out=90] node[below, at start]{$v$} (2,0)
      to[in=-90,out=90] (2.5,1); 
\node at (3,0) {=}; 
 \end{tikzpicture}
 \begin{cases}
      \begin{tikzpicture}[very thick,xscale=1.25 ,yscale=.8,baseline]
\draw (3.7,-1) --
      (3.7,1) node[below, at start]{$u$} ; \draw[dashed] (4.2,-1) to
      (4.2,1); \draw (5.2,-1) -- (5.2,1) node[below, at start]{$v$};
    \end{tikzpicture} & \text{for $u\neq qv$}\\
\begin{tikzpicture}[very thick,xscale=1.25,yscale=.8,baseline]
  \draw (3.7,-1) --
      (3.7,1) node[below, at start]{$u$} ; \draw[dashed] (4.2,-1) to
      (4.2,1); \draw (5.2,-1) -- (5.2,1) node[below, at start]{$v$}
      node[midway,fill,inner sep=2.5pt,circle]{}; \node at (5.75,0)
      {$-$};

      \draw (6.2,-1) -- (6.2,1) node[below, at start]{$u$}
      node[midway,fill,inner sep=2.5pt,circle]{}; \draw[dashed]
      (6.7,-1)-- (6.7,1); \draw (7.7,-1) -- (7.7,1) node[below, at
      start]{$v$};
\node at (8.25,0)
      {$+h$};
  \draw (8.7,-1) -- (8.7,1) node[below, at start]{$u$}; \draw[dashed]
      (9.2,-1)-- (9.2,1); \draw (10.2,-1) -- (10.2,1) node[below, at
      start]{$v$};
    \end{tikzpicture} & \text{for $u= qv$}
 \end{cases}
  \end{equation*}  
 \begin{equation*}\label{ghost-bigon1a}\subeqn
    \begin{tikzpicture}[very thick,xscale=1.25 ,yscale=.8,baseline]
      \draw (1.5,-1) to[in=-90,out=90] node[below, at start]{$u$}
      (1,0) to[in=-90,out=90] (1.5,1) ; \draw[dashed] (1,-1)
      to[in=-90,out=90] (1.5,0) to[in=-90,out=90] (1,1); \draw
      (2,-1) to[in=-90,out=90] node[below, at start]{$v$} (2.5,0)
      to[in=-90,out=90] (2,1); 
\node at (3,0) {=}; 
 \end{tikzpicture}
 \begin{cases}
      \begin{tikzpicture}[very thick,xscale=1.25 ,yscale=.8,baseline]
\draw (4.7,-1) --
      (4.7,1) node[below, at start]{$u$} ; \draw[dashed] (4.2,-1) to
      (4.2,1); \draw (5.2,-1) -- (5.2,1) node[below, at start]{$v$};
    \end{tikzpicture} & \text{for $u\neq qv$}\\
\begin{tikzpicture}[very thick,xscale=1.25,yscale=.8,baseline]
  \draw (4.7,-1) --
      (4.7,1) node[below, at start]{$u$} ; \draw[dashed] (4.2,-1) to
      (4.2,1); \draw (5.2,-1) -- (5.2,1) node[below, at start]{$v$}
      node[midway,fill,inner sep=2.5pt,circle]{}; \node at (5.75,0)
      {$-$};

      \draw (6.7,-1) -- (6.7,1) node[below, at start]{$u$}
      node[midway,fill,inner sep=2.5pt,circle]{}; \draw[dashed]
      (6.2,-1)-- (6.2,1); \draw (7.2,-1) -- (7.2,1) node[below, at
      start]{$v$};
\node at (7.75,0)
      {$+h$};
  \draw (8.7,-1) -- (8.7,1) node[below, at start]{$u$}; \draw[dashed]
      (8.2,-1)-- (8.2,1); \draw (9.2,-1) -- (9.2,1) node[below, at
      start]{$v$};
    \end{tikzpicture} & \text{for $u= qv$}
 \end{cases}
  \end{equation*}  
  \begin{equation*}\subeqn\label{triple-boring}
    \begin{tikzpicture}[very thick,scale=1 ,scale=.8,baseline]
      \draw (-3,0) +(1,-1) -- +(-1,1) node[below,at start]{$w$}; \draw
      (-3,0) +(-1,-1) -- +(1,1) node[below,at start]{$u$}; \draw
      (-3,0) +(0,-1) .. controls (-4,0) ..  +(0,1) node[below,at
      start]{$v$}; \node at (-1,0) {=}; \draw (1,0) +(1,-1) -- +(-1,1)
      node[below,at start]{$w$}; \draw (1,0) +(-1,-1) -- +(1,1)
      node[below,at start]{$u$}; \draw (1,0) +(0,-1) .. controls
      (2,0) ..  +(0,1) node[below,at start]{$v$};
    \end{tikzpicture}
  \end{equation*}
  \begin{equation*}\subeqn \label{eq:triple-point1}
    \begin{tikzpicture}[very thick,xscale=1.1,yscale=.8,baseline]
      \draw[dashed] (-3,0) +(.4,-1) -- +(-.4,1); \draw[dashed] (-3,0)
      +(-.4,-1) -- +(.4,1); \draw (-1.5,0) +(.4,-1) -- +(-.4,1)
      node[below,at start]{$v$}; \draw (-1.5,0) +(-.4,-1) -- +(.4,1)
      node[below,at start]{$w$}; \draw (-3,0) +(0,-1) .. controls
      (-3.5,0) ..  +(0,1) node[below,at start]{$u$};\node at (-.75,0)
      {$-$}; \draw[dashed] (0,0) +(.4,-1) -- +(-.4,1); \draw[dashed]
      (0,0) +(-.4,-1) -- +(.4,1); \draw (1.5,0) +(.4,-1) -- +(-.4,1)
      node[below,at start]{$v$}; \draw (1.5,0) +(-.4,-1) -- +(.4,1)
      node[below,at start]{$w$}; \draw (0,0) +(0,-1) .. controls
      (.5,0) ..  +(0,1) node[below,at start]{$u$}; \node at (2.25,0)
      {$=$}; 
    \end{tikzpicture}
    \begin{cases}
      \begin{tikzpicture}[very thick,xscale=1.1,yscale=.8,baseline]
        \draw (4.5,0) +(.4,-1) -- +(.4,1) node[below,at
      start]{$v$}; \draw (4.5,0) +(-.4,-1) -- +(-.4,1) node[below,at
      start]{$w$}; \draw[dashed] (3,0) +(.4,-1) -- +(.4,1);
      \draw[dashed] (3,0) +(-.4,-1) -- +(-.4,1); \draw (3,0) +(0,-1)
      -- +(0,1) node[below,at start]{$u$};
      \end{tikzpicture}& \text{if $v=w=qu$}\\
      0 &  \text{unless $v=w=qu$}
    \end{cases}
  \end{equation*}
  \begin{equation*}\subeqn\label{eq:KLRtriple-point2}
    \begin{tikzpicture}[very thick,xscale=1.1,yscale=.8,baseline]
      \draw[dashed] (-3,0) +(0,-1) .. controls (-3.5,0) ..  +(0,1) ;
      \draw (-3,0) +(.4,-1) -- +(-.4,1) node[below,at start]{$v$};
      \draw (-3,0) +(-.4,-1) -- +(.4,1) node[below,at start]{$u$};
      \draw (-1.5,0) +(0,-1) .. controls (-2,0) ..  +(0,1)
      node[below,at start]{$w$};\node at (-.75,0) {$-$}; \draw (0,0)
      +(.4,-1) -- +(-.4,1) node[below,at start]{$v$}; \draw (0,0)
      +(-.4,-1) -- +(.4,1) node[below,at start]{$u$}; \draw[dashed]
      (0,0) +(0,-1) .. controls (.5,0) ..  +(0,1); \draw (1.5,0)
      +(0,-1) .. controls (2,0) ..  +(0,1) node[below,at
      start]{$w$}; \node at (2.25,0) {$=$};     \end{tikzpicture}
    \begin{cases}
      \begin{tikzpicture}[very thick,xscale=1.1,yscale=.8,baseline]\draw (3,0) +(.4,-1) --
      +(.4,1) node[below,at start]{$v$}; \draw (3,0) +(-.4,-1) --
      +(-.4,1) node[below,at start]{$u$}; \draw[dashed] (3,0) +(0,-1)
      -- +(0,1);\draw (4.5,0) +(0,-1) -- +(0,1) node[below,at
      start]{$w$};
      \end{tikzpicture}& \text{if $w=qu=qv$}\\
      0 &  \text{unless $w=qu=qv$}
    \end{cases}.
  \end{equation*}
\end{definition}
For the sake of completeness, here is an example of a weighted KLR
diagram:
\[
\begin{tikzpicture}[baseline,very thick,xscale=2.5,yscale=1.4]
  \draw (-.5,-1) to[out=90,in=-90] node[below,at start]{$u_2$}  (-1,0) to[out=90,in=-90](.5,1);
  \draw[dashed] (-1.25,-1) to[out=90,in=-90] (-1.75,0) to[out=90,in=-90](-.25,1);
  \draw (.5,-1) to[out=90,in=-90]node[below,at start]{$u_4$} (1,1);
    \draw[dashed] (-.25,-1) to[out=90,in=-90] (.25,1);
  \draw  (1,-1) to[out=90,in=-90] node[below,at start]{$u_5$}
  node[circle, midway,fill=black,inner sep=3pt]{} (0,1);
    \draw[dashed]  (.25,-1) to[out=90,in=-90] (-.75,1);
  \draw (-1, -1) to[out=90,in=-90] node[below,at start]{$u_1$} (-.5,0) to[out=90,in=-90] (-1,1);
  \draw (0,-1) to[out=90,in=-90] node[below,at start]{$u_3$} (-.5,1);
  \draw[dashed] (-1.75, -1) to[out=90,in=-90] (-1.25,0) to[out=90,in=-90] (-1.75,1);
  \draw[dashed] (-.75,-1) to[out=90,in=-90] (-1.25,1);
\end{tikzpicture}\]
We can define a {\bf degree} function on KL diagrams, a special case
of the degree function in \cite{WebwKLR}.  The degrees are
given on elementary diagrams by 
\begin{equation}
  \deg\tikz[baseline,very thick,scale=1.5]{\draw (.2,.3) -- (-.2,-.1)
    node[at end,below,
    scale=.8]{$u$}; \draw (.2,-.1) -- (-.2,.3) node[at
    start,below,scale=.8]{$v$};}
  =-2 \qquad
  \deg\tikz[baseline,very thick,scale=1.5]{\draw (0,.3) -- (0,-.1)
    node[at
    end,below,scale=.8]{$u$} node[midway,circle,fill=black,inner
    sep=2pt]{};}=2
\qquad   \deg\tikz[baseline,very thick,scale=1.5]{\draw (.2,.3) -- (-.2,-.1)
    node[at
    end,below,scale=.8]{$u$}; \draw[dashed] (.2,-.1) -- (-.2,.3)
    node[at
    start,below,scale=.8]{$v$};} = 
\deg\tikz[baseline,very
  thick,scale=1.5]{\draw[dashed] (.2,.3) -- (-.2,-.1) node[at
    end,below,scale=.8]{$u$}; \draw (.2,-.1) -- (-.2,.3) node[at
    start,below,scale=.8]{$v$};}
  =
\begin{cases}
    2 & u=qv=q^{-1}v\\
    1 & u=q^{\pm 1}v\neq q^{\mp 1}v\\
    0 & u\neq q^{\pm 1}v
  \end{cases}\label{wKLR-grading}
\end{equation}
and $h$ is given grading $2$. Note that the relations
(\ref{dots-1}--\ref{eq:KLRtriple-point2}) are all homogeneous with wKLR diagrams given the
grading of \eqref{wKLR-grading}.
\begin{proposition}[\mbox{\cite[Prop. 2.7]{WebwKLR}}]\label{W-KLR-poly}
  The wKLR algebra $\gls{W}_{\mathscr{D}}(\pq)$ for a collection
  $\mathscr{D}$ has a faithful polynomial representation
  \[P_{\mathscr{D}} := \oplus_{D\in \mathscr{D}}\K[h,y_1,\dots,
  y_{|D|}]\]
defined by the rule that 
\begin{itemize}
\item Each crossing of the $r$ and $r+1$st strands acts by the Demazure operator \[\partial_r(f)=
  \frac{f^{s_r}-f}{y_{r+1}-y_{r}}.\]
\item 
A crossing between the $r$th strand and a ghost of $s$th strand acts
by
\begin{itemize}
\item the identity if $\ck <0$ and the
  strand is NE/SW or $\ck >0$ and the strand is NW/SE,
\item the multiplication operator of  $y_s-y_r+h$ if $\ck <0$ and
  the strand is NW/SE or $\ck >0$ and the strand is NE/SW
\end{itemize}
\item A square on the $r$th strand acts by the multiplication operator
  $Y_r$.  
\end{itemize}
\end{proposition}

Thus, we can again apply the results of Section \ref{sec:polyn-style-repr},
with 
\begin{equation}
\mathbb{K}=\K[h]\qquad A=\gls{W}_{\mathscr{D}}(\pq) \qquad
  B=\K[h,y]\qquad I=B(h,y)\qquad P=P_{\mathscr{B}}.\label{eq:wKLR-PR}
\end{equation}

\begin{lemma}
  The polynomial representation $P_{\mathscr{D}}$ is graded polynomial-style
  with the data of (\ref{eq:wKLR-PR}).  
\end{lemma}
\begin{proof}
  The algebra $\gls{W}_{\mathscr{D}}(\pq) $ is free over $B^{\otimes
    n}=\K[h,y_1,\dots, y_n]$ by \cite[Thm. 2.8]{WebwKLR}, the
  centrality of $Z$ is clear from the relations and the faithfulness
  of $P$ follows from Proposition \ref{W-KLR-poly}.  The compatibility
  with grading is also clear from the definition
  (\ref{wKLR-grading}).  
\end{proof}

\begin{definition}\label{def:hW}
We let $\gls{hW}(h)$ be the completion of the weighted KLR algebra
$\gls{W}$ for $\gls{U}$
with respect to the grading; since $h$ has degree 2, this completion
is naturally a complete $\K[[h]]$-module.  For any collection $\mathscr{D}$, we let
$\gls{W}_{\mathscr{D}}(\pq),\gls{hW}_{\mathscr{D}}(\pq)$ be the sum of images of
the idempotents corresponding to loadings on a set of points in $\mathscr{D}$.
\end{definition}  

Let $\Bi$ be a loading in
the sense of \cite{WebwKLR}, that is, a finite subset $D=\{d_1,\dots, d_n\}$ with
$d_1<\cdots <d_n$ of $\R$
together with a map $\Bi\colon D\to \gls{U}$.
In the algebra $\widehat{\waha}_{\mathscr{D}}(\pq)$, we have an idempotent
$\epsilon_{\Bi}$ projecting to the stable kernel of $X_j-\Bi(d_j)$
(that is, the kernel of a sufficiently large power). 
We represent $\epsilon_{\Bi}$ as a type W
diagram, with the strands labeled by the elements
$u_i=\Bi(d_j)$.  

\begin{theorem}\label{W-isomorphism}
  There is an isomorphism $\gamma\colon
  \widehat{\waha}_{\mathscr{D}}(\pq)\to \gls{hW}_{\mathscr{D}}(\pq)$ such
  that $\gamma(X_r)=\sum_{\Bu}u_rb({y_r})e_{\Bu}$,
\newseq
\[\subeqn\label{crossing-match}
\tikz[baseline,very thick,scale=1.5, green!50!black]{\draw (.2,.3) --
  (-.2,-.1); \draw
  (.2,-.1) -- (-.2,.3);} \epsilon_{\Bu}\mapsto
\begin{cases}
\displaystyle\frac{1}{u_{r+1}b(y_{r+1})-u_rb(y_{r})}(\psi_r-1) e_{\Bu} & u_r\neq u_{r+1}\\
 \displaystyle \frac{y_{r+1}-y_r}{u_{r+1}(b(y_{r+1})-b(y_{r}))}\psi_r e_{\Bu}& u_r=u_{r+1}
\end{cases}
\]
\[\subeqn\label{ghost-match}
\tikz[baseline,very thick,scale=1.5, green!50!black]{\draw[densely dashed] 
  (-.2,-.1)-- (.2,.3); \draw
  (.2,-.1) -- (-.2,.3);}\epsilon_{\Bu} \mapsto
\begin{cases}
\displaystyle u_rb(y_r)-\pq u_sb(y_s)\tikz[baseline,very thick,scale=1.5]{\draw[densely dashed]
    (-.2,-.1)-- (.2,.3); \draw
    (.2,-.1) -- (-.2,.3);} e_{\Bu}&
  u_r\neq qu_s\\ 
 \displaystyle \frac {u_rb(y_r)-\pq u_sb(y_s)}{y_{s}-y_{r}+d_1h}\tikz[baseline,very thick,scale=1.5]{\draw[densely dashed]
    (-.2,-.1)-- (.2,.3); \draw
    (.2,-.1) -- (-.2,.3);}e_{\Bu}& u_r=qu_s
\end{cases}
\qquad \qquad\tikz[baseline,very thick,scale=1.5, green!50!black]{\draw (.2,.3) --
  (-.2,-.1); \draw [densely dashed]
  (.2,-.1) -- (-.2,.3);} \mapsto \tikz[baseline,very thick,scale=1.5]{\draw (.2,.3) --
  (-.2,-.1); \draw [densely dashed]
  (.2,-.1) -- (-.2,.3);} \]
\end{theorem}
\begin{proof}
  This follows from comparing the polynomial representations.  Exactly
  as argued in Lemma \ref{lem:gammap}, the map is an isomorphism of
  vector spaces between the polynomial representations: the polynomial
  representation ${\cP}_{\mathscr{B}}$ has one copy of $\mathcal{C}$
  for each subset in $\mathscr{B}$.  In $\widehat{\cP}_{\mathscr{B}}$,
  each of these copies is completed at $\gls{U}^n$, and becomes the direct
  sum of the images of $e_{\Bu}$, which is a copy of the completed
  polynomial ring.  We can think of the choice of subset and of $\Bu$
  as giving a loading, which has a corresponding copy of $\widehat{C}$
  in $P_{\mathscr{B}}$.  The map $\gamma_p$ induces an isomorphism
  between these completed polynomial rings.

 Now,  we should
 consider how identifying completed polynomial representations via
 $\gamma_p$ affects how the basic diagrams of the WAHA act on the polynomial
 representation.  

 We have that \[\tikz[baseline,very thick,scale=1.5, green!50!black]{\draw (.2,.3) --
  (-.2,-.1); \draw
  (.2,-.1) -- (-.2,.3);} \cdot
f\epsilon_{\Bu}=\frac{f^{s_r}-f}{y_{r+1}-y_r}\epsilon_{\Bu}\]

If $u_r\neq u_{r+1}$, then $\psi_r\cdot
f\epsilon_\Bu=f^{s_r}\epsilon_\Bu$ and $u_{r+1}b(y_{r+1})-u_rb(y_r)$
is invertible, so the appropriate case of (\ref{crossing-match}) holds.  If $u_r= u_{r+1}$, then
$\frac{u_{r}(b(y_{r+1})-b(y_r))}{y_{r+1}-y_r}$  is invertible, so the
formula is clear.  

Now, we turn to (\ref{ghost-match}).  We find that $\tikz[baseline,very thick,scale=1.5, green!50!black]{\draw[densely dashed] 
  (-.2,-.1)-- (.2,.3); \draw
  (.2,-.1) -- (-.2,.3);}\cdot f\epsilon_{\Bu}=(u_rb(y_r)-\pq u_sb(y_s) )f\epsilon_{\Bu}$. The first case of the isomorphism (\ref{ghost-match}) thus
follows directly from the polynomial representation of the wKLR algebra
given in Proposition \ref{W-KLR-poly}. The second case of
(\ref{ghost-match}) is clear.
\end{proof}

  The reader will note that the image of the idempotent $\gls{eprime}$ under
  this isomorphism is not homogeneous.  On abstract grounds, there
  must exist a homogeneous idempotent $e''$ with isomorphic image.  Let
  us give a description of one such, which is philosophically quite
  close to the approach of \cite{SWschur}.  

Choose an arbitrary order
  on the elements of $\gls{U}$.  The idempotent $e_\mu'$ for a composition
  $\mu$ is replaced by the sum of contributions from a list of
  multi-subsets $Z_i$ of $\gls{U}$ such that $|Z_i|=\mu_i$.  There's a
  loading corresponding to these subsets, which we'll denote
  $\Bi_{Z_*}$.  The underlying subset is $C_{\mu}$ as defined before;
  the points associated to the $j$th part at $x=js+\epsilon,\dots,
  js+\mu_j\epsilon$ are labeled with the elements of $Z_j$ in our
  fixed order.  Finally, $e''_{Z_*}$ is the idempotent on this loading
  that acts on each group of strands with the same label in $\gls{U}$ and attached
  to the same part of $\mu$ with a fixed homogeneous primitive
  idempotent in the nilHecke algebra, for example, that acts as
$y^{k-1}_1\cdots y_{k-1}\partial_{w_0}$ in the polynomial
representation. Consider the sum $e''$ of the idempotents $e''_{Z_*}$ over all
$p$-tuples of multi-subsets.  

The idempotent $e''$ has
isomorphic image to $\gls{eprime}$, since $W\gls{eprime}'$ is a sum of projectives for
each composition $\mu$ whose
$(\mu_1!\cdots\mu_p!)$-fold direct sum is $We_{C_\mu}$.
Thus, the algebra $e''We''$ is graded and isomorphic to the Schur algebra.
It would be interesting to make this isomorphism a bit more
explicit, but we will leave that to other work.  

\section{Type F} 

\subsection{Type F Hecke algebras}
\label{sec:type-f-hecke}

Now let us turn to our other complication, analogous to that which
appeared in 
\cite{Webmerged}:
\begin{definition}
  A rank $n$ {\bf type F$_1$ Hecke diagram} is a rank $n$ affine Hecke diagram
  with a vertical red line inserted at $x=0$. The diagram must avoid
  tangencies and triple points with this strand as well, and only
  allow isotopies that preserve these conditions.
\end{definition}
We give an example of such a diagram below: 
\[
\begin{tikzpicture}[baseline,very thick,green!50!black, xscale=2,yscale=1.3]
  \draw (-.5,-1) to[out=90,in=-90] (-1,0) to[out=90,in=-90](.5,1);
  \draw (.5,-1) to[out=90,in=-90] (1,1);
  \draw[wei] (.25,-1) --(.25,1);
  \draw  (1,-1) to[out=90,in=-90] 
  node[midway,fill=green!50!black,inner sep=3pt]{} (0,1);
  \draw (-1, -1) to[out=90,in=-90] (-.5,0) to[out=90,in=-90] (-1,1);
  \draw (0,-1) to[out=90,in=-90] (-.5,1);
\end{tikzpicture}\]

We decorate this red strand with a multisubset
$Q_\bullet=\{Q_1,\dots, Q_\ell\}\subset \gls{U}$ and let
$\PQ_i=Q_ie^{-z_i}$.  To distinguish from other uses of the letter, we
let $\mathsf{e}_k(\Bz)$ be the degree $k$ elementary symmetric
function in an alphabet $\Bz$.
\begin{definition}
  Let the {\bf type F$_1$ affine Hecke algebra} $\tilde{\EuScript{F}} (\pq,\PQ_{\bullet})$ be the algebra generated over
  $\K[[h,\Bz]]$ by type F$_1$ Hecke diagrams with $m$ strands modulo the
  local relations (\ref{qHecke-1}--\ref{qHecke-triple})  and  the
  local relations: \newseq
      \begin{equation*}\subeqn\label{qdumb}
      \begin{tikzpicture}[very thick,baseline=2.85cm,scale=.8]
        \draw[wei] (-3,3) +(1,-1) -- +(-1,1); \draw [green!50!black] (-3,3) +(0,-1)
        .. controls (-4,3) ..  +(0,1); \draw [green!50!black] (-3,3) +(-1,-1) --
        +(1,1); \node at (-1,3) {=}; \draw[wei] (1,3) +(1,-1) --
        +(-1,1); \draw [green!50!black] (1,3) +(0,-1) .. controls (2,3) ..  +(0,1);
        \draw [green!50!black] (1,3) +(-1,-1) -- +(1,1); \end{tikzpicture}
    \end{equation*}
    \begin{equation*}\subeqn\label{qred-dot}
      \begin{tikzpicture}[very thick,baseline,scale=.8]
        \draw [green!50!black](-3,0) +(-1,-1) -- +(1,1); \draw[wei](-3,0) +(1,-1) --
        +(-1,1); \node[ fill=green!50!black,inner
      sep=2.5pt] at (-3.5,-.5){}; \node at (-1,0) {=};
        \draw [green!50!black](1,0) +(-1,-1) -- +(1,1); \draw[wei](1,0) +(1,-1) --
        +(-1,1); \node[ fill=green!50!black,inner
      sep=2.5pt] at (1.5,.5){};
      \end{tikzpicture}
    \end{equation*}
  \begin{equation*}\label{qHcost}\subeqn
    \begin{tikzpicture}[very thick,baseline,scale=.7]
      \draw [wei] (-1.8,0) +(0,-1) -- +(0,1);
      \draw[green!50!black](-1.2,0) +(0,-1) .. controls (-2.8,0) ..
      +(0,1); \node at (-.3,0) {=}; \draw [green!50!black] (2.2,0)
      +(0,-1) -- node[midway, fill=green!50!black,inner
      sep=2.5pt,label=right:{$\ell$}]{}+(0,1); \draw[wei] (1.2,0)
      +(0,-1) -- +(0,1); \node at (4.3,0) {$+\,\mathsf{e}_1(-\PQ_\bullet)$};
      \draw[green!50!black] (6.8,0) +(0,-1) -- node[midway,
      fill=green!50!black,inner
      sep=2.5pt,label=right:{$\ell-1$}]{}+(0,1); \draw [wei] (5.8,0)
      +(0,-1) -- +(0,1); \node at (8.8,0) {$+$}; \node at (9.6,-.07)
      {$\cdots$}; \node at (11.35,0) {$+\mathsf{e}_{\ell}(-\PQ_\bullet)$};
      \draw[green!50!black] (13.8,0) +(0,-1) -- +(0,1); \draw [wei]
      (12.8,0) +(0,-1) -- +(0,1);
    \end{tikzpicture}
  \end{equation*}
  That is, on the RHS, we have the product $p_{\PQ}=(X_j-\PQ_1)\cdots
  (X_j-\PQ_\ell)$, where the green strand shown is the $j$th, and
  \begin{equation*}\label{qred-triple}\subeqn
    \begin{tikzpicture}[very thick,baseline=-2pt,scale=.7]
      \draw [wei] (0,-1) -- (0,1); \draw[green!50!black](.5,-1)
      to[out=90,in=-30] (-.5,1); \draw[green!50!black](-.5,-1)
      to[out=30,in=-90] (.5,1);
    \end{tikzpicture}- \begin{tikzpicture}[very
      thick,baseline=-2pt,scale=.7] \draw [wei] (0,-1) -- (0,1);
      \draw[green!50!black](.5,-1) to[out=150,in=-90] (-.5,1);
      \draw[green!50!black](-.5,-1) to[out=90,in=-150] (.5,1);
    \end{tikzpicture}
    =\sum_{i=1}^\ell\sum_{a+b=i-1}\mathsf{e}_{\ell-i}(-\PQ_\bullet)\cdot \Bigg(\begin{tikzpicture}[very thick,baseline=-2pt,scale=.7]
      \draw [wei]  (0,-1) -- (0,1);
      \draw[green!50!black](.5,-1) to[out=90,in=-90] node[midway,
      fill=green!50!black,inner sep=2.5pt,label=right:{$b$}]{} (.5,1);
      \draw[green!50!black](-.5,-1) to[out=90,in=-90] node[midway,
      fill=green!50!black,inner sep=2.5pt,label=left:{$a+1$}]{} (-.5,1);
    \end{tikzpicture}-\pq \begin{tikzpicture}[very
      thick,baseline=-2pt,scale=.7] \draw [wei] (0,-1) -- (0,1);
      \draw[green!50!black](.5,-1) to[out=90,in=-90] node[midway,
      fill=green!50!black,inner sep=2.5pt,label=right:{$b+1$}]{}
      (.5,1); \draw[green!50!black](-.5,-1) to[out=90,in=-90]
      node[midway, fill=green!50!black,inner
      sep=2.5pt,label=left:{$a$}]{} (-.5,1);
    \end{tikzpicture}\Bigg).
  \end{equation*}
  The RHS can alternately by written as $(X_i-\pq
  X_{i+1})\frac{p_{\PQ}(X_i)-p_\PQ(X_{i+1})}{X_i-X_{i+1}}$.
\end{definition}
\begin{remark}
  As in the earlier cases, there is a degenerate version of this
  algebra, where we use the local relations
  (\ref{deg-Hecke-1}--\ref{deg-Hecke-triple}), leave (\ref{qdumb}--\ref{qHcost})
  unchanged, and replace the RHS of (\ref{qred-triple}) with $(X_i-
  X_{i+1}+1)\frac{p_{\PQ}(X_i)-p_\PQ(X_{i+1})}{X_i-X_{i+1}}$.
\end{remark}
We'll continue to use our convention of letting $X_r$ denote the sum
of all straight-line diagrams with a square on the $r$th green strand from
the left (ignoring red strands).

Given ${\mathscr{D}}$ a collection of subsets of $\R$, we'll let $\tilde{\EuScript{F}}_{\mathscr{D}}
(\pq,\PQ_{\bullet}),{\EuScript{F}}_{\mathscr{D}} (\pq,\PQ_{\bullet})$
denote the subalgebras of $\tilde{\EuScript{F}}
(\pq,\PQ_{\bullet}),{\EuScript{F}} (\pq,\PQ_{\bullet})$ spanned by
diagrams whose tops and bottoms lie in the set ${\mathscr{D}}$.

  Let $e_i$ be an arbitrarily fixed idempotent in $\tilde{\EuScript{F}} (\pq,\PQ_{\bullet})$ given by
  $i$ strands left of the red strand and $m-i$ right of it; let
  $\mathscr{D}^\circ$ be the collection of the corresponding
  sets. Since any idempotent is isomorphic to one of these by a
  straight-line diagram, enlarging $\mathscr{D}^\circ$ will give a
  Morita equivalent algebra.  Let
  $\tilde{P}_n$ be the free $S[X_1^{\pm 1},\dots, X_n^{\pm 1}]$-module generated
  by elements $f_p$ for $p=0,\dots, m$.
\begin{proposition}\label{tilde-poly}
  The algebra $\tilde{\EuScript{F}}_{\mathscr{D}^\circ} (\pq,\PQ_{\bullet})$ has a polynomial
  representation that sends
  \begin{itemize}
\item    $e_i$ to the identity on the submodule generated by $f_i$.
\item $X_i$ to the multiplication operator and \[(T_i+1)\cdot
  F(X_1,\dots,X_n)f_p\mapsto (X_i-\pq X_{i+1})\frac{F^{s_i}-F}{X_{i+1}-X_i}f_p.\]
 \item the action of positive to negative crossing to the identity
   \[F(X_1,\dots,X_n)f_i\mapsto
   F(X_1,\dots,X_n)f_{i+1},\] and the opposite crossing to
 \[F(X_1,\dots,X_n)f_i\mapsto
   p_\PQ(X_i)F(X_1,\dots,X_n)f_{i-1}.\]
  \end{itemize}
\end{proposition}
\begin{proof}
  This is a standard computation with Demazure operators.
\end{proof}

Now, we can allow several red
lines at various values of $x$, each of which carries a multiset of
values in $\gls{U}$.  For the sake of notation, we'll still denote the
multiset given by all such labels as $\{Q_1,\dots, Q_\ell\}$, with a strand with the label
$Q_{i}$ at  $x$-value
$\vartheta_i$.  So, the situation we had previously considered was
$\vartheta_i=0$ for all $i$.

\begin{definition}\label{def:tF}
  A  rank $n$ {\bf type F Hecke diagram} is a rank $n$ affine Hecke diagram
  with a vertical red lines inserted at $x=\vartheta_i$. The diagram must avoid
  tangencies and triple points with these strands as well, and only
  allow isotopies that preserve these conditions.  We give an example
  of such a diagram below:
  \[
\begin{tikzpicture}[baseline,very thick,green!50!black, xscale=2,yscale=1.3]
  \draw (-.5,-1) to[out=90,in=-90] (-1,0) to[out=90,in=-90](.5,1);
  \draw (.5,-1) to[out=90,in=-90] (1,1);
  \draw  (1,-1) to[out=90,in=-90] 
  node[midway,fill=green!50!black,inner sep=3pt]{} (0,1);
  \draw (-1, -1) to[out=90,in=-90] (-.5,0) to[out=90,in=-90] (-1,1);
  \draw (0,-1) to[out=90,in=-90] (-.5,1);
  \draw[wei](.25,-1) --(.25,1); \draw[wei](-.25,-1) --(-.25,1);
    \draw[wei](1.25,-1) --(1 .25,1);
\end{tikzpicture}\]

  Let the rank $n$ {\bf type F affine Hecke algebra} $\gls{tF}$ be the algebra generated over $\K[[h,\Bz]]$ by
  rank $n$  type F Hecke diagrams for $\vartheta$ with $n$ strands modulo the
  local relations (\ref{qHecke-1}--\ref{qHecke-triple}) and
  (\ref{qdumb}--\ref{qred-triple})).
\end{definition}
These algebras have a polynomial representation $\cP^{\vartheta}$ using the same maps
attached to basic diagrams as Proposition \ref{tilde-poly}, but now
with idempotents, and thus copies of Laurent polynomials, indexed by weakly increasing
functions $\nu\colon [1,\ell]\to [0,m]$ with $\nu(i)$ giving the number of green strands to
the left of the $i$th red strand.   This was carried out in more
detail in \cite[Prop. 1.10]{MS18}.  As before, any two idempotents
corresponding to $\nu$ are isomorphic by straight-line diagrams.

These affine type F algebras have ``finite-type'' quotients.  In other
contexts, these have been called ``steadied'' or ``cyclotomic'' quotients.
\begin{definition}\label{def:cF}
  The rank $n$ {\bf type F Hecke algebra} $\gls{cF}$ is the quotient of $\gls{tF}$ by the 2-sided ideal generated by $e_{B}$ for
  every set $B$ possessing an element $b\in B$ with $b<\vartheta_i$
  for all $i$.  
\end{definition} 
Pictorially, the idempotents $e_B$ we kill possess a green strand
which is left of all the red strands.  In \cite{Webmerged}, the corresponding ideal for KLR
algebras is called the {\bf violating ideal} and we will use the same
terminology here. 
Given ${\mathscr{D}}$ a collection of subsets of $\R$, we'll let $\tilde{\EuScript{F}}^\vartheta_{\mathscr{D}}
(\pq,\PQ_{\bullet}),{\EuScript{F}}^\vartheta_{\mathscr{D}} (\pq,\PQ_{\bullet})$
denote the subalgebras of $\gls{tF},{\EuScript{F}}^\vartheta (\pq,\PQ_{\bullet})$ spanned by
diagrams whose tops and bottoms lie in the set ${\mathscr{D}}$.

\begin{proposition}\label{prop:cyclo-Hecke}
  The rank $n$  cyclotomic affine Hecke  algebra $\gls{cmH}(\pq,\PQ_\bullet)$ for the parameters $\{\PQ_1,\dots,
  \PQ_\ell\}$ is isomorphic to the rank $n$  type F$_1$ Hecke algebra
  $\EuScript{F}_{\mathscr{D}^\circ} (\pq,\PQ_{\bullet})$.
\end{proposition}
\begin{proof}
  If we let $e$ be the idempotent given by green lines at $x=1,\dots, 
 n$, then we see by Theorem \ref{wdHecke}, there is a map from the
  affine Hecke algebra sending $X_i$ and $T_i+1$ to diagrams as in
  (\ref{Hecke-gens}) which induces a map
  $\iota\colon \gls{mH}(\pq)\to \EuScript{\tilde F}_{\mathscr{D}^\circ}
  (\pq,\PQ_{\bullet})$. 
Pulling back the polynomial representation of $\EuScript{\tilde F}_{\mathscr{D}^\circ}
  (\pq,\PQ_{\bullet})$ gives the polynomial representation of
  $\gls{mH}(\pq)$, which is faithful, so this map is injective.

  Applying 
  (\ref{qHcost}) at the leftmost strand shows that $p_\PQ(X_1)$ lies
  in the violating ideal, which is the kernel of the map to
  $\EuScript{F}(\pq,\PQ_{\bullet})$.  Thus, $\iota$ induces a
  map $\gls{cmH}(\pq,\PQ_\bullet)\to \EuScript{F}_{\mathscr{D}^\circ}
  (\pq,\PQ_{\bullet}). $  This map is clearly surjective, since any
  F$_1$ Hecke diagram with no violating strand is a composition of the
  images.
  
 Thus, we need only show that the preimage of
  the violating ideal under $\iota$ lies in the cyclotomic ideal.  As in the proof of
  \cite[3.16]{Webmerged}, the relations (\ref{qHcost},\ref{qred-triple})
  allow us to reduce to the case where only a single green strand
  passes into the left half of the plane.  In this case, we gain a
  factor of $p_\PQ(X_1)$, showing that this is in the cyclotomic
  ideal. 
\end{proof}

\subsection{Stendhal algebras}
\label{sec:stendhal-algebras}

The type F algebras in the KLR family have been introduced in
\cite{Webmerged}.  
Let $o_1=\min(\vartheta_i)$, and
$o_j=\min_{\vartheta_i>o_{j-1}}(\vartheta_i)$; so these are the real
numbers that occur as $\vartheta_i$ in increasing order.
Consider the sequence $\la_j=\sum_{\vartheta_i=o_j}\omega_{Q_i}$ of
dominant weights for $\mathfrak{g}_{\gls{U}}$, and let $S_{u,j}=\{s\in
[1,\ell]|\vartheta_s=o_j, u=Q_s\}$.

In
\cite[Def. 4.7]{Webmerged}, we defined algebras
$T^\bla,\tilde{T}^\bla$ attached to this list of weights.  These
cannot match $\gls{htF},\gls{cF}$ since they are not naturally modules over
  $\K[[h,\Bz]]$; however, we will recover them when we set
  $h=z_1=\cdots=z_\ell=0$.  Instead, we should consider deformed
  versions of these algebras $\tilde{T}^\bla(h,\Bz),T^\bla(h,\Bz)$
  introduced in \cite[\S 3.2]{WebwKLR} based on
  the canonical deformation of weighted KLR algebras.  As usual, we'll
  let $y_r$ denote the sum of all straight line Stendhal diagrams with
  a dot on the $r$th strand.
  \begin{definition}\label{def:tT}
    We let the rank $n$ affine Stendhal algebra $\gls{tT}$ be the quotient of the algebra
    freely spanned over $\K[h,\Bz]$ by  Stendhal diagrams (as defined
    in \cite[\S 3.2]{Webmerged}) with $n$
    black strands, with the local
    relations (\ref{first-QH}--\ref{triple-dumb}) and the
    local relations\newseq
    \begin{equation*}\label{cost}\subeqn
      \begin{tikzpicture}[very thick,scale=.8,baseline=1.2cm]
        \draw (-2.8,0) +(0,-1) .. controls (-1.2,0) ..  +(0,1)
        node[below,at start]{$u$}; \draw[wei] (-1.2,0) +(0,-1)
        .. controls (-2.8,0) ..  +(0,1) node[below,at
        start]{$\la_j$}; \node at (-.3,0) {$=p_{u,j}$};
        \node[scale=1.5] at (.5,0) {$\Bigg($}; \node[scale=1.5] at
        (3.5,0) {$\Bigg)$}; \draw[wei] (2.8,0) +(0,-1) -- +(0,1)
        node[below,at start]{$\la_j$}; \draw (1.2,0) +(0,-1) -- +(0,1)
        node[below,at start]{$u$}; \fill (1.2,0) circle (3pt);
        \draw[wei] (-2.8,3) +(0,-1) .. controls (-1.2,3) ..  +(0,1)
        node[below,at start]{$\la_j$}; \draw (-1.2,3) +(0,-1)
        .. controls (-2.8,3) ..  +(0,1) node[below,at start]{$u$};
        \node at (-.3,3) {$=p_{u,j}$};\node[scale=1.5] at (.5,3)
        {$\Bigg($}; \node[scale=1.5] at (3.5,3) {$\Bigg)$}; \draw
        (2.8,3) +(0,-1) -- +(0,1) node[below,at start]{$u$};
        \draw[wei] (1.2,3) +(0,-1) -- +(0,1) node[below,at
        start]{$\la_j$}; \fill (2.8,3) circle (3pt);
      \end{tikzpicture}\qquad
      p_{u,j}(y)=\prod_{s\in S_{u,j}}(y-z_{s})
    \end{equation*}
    \begin{equation*}\subeqn\label{dumb}
      \begin{tikzpicture}[very thick,baseline=2.85cm,scale=.8]
        \draw[wei] (-3,3) +(1,-1) -- +(-1,1); \draw (-3,3) +(0,-1)
        .. controls (-4,3) ..  +(0,1); \draw (-3,3) +(-1,-1) --
        +(1,1); \node at (-1,3) {=}; \draw[wei] (1,3) +(1,-1) --
        +(-1,1); \draw (1,3) +(0,-1) .. controls (2,3) ..  +(0,1);
        \draw (1,3) +(-1,-1) -- +(1,1); \end{tikzpicture}
    \end{equation*}
    \begin{equation*}\subeqn\label{red-dot}
      \begin{tikzpicture}[very thick,baseline,scale=.8]
        \draw(-3,0) +(-1,-1) -- +(1,1); \draw[wei](-3,0) +(1,-1) --
        +(-1,1); \fill (-3.5,-.5) circle (3pt); \node at (-1,0) {=};
        \draw(1,0) +(-1,-1) -- +(1,1); \draw[wei](1,0) +(1,-1) --
        +(-1,1); \fill (1.5,.5) circle (3pt);
      \end{tikzpicture}
    \end{equation*}
    \begin{equation*}\label{red-triple}\subeqn
      \begin{tikzpicture}[very thick,baseline=-2pt,scale=.8]
        \draw [wei] (0,-1) -- node[below,at start]{$\la_j$} (0,1);
        \draw(.5,-1) to[out=90,in=-30] node[below,at start]{$u$}
        (-.5,1); \draw(-.5,-1) to[out=30,in=-90] node[below,at
        start]{$v$} (.5,1);
      \end{tikzpicture}- \begin{tikzpicture}[very thick,baseline=-2pt,scale=.8]
        \draw [wei] (0,-1) --node[below,at start]{$\la_j$} (0,1);
        \draw(.5,-1) to[out=150,in=-90] node[below,at start]{$u$}
        (-.5,1); \draw(-.5,-1) to[out=90,in=-150] node[below,at
        start]{$v$} (.5,1);
      \end{tikzpicture}
      =\delta_{u,v}\sum_{p=1}^{\al_u^\vee(\la_j)}\sum_{a+b=p-1}\mathsf{e}_{\al_i^\vee(\la_j)-p}\big(\{-z_s\mid
      s\in S_{u,j}\}\big) \cdot \Bigg(\begin{tikzpicture}[very thick,baseline=-2pt,scale=.8]
        \draw [wei]  (0,-1) -- (0,1);
        \draw(.5,-1) to[out=90,in=-90] node[midway,circle,
        fill=black,inner sep=2pt,label=right:{$b$}]{} (.5,1);
        \draw(-.5,-1) to[out=90,in=-90] node[midway,circle,
        fill=black,inner sep=2pt,label=left:{$a$}]{} (-.5,1);
      \end{tikzpicture}\Bigg).
    \end{equation*}
    The rank $n$  Stendhal algebra $\gls{cT}$ is the quotient of
    $\gls{tT}$ by violating diagrams as defined in
    \cite[Def. 4.3]{Webmerged}.
  \end{definition}
  Again for the sake of comparison, here is an example of a Stendhal
  diagram of rank 5: \[
\begin{tikzpicture}[baseline,very thick,xscale=2,yscale=1.3]
  \draw (-.5,-1) to[out=90,in=-90]  node[below,at start]{$u_2$} (-1,0) to[out=90,in=-90](.5,1);
  \draw (.5,-1) to[out=90,in=-90] node[below,at start]{$u_4$} (1,1);
  \draw  (1,-1) to[out=90,in=-90] node[below,at start]{$u_5$}
  node[circle, midway,fill=black,inner sep=3pt]{} (0,1);
  \draw (-1, -1) to[out=90,in=-90]node[below,at start]{$u_1$} (-.5,0) to[out=90,in=-90] (-1,1);
  \draw (0,-1) to[out=90,in=-90] node[below,at start]{$u_3$} (-.5,1);
  \draw[wei](.25,-1) -- node[below,at start]{$\la_2$} (.25,1); \draw[wei](-.25,-1) --node[below,at start]{$\la_1$} (-.25,1);
    \draw[wei](1.25,-1) --node[below,at start]{$\la_3$} (1 .25,1);
\end{tikzpicture}\] 
This algebra is graded with Stendhal diagrams given their usual
grading, summing local contributions given by 
\[
  \deg\tikz[baseline,very thick,scale=1.5]{\draw (.2,.3) --
    (-.2,-.1) node[at end,below, scale=.8]{$u$}; \draw (.2,-.1) --
    (-.2,.3) node[at start,below,scale=.8]{$v$};}
  =-\langle\al_u,\al_v\rangle \qquad \deg\tikz[baseline,very
  thick,scale=1.5]{\draw (0,.3) -- (0,-.1) node[at
    end,below,scale=.8]{$u$} node[midway,circle,fill=black,inner
    sep=2pt]{};}=2 \qquad
  \deg\tikz[baseline,very thick,scale=1.5]{\draw[wei] (.2,.3) --
    (-.2,-.1) node[at end,below, scale=.8]{$\la$}; \draw (.2,-.1) --
    (-.2,.3) node[at start,below,scale=.8]{$u$};}
  =\deg\tikz[baseline,very thick,scale=1.5]{\draw (.2,.3) --
    (-.2,-.1) node[at end,below, scale=.8]{$u$}; \draw[wei] (.2,-.1) --
    (-.2,.3) node[at start,below,scale=.8]{$\la$};}=\al_u^{\vee}(\la),\]
and the variables $h$ and $z_i$ each have degree 2.

 The algebra  $\gls{tT}$ has a polynomial representation $P_{\bla}$, given in
 \cite[Lem. 4.12]{Webmerged}.  In order to match the Hecke side,
 we will use the version of this representation that has
  \[P_{uv}(a,b)=
  \begin{cases}
    b-a+h & u=qv\\
1 & u\neq qv.
  \end{cases}
  \] 

For every loading, we have an associated function $\kappa$, with
$\kappa(k)$ equal to
the number of black strands to the left of $o_k$, and a sequence
$(u_1,\dots, u_n)$ given by the eigenvalues we've attached to each
black strand.  We let $e_{\Bu,\kappa}$ be the idempotent associated to this
data in $\gls{tT}$ and by extension in
$\gls{htT}$ and $\gls{cT}$.

\subsection{Isomorphisms}
\label{sec:isomorphisms}

As in types O and W, these algebras have polynomial-style
representations (graded in the case of $\gls{tT}$, with the data
\begin{equation*}
\mathbb{K}=\K[[h]]\qquad A=\gls{tF}\qquad
  B=\K[[h]][X^{\pm}]\qquad I=Bh+B\prod_{u\in \gls{U}}(X-u)\qquad P=\cP_{\theta}
\end{equation*}
\begin{equation*}
\mathbb{K}=\K[h]\qquad A=\gls{tT} \qquad
  B=\K[h,y]\qquad I=Bh+By \qquad P=P_{\bla}.
\end{equation*}
and the polynomial representations we have defined, with the latter
being graded.  This is proven exactly as in the earlier cases:
\begin{enumerate}
\item An explicit basis indexed by permutations, constructed for
  $\gls{tF}$ in \cite[Prop. 1.15]{MS18} and $\gls{tT} $ in
  \cite[Prop. 4.16]{Webmerged}, shows the required freeness.
  \item The
    centrality of $Z$ is immediate from the relations.
    \item The
  faithfulness of the representations is checked in
  \cite[Prop. 1.11]{MS18} and implicit in the proof of
  \cite[Prop. 4.16]{Webmerged}, respectively.
\end{enumerate}

  We let $\gls{htF}$ and 
  $\gls{htT}$ be the completions of these rings with respect to the
  induced topology.  Since $h$ and $\Bz$ both have positive
  degree, $\gls{htT}$  is a complete module over $\K[[h,\Bz]]$.

\begin{theorem}\label{F-isomorphism}
  We have an isomorphism $\gls{htF}\cong \gls{htT}$ which induces an
  isomorphism $\gls{cF}\cong \gls{cT}$, given by 
\[\epsilon_{\Bu,\kappa}\mapsto e_{\Bu,\kappa}\qquad\qquad X_r\mapsto\sum_{\Bu,\kappa}u_rb({y_r})e_{\Bu,\kappa}\qquad \qquad \tikz[baseline,very thick,scale=1.5, green!50!black]{\draw [wei] (.2,.3) --
  (-.2,-.1); \draw 
  (.2,-.1) -- (-.2,.3);} \mapsto\tikz[baseline,very
  thick,scale=1.5]{\draw [wei] (.2,.3) -- (-.2,-.1); \draw (.2,-.1) --
    (-.2,.3);}\]
 \vspace{-3mm}

\[\tikz[baseline,very thick,scale=1.5, green!50!black]{\draw [wei] (-.2,.3) --
  (.2,-.1); \draw 
  (-.2,-.1) -- (.2,.3);}  \epsilon_{\Bu,\kappa} \mapsto
\frac{\displaystyle\prod_{\vartheta_s=o_k}(u_rb({y_r-z_s})-\PQ_s)}{\displaystyle\prod_{s\in
  S_{u_r,j}} (y_r-z_s)}\, \tikz[baseline,very thick,scale=1.5]{\draw [wei] (-.2,.3) --
  (.2,-.1); \draw 
  (-.2,-.1) -- (.2,.3);} e_{\Bu,\kappa} 
\qquad \qquad \tikz[baseline,very thick,scale=1.5, green!50!black]{\draw (.2,.3) --
  (-.2,-.1); \draw
  (.2,-.1) -- (-.2,.3);}\, e_{\Bu,\kappa}\mapsto
A_r^{\Bu}\tikz[baseline,very thick,scale=1.5]{\draw (.2,.3) --
  (-.2,-.1); \draw
  (.2,-.1) -- (-.2,.3);} e_{\Bu,\kappa} \]
where the leftmost green/black strand shown is the $r$th from the
left, and the red strand shown is the $j$th from the left.
\end{theorem}
\begin{proof}
 Since all generators and relations involve at most one red line, we
  can assume that $\ell=1$, and use the representation of Proposition
  \ref{tilde-poly} for the Hecke side.  That diagrams with only green
  strands have actions that match is just Theorem \ref{O-isomorphism}.
  Thus, we only need to check the crossing of green and red strands is
  intertwined with a crossing of red and black strands.  Since we have
  only one red strand, we have that
  $\prod_{\vartheta_s=o_k}(u_re^{y_r}-\PQ_s)=p_{\PQ_\bullet}(u_re^{y_r})$.
  Thus, comparing the representation of Proposition \ref{tilde-poly} with
  the obvious $\K[h,\Bz]$-deformation of the action in
  \cite[Lem. 4.12]{Webmerged} yields the result.
\end{proof}

\section{Type WF}
\label{sec:antigens-ab}

\subsection{Type WF Hecke algebras}

Finally, we consider these two complications jointly.  As mentioned
before, these are unlikely to be familiar algebras for the reader, but 
these results will ultimately be
useful in understanding category $\cO$ of rational Cherednik algebras
in \cite{WebRou}.  
\begin{definition}\label{def:tWF}
  A rank $n$ {\bf type WF Hecke diagram}  is a type W Hecke diagram
  with vertical red lines inserted at $x=\vartheta_i$. The diagram must avoid
  tangencies and triple points between any combination of these
  strands, green strands and ghosts, and only
  allow isotopies that preserve these conditions.  An example of a rank 5 type WF diagram with $g<0$ is given below:
\[
\begin{tikzpicture}[baseline,very thick,green!50!black, xscale=2.5,yscale=1.4]
  \draw (-.5,-1) to[out=90,in=-90] (-1,0) to[out=90,in=-90](.5,1);
  \draw[dashed] (-1.25,-1) to[out=90,in=-90] (-1.75,0) to[out=90,in=-90](-.25,1);
  \draw (.5,-1) to[out=90,in=-90] (1,1);
    \draw[dashed] (-.25,-1) to[out=90,in=-90] (.25,1);
  \draw  (1,-1) to[out=90,in=-90] 
  node[midway,fill=green!50!black,inner sep=3pt]{} (0,1);
    \draw[dashed]  (.25,-1) to[out=90,in=-90] (-.75,1);
  \draw (-1, -1) to[out=90,in=-90] (-.5,0) to[out=90,in=-90] (-1,1);
  \draw (0,-1) to[out=90,in=-90] (-.5,1);
  \draw[dashed] (-1.75, -1) to[out=90,in=-90] (-1.25,0) to[out=90,in=-90] (-1.75,1);
  \draw[dashed] (-.75,-1) to[out=90,in=-90] (-1.25,1);
    \draw[wei](.125,-1) --(.125,1); \draw[wei](-.9,-1) --(-.9,1);
    \draw[wei](1.125,-1) --(1.125,1);
\end{tikzpicture}\]

Let the rank $n$ {\bf type WF affine Hecke algebra} $\gls{tWF}^\vartheta (\pq,\PQ_{\bullet})$ be the 
$\K[[h,\Bz]]$-algebra generated by
  type WF Hecke diagrams modulo the local relations
  (\ref{nilHecke-2}--\ref{eq:triple-point-2}, \ref{qdumb}--\ref{qHcost}) and \newseq
  \begin{equation*}\label{PCsmart-red-triple}\subeqn
    \begin{tikzpicture}[very thick,baseline=-2pt,scale=.7]
      \draw [wei] (0,-1) -- (0,1); \draw[green!50!black](.5,-1)
      to[out=90,in=-30] (-.5,1); \draw[green!50!black](-.5,-1)
      to[out=30,in=-90] (.5,1);
    \end{tikzpicture}- \begin{tikzpicture}[very
      thick,baseline=-2pt,scale=.7] \draw [wei] (0,-1) -- (0,1);
      \draw[green!50!black](.5,-1) to[out=150,in=-90] (-.5,1);
      \draw[green!50!black](-.5,-1) to[out=90,in=-150] (.5,1);
    \end{tikzpicture}
    =\sum_{i=1}^\ell\sum_{a+b=i-1}\mathsf{e}_{\ell-i}(-\PQ_\bullet)\cdot \Bigg(\begin{tikzpicture}[very thick,baseline=-2pt,scale=.7]
      \draw [wei]  (0,-1) -- (0,1);
      \draw[green!50!black](.5,-1) to[out=90,in=-90] node[midway,
      fill=green!50!black,inner sep=2.5pt,label=right:{$b$}]{} (.5,1);
      \draw[green!50!black](-.5,-1) to[out=90,in=-90] node[midway,
      fill=green!50!black,inner sep=2.5pt,label=left:{$a$}]{} (-.5,1);
    \end{tikzpicture}\Bigg).
  \end{equation*}
 \begin{equation*}\label{PC-dumb-red-triple}\subeqn
   \begin{tikzpicture}[very thick,baseline=-2pt,yscale=.8,]
    \draw [wei]  (0,-1) -- (0,1);
       \draw[green!50!black,dashed](.5,-1) to[out=90,in=-30] (-.5,1);
       \draw[green!50!black](-.5,-1) to[out=30,in=-90] (.5,1);
  \end{tikzpicture}= \begin{tikzpicture}[very thick,baseline=-2pt,yscale=.8,]
    \draw [wei]  (0,-1) -- (0,1);
       \draw[green!50!black,dashed](.5,-1) to[out=150,in=-90] (-.5,1);
       \draw[green!50!black](-.5,-1) to[out=90,in=-150] (.5,1);
  \end{tikzpicture}\qquad \qquad   \begin{tikzpicture}[very thick,baseline=-2pt,yscale=.8,]
    \draw [wei]  (0,-1) -- (0,1);
       \draw[green!50!black](.5,-1) to[out=90,in=-30] (-.5,1);
       \draw[green!50!black,dashed](-.5,-1) to[out=30,in=-90] (.5,1);
  \end{tikzpicture}= \begin{tikzpicture}[very thick,baseline=-2pt,yscale=.8]
    \draw [wei]  (0,-1) -- (0,1);
       \draw[green!50!black](.5,-1) to[out=150,in=-90] (-.5,1);
       \draw[green!50!black,dashed](-.5,-1) to[out=90,in=-150] (.5,1);
  \end{tikzpicture}\qquad \qquad   \begin{tikzpicture}[very thick,yscale=.8,baseline=-2pt]
    \draw [wei]  (0,-1) -- (0,1);
       \draw[green!50!black,dashed](.5,-1) to[out=90,in=-30] (-.5,1);
       \draw[green!50!black,dashed](-.5,-1) to[out=30,in=-90] (.5,1);
  \end{tikzpicture}= \begin{tikzpicture}[very thick,yscale=.8,baseline=-2pt]
    \draw [wei]  (0,-1) -- (0,1);
       \draw[green!50!black,dashed](.5,-1) to[out=150,in=-90] (-.5,1);
       \draw[green!50!black,dashed](-.5,-1) to[out=90,in=-150] (.5,1);
  \end{tikzpicture}
\end{equation*}
\end{definition}
\begin{remark}
  The degenerate version of this algebra is defined by the local relations
  (\ref{nilHecke-2}--\ref{NilHecke3}, \ref{eq:triple-point-2}--\ref{eq:deg-triple-point-1},
  \ref{qdumb}--\ref{qHcost}, \ref{PCsmart-red-triple}--\ref{PC-dumb-red-triple})
\end{remark}
Note that relation (\ref{qred-triple}) is {\it not} true in this
algebra.  As before, we should think of type F diagrams as type WF
diagrams with $\ck$ so small that we cannot see that the ghost and
strand are separate.  Using this approach, we can see that relation
(\ref{qred-triple}) for a strand and a ghost together is a consequence
of (\ref{ghost-bigon1}) and (\ref{PCsmart-red-triple}), much as in
Theorem \ref{wdHecke}.

This algebra has a polynomial representation $P^\vartheta$, defined
using the same formulae as those of Propositions \ref{W-poly} and
\ref{tilde-poly}.  We leave the routine computations that these are compatible
with (\ref{PCsmart-red-triple}) and (\ref{PC-dumb-red-triple}) to the reader.

We call an idempotent {\bf unsteady} if the strands can be
divided into two groups with a gap $>|\ck|$ between them
and all red strands in the right hand group, and {\bf steady} otherwise. 

Thus, the idempotents shown in (\ref{steady}) are steady, and those in
(\ref{unsteady}) are unsteady.
\newseq
\[\subeqn\label{steady}\tikz[baseline=-2pt,very thick, xscale=3,green!50!black] { 
\draw (.05,-.5) -- (0.05,.5);
\draw[dashed] (-.55,-.5) -- (-.55,.5);
\draw[dashed] (-.1,-.5) -- (-.1,.5);
\draw (.5,-.5) -- (.5,.5);
\draw[wei](.35,-.5) -- (.35,.5);
 }
 \hspace{3cm} \tikz[baseline=-2pt,very thick, xscale=3,green!50!black] { 
\draw (1.25,-.5) -- (1.25,.5);
\draw[dashed] (.65,-.5) -- (.65,.5);
\draw[dashed] (-.1,-.5) -- (-.1,.5);
\draw (.5,-.5) -- (.5,.5);
\draw[wei](.35,-.5) -- (.35,.5);
 }\]
\[\subeqn\label{unsteady}\tikz[baseline=-2pt,very thick, xscale=3,green!50!black] { 
\draw (-.25,-.5) -- (-0.25,.5);
\draw[dashed] (-.85,-.5) -- (-.85,.5);
\draw[dashed] (-.1,-.5) -- (-.1,.5);
\draw (.5,-.5) -- (.5,.5);
\draw[wei](.35,-.5) -- (.35,.5);
 } \hspace{3cm}\tikz[baseline=-2pt,very thick, xscale=3,green!50!black] { 
\draw (-.25,-.5) -- (-0.25,.5);
\draw[dashed] (-.85,-.5) -- (-.85,.5);
\draw[dashed] (-.4,-.5) -- (-.4,.5);
\draw (.2,-.5) -- (.2,.5);
\draw[wei](.35,-.5) -- (.35,.5);
 }\]

 \begin{definition}\label{def:cWF} 
   Let the rank $n$  {\bf type WF Hecke algebra}
   $\gls{cWF}^\vartheta (\pq,\PQ_{\bullet})\cong \PC^\vartheta$
   be the quotient of
   $\gls{tWF}^\vartheta (\pq,\PQ_{\bullet})$ by the
   ideal generated by all unsteady idempotents.
 \end{definition}

We can also call this a ``pictorial
  Cherednik algebra,'' referring to the fact that the representation
  category of this algebra when $\K=\C$ and we set $h=z_i=0$ is
  equivalent to the category
  $\cO$ over a Cherednik algebra for the group $\Z/\ell\Z\wr S_n$ for
  certain parameters.  More precisely, we consider the category $\cO$  over the
  rational Cherednik algebra $\mathsf{H}$ for the group $\Z/\ell\Z\wr
  S_n$  with arbitrary $\C$-valued parameters  $k,
  s_1,\dots, s_\ell$, using the conventions of \cite[\S 2.1]{WebRou}
  and consider the algebra
  $\gls{cWF}^\vartheta (\pq,\PQ_{\bullet})$ where fix the
  number of green strands to be $n$, and fix the parameters $g=\Re(k),
 \vartheta_p=\Re(ks_p),  q=e^{2\pi ik}$, and $Q_p=e^{2\pi i k s_p}$.
 \begin{theorem}[\mbox{\cite[Cor. 3.10]{WebRou}}]
 For the parameters discussed above, we have an equivalence of
 categories $\cO\cong
 \gls{cWF}^\vartheta (\pq,\PQ_{\bullet})\mmod$.  
\end{theorem}
Obviously, any choice of $q,Q_\bullet$ can be realized this way for
some $k,s_\bullet$, which are not unique; for any $g$ and $\vartheta$,
we can adjust the choice of parameters $ k,s_\bullet$ to yield an
block of the Cherednik category $\cO$ that matches the representations
of $\gls{cWF}^\vartheta (\pq,\PQ_{\bullet})$, using the process
of Uglovation discussed in \cite[Def. 2.8]{WebRou}.
  This is also useful to consider as a common generalization of all
  the algebras we have considered.  
Given a collection ${\mathscr{D}}$ of subsets of $\R$, we'll let $\gls{tWF}^\vartheta_{\mathscr{D}}
(\pq,\PQ_{\bullet}),\gls{cWF}_{\mathscr{D}}^\vartheta (\pq,\PQ_{\bullet})$
denote the subalgebras of $\gls{tWF}^\vartheta
(\pq,\PQ_{\bullet}), \gls{cWF}^\vartheta (\pq,\PQ_{\bullet})$ spanned by
diagrams whose tops and bottoms lie in the set ${\mathscr{D}}$. 

As in earlier cases, the algebra  $\gls{tWF}^\vartheta  (\pq,\PQ_{\bullet})$
  is equipped with a polynomial representation $P^\vartheta$ using the rules of
  Proposition \ref{W-poly} for diagrams only involving green strands and
  Proposition \ref{tilde-poly} for basic diagrams involving red and
  green strands.

  \subsubsection{Relation to cyclotomic Schur algebras}
\label{sec:relation-hecke-schur}

We can extend Theorem \ref{wdHecke} to this setting.  As before, let
$\wellsep=\{B_s=\{s,2s,3s,\dots, ns\}\}$ for $s$ some real
number with $s\gg |\ck|,|\vartheta_i|$. 
\begin{theorem}\label{wfHecke}
  There is an isomorphism of $\gls{cWF}^\vartheta_\wellsep
  (\pq,\PQ_{\bullet})$ to the rank $n$  cyclotomic affine Hecke algebra
 $\gls{cmH}(\pq,\PQ_\bullet)$ for the 
  parameters $\{\PQ_1,\dots, \PQ_\ell\}$.  
\end{theorem}
\begin{proof}
  First, since $s\gg |\vartheta_i|$, all strands start and end to the
  right of all red strands.  Thus, we have that every diagram can be
  written, using the relations, in terms of diagrams that remain to
  the right of all red strands.  Thus, we have a surjective map from
  the type W affine Hecke algebra $\waha_\cO$ onto  $\gls{cWF}^\vartheta_\wellsep
  (\pq,\PQ_{\bullet})$.  
By Theorem \ref{wdHecke}, we can identify $\waha_\cO$ with the usual
affine Hecke algebra $\hmH$. 

Now consider a diagram
where the first strand starts at $(s,0)$, goes linearly to
$(-s,\nicefrac 12)$ then back to $(s,1)$, while all others remain
straight. This diagram is unsteadied, since the horizontal slice at
$y=\nicefrac{1}2$ is unsteadied by the leftmost strand.  By the relation (\ref{qHcost}), this diagram is equal to $\prod_{i=1}^\ell
  (X_1-\PQ_i)$ which thus lies in the kernel of the map of the affine
  Hecke algebra to $\gls{cWF}^\vartheta_\wellsep
  (\pq,\PQ_{\bullet})$.   

  As in the proof of \ref{prop:cyclo-Hecke}, we can easily check that
  the diagram discussed above generates the kernel so $\gls{cWF}^\vartheta_\wellsep (\pq,\PQ_{\bullet})$ is isomorphic to
  this cyclotomic quotient.
\end{proof}

There is also a version of this theorem relating the type WF Hecke
algebras to cyclotomic Schur algebras.  Assume that the parameters
$\vartheta_i$ are ordered with $\vartheta_1<\dots <\vartheta_\ell$.
Fix a set $\Lambda$ of $\ell$-multicompositions of $n$ which is an
upper order ideal in dominance order.  We'll be interested in the
cyclotomic $q$-Schur algebra $\gls{cycSchur}$ of rank $n$  attached to the
data $(\pq,\PQ^\bullet)$ defined by Dipper, James and Mathas
\cite[6.1]{DJM}; let $\glsdisp{cycSchur}{\mathscr{S}^-(\Lambda)}$ be the signed version of
this algebra defined using signed permutation modules.

Let $r$ be
the maximum number of parts of one of the components of $\xi\in
\Lambda$.  Choose constants $\epsilon \ll \ck$ and $s$ so
that \[|\ck|+m\epsilon < s < \min_{k\neq
  n}(|\vartheta_k-\vartheta_n|/r);\] of course, this is only possible
is $r|\ck| <|\vartheta_k-\vartheta_n|$ for all $k\neq n$.  In this
case, we associate to every multicomposition $\xi\in \Lambda$ a subset $E_\xi$
that consists of the points $\vartheta_p+i\epsilon+js$ for every
$1\leq j\leq \xi^{(p)}_i$.

In order to simplify the proof below, we'll use some results from
\cite{WebRou}, in particular, a dimension calculation based on the
cellular basis constructed in \cite[Thm. 2.26]{WebRou}.  Since
\cite{WebRou} cites some of the results of this paper, the reader
might naturally worry that the author has created a loop of citations
and thus utilized circular reasoning.  However, we only use these in the
proof of Proposition \ref{cqs-morita}, which is not used in \cite{WebRou}.

As in
Section \ref{sec:weight-gener}, there is an idempotent diagram
$e'_\xi$ on this
subset where we act on the strands with $x$-value in $[
\vartheta_p+js,\vartheta_p+js+\epsilon\mu^{(p)}_j]$ the idempotent 
$y^{\mu_j-1}_1\cdots y_{\mu_j-1}\partial_{w_0}$.  Let
$e_\Lambda=\sum_{\xi\in \Lambda} e_\xi'$.  Let $\mathscr{D}$ be any
collection of $n$-element subsets containing $E_\xi$ for all $\xi\in \Lambda$.

\begin{proposition}\label{cqs-morita}
We have an isomorphism $\gls{cycSchur}\cong e_\Lambda \gls{cWF}^\vartheta_\mathscr{D}
  (\pq,\PQ_{\bullet})e_\Lambda $ if  $\ck<0$, and $\glsdisp{cycSchur}{\mathscr{S}^-(\Lambda)}\cong e_\Lambda \gls{cWF}^\vartheta_\mathscr{D}
  (\pq,\PQ_{\bullet})e_\Lambda $ if $\ck>0$.  If $\Lambda$ contains all
  $\ell$-multipartitions of $n$,  this subalgebra is Morita equivalent to $\gls{cWF}^\vartheta_\mathscr{D}
  (\pq,\PQ_{\bullet})$ via the obvious bimodule.
\end{proposition}
\begin{proof}
  For $t\gg 0$ sufficiently large, we have that $e_{D_{t,m}}\gls{cWF}^\vartheta_\mathscr{D} (\pq,\PQ_{\bullet}) e_{D_{t,m}}$ is the cyclotomic
  Hecke algebra $\gls{cmH}(\pq,\PQ_\bullet) $ by Theorem \ref{wfHecke}.
  Thus, we have that $e_{D_{t,m}}\gls{cWF}^\vartheta_\mathscr{D}
  (\pq,\PQ_{\bullet}) e_\Lambda $ is a bimodule over $\gls{cmH}(\pq,\PQ_\bullet)$
  and the algebra $e_\Lambda \gls{cWF}^\vartheta_\mathscr{D}
  (\pq,\PQ_{\bullet})e_\Lambda $.

Let $q_\xi$ be the diagram that linearly
  interpolates between $D_{t,m}$ and $E_\xi$, times $e'_\xi$ on the
  right.  We'll concentrate on the case where $\kappa<0$.  The same argument as the proof of Theorem \ref{waha-Schur}
  shows that $(T_i-q)q_\xi=0$ if the $i$th and $i+1$st strands lie in
  one of the segments $[
\vartheta_p+js,\vartheta_p+js+\epsilon\mu^{(p)}_j]$ in $E_\xi$. If
$\kappa>0$, we instead see that $(T_i+1)q_\xi=0$. Note
that $q_\xi$ generates $e_{D_{t,m}}\gls{cWF}^\vartheta_\mathscr{D}
  (\pq,\PQ_{\bullet}) e_\Lambda $ as a left module.

If $\xi^{(p)}=\emptyset$ for $p<\ell$, then this shows that sending
$m_\xi\mapsto q_\xi$ induces a map of $P_\xi$
to   $e_{D_{t,m}}\gls{cWF}^\vartheta_\mathscr{D}
  (\pq,\PQ_{\bullet}) e'_\xi$, which is surjective since $q_{\xi}$
  generates.  

For an arbitrary $\xi$, let $\xi^\circ$ be the multicomposition where
$(\xi^\circ)^{(p)}=\emptyset$ for $p<\ell$, and $(\xi^\circ)^{(\ell)}$
is the concatenation of $\xi^{(p)}$ for all $p$.  We have a natural
map $e_{D_{t,m}}\gls{cWF}^\vartheta_\mathscr{D}
  (\pq,\PQ_{\bullet}) e'_\xi\to e_{D_{t,m}}\gls{cWF}^\vartheta_\mathscr{D}
  (\pq,\PQ_{\bullet}) e'_{\xi^\circ}$ given by the straight-line
  diagram interpolating between $\xi$ and $\xi^\circ$.  Applying
  relation (\ref{qHcost}) many times, we find that this map sends
  \[q_{\xi}\mapsto\prod_{j\leq|\xi^{(1)}|+\cdots +|\xi^{(k-1)}|}
  (L_j-Q_k)q_{\xi_\circ}.\]  The submodule of $P_{\xi^\circ}$
  generated by this element is a copy of $P_{\xi}$, thus we have a
  surjective map $e_{D_{t,m}}\gls{cWF}^\vartheta_\mathscr{D}
  (\pq,\PQ_{\bullet}) e'_\xi\to P_{\xi}$.

  Dimension considerations show that this map is an isomorphism.  
The dimension of $e_{D_{t,m}}\gls{cWF}^\vartheta_\mathscr{D}
  (\pq,\PQ_{\bullet}) e'_\xi$ is $1/\xi! $ times the
  dimension of $e_{D_{t,m}}\gls{cWF}^\vartheta
  (\pq,\PQ_{\bullet}) e_{E_\xi}$, since $e_{E_\xi}$ is the sum of
  $\xi! $ orthogonal idempotents isomorphic to $e'_\xi$.  Thus, by
  \cite[Th. 2.26]{WebRou}, it is equal to $1/\xi!$ times the
number of pairs of tableaux of the same shape, one standard and of
type $E_{\xi}$.  
    The entries of an $E_{\xi}$-tableau are of the form 
$\vartheta_p+i\epsilon+js$ for $(i,j,p)$ a box of the diagram of
$\xi$.  A filling will be a $E_{\xi}$ if and only if the replacement
$\vartheta_p+i\epsilon+js\mapsto j_p$ is a semi-standard
tableau\footnote{This tableau uses $\ell$ alphabets (denoted using
  subscripts) with the order $1_1<2_1<3_1\cdots< 1_2<2_2<3_2\cdots<1_3<\cdots.$} increasing weakly along columns and strongly
  along rows if $\kappa>0$ and {\it vice versa} if $\kappa<0$.  In
  fact, this gives a $\xi!:=\prod \xi^{(p)}_k!$-to-$1$ map from
  $E_\xi$-tableau to semi-standard tableau of type $\xi$. 

Thus, the dimension of $e_{D_{t,m}}\gls{cWF}^\vartheta_\mathscr{D}
  (\pq,\PQ_{\bullet}) e'_\xi$ is the number of pairs of tableaux of
the same shape, one standard and one
semi-standard of type $\xi$.  This is the same as the dimension of the
permutation module associated to $\xi$, so the surjective map $e_{D_{t,m}}\gls{cWF}^\vartheta_\mathscr{D}
  (\pq,\PQ_{\bullet}) e'_\xi\to P_{\xi}$ must be an isomorphism.

We have from \cite[Lem. 3.3]{WebRou} that the map
\begin{equation}
e_\Lambda \gls{cWF}^\vartheta
  (\pq,\PQ_{\bullet})e_\Lambda \to \End_{\gls{cmH}(\pq,\PQ_\bullet)}(e_{D_{t,m}}\gls{cWF}^\vartheta_\mathscr{D}
  (\pq,\PQ_{\bullet}) e_\Lambda )\label{eq:2}
\end{equation}
is injective.  Applying \cite[Th. 2.26]{WebRou} again, the
dimension of $e_\Lambda \gls{cWF}^\vartheta_\mathscr{D}
(\pq,\PQ_{\bullet})e_\Lambda $ is equal to the number of pairs of
semi-standard tableaux of the same shape and (possibly different) type
in $\Lambda$.  Thus, the dimension coincides with $\dim
\gls{cycSchur}$. This shows that the injective map \eqref{eq:2} must be
an isomorphism.

  Finally, we wish to show that the bimodules $e_\Lambda  \gls{cWF}^\vartheta_\mathscr{D}
  (\pq,\PQ_{\bullet})$ and $\gls{cWF}^\vartheta_\mathscr{D}
  (\pq,\PQ_{\bullet})e_\Lambda $ induce a Morita equivalence.  For this, it
  suffices to show that no simple $\gls{cWF}^\vartheta_\mathscr{D}
  (\pq,\PQ_{\bullet})$-module is killed by $e_\Lambda $. If this were the
  case, $\gls{cWF}^\vartheta_\mathscr{D}
  (\pq,\PQ_{\bullet})$ would have strictly more simple modules than the
  cyclotomic $q$-Schur algebra. However, in
  \cite[Th. 2.26]{WebRou}, we show that this algebra is
  cellular with the number of cells equal to the number of
  $\ell$-multipartitions of $n$.  By \cite[6.16]{DJM}, this is the
  number of simples over  $\gls{cycSchur}$ as well.
\end{proof}
This also allows us to show:
\begin{theorem}\label{affine-schur-morita}
  The idempotent $\gls{eprime}$ induces a Morita equivalence between the affine
  Schur algebra $\gls{cS}_h(n,m)$ and the type W affine Hecke algebra
  $\waha_{{\gls{Cset}}}(\pq)$.  
\end{theorem}
\begin{proof}
  Since the algebra $\waha_{\mathscr{B}}(\pq)$ is Noetherian, if
  $\waha_{\mathscr{B}}(\pq)\gls{eprime} \waha_{\mathscr{B}}(\pq)\neq
  \waha_{\mathscr{B}}(\pq)$, then there is at least one simple
  module $L$ over $\waha_{\mathscr{B}}(\pq)/\waha_{\mathscr{B}}(\pq)\gls{eprime}
  \waha_{\mathscr{B}}(\pq)$, which must be killed by $\gls{eprime}$.  

This simple module must be finite dimensional since $\waha_{\mathscr{B}}(\pq)$ is of finite rank
  over the center of this module.  Thus, $X_1$ acting on this simple
  module satisfies some polynomial equation $p(X_1)=0$, and $L$ factors through the map to a type
  WF Hecke algebra $\gls{cWF}^\vartheta$ where we choose
  $\vartheta_i\ll \vartheta_{i+1}$ for all $i$, and $\vartheta_\ell\ll
  0$, with $Q_i$ being the roots of $p$ with multiplicity.

  By Proposition \ref{cqs-morita}, the identity of
  $\gls{cWF}^\vartheta$ can be written as a sum of cellular basis
  vectors factoring through the idempotent $e_{\xi}'$ at
  $y=\nicefrac 12$.  We have some choice in the definition of these
  vectors, and we can assure that all crossings in them occur to the
  right of all red line.  The relation (\ref{qHcost}) allows us to
  pull all strands to the right.  Once all the strands are to the
  right of all red lines, this slice at $y=\nicefrac 12$ will be the
  idempotent $e_{\xi^\circ}'$, times a polynomial in the dots.  Since
  this idempotent $e_{\xi^\circ}'$ lies in
  $\gls{eprime} \waha_{\mathscr{B}}(\pq)\gls{eprime}$, we must have that $\gls{eprime}$ acts
  non-trivially on $L$, contradicting our assumption.  This shows that
  $\waha_{\mathscr{B}}(\pq)\gls{eprime} \waha_{\mathscr{B}}(\pq)=
  \waha_{\mathscr{B}}(\pq)$, proving the Morita equivalence.
\end{proof}

\subsection{Weighted KLR algebras}
\label{sec:weight-klr-algebr-1}

There's also a KLR algebra in type WF.  This is also a weighted KLR
algebra as defined in \cite{WebwKLR}, but now for the Dynkin diagram
$\gls{U}$ with a Crawley-Boevey vertex added, as discussed in \cite[\S
3.1]{WebwKLR}.  
\begin{definition}\label{def:WF-KLR}
  A rank $n$ {\bf WF KLR diagram} is a wKLR diagram (as defined in Definition
  \ref{def:wKLR}, with labels in $\gls{U}$)
  with vertical red lines inserted at $x=\vartheta_i$. The diagram must avoid
  tangencies and triple points between any combination of these
  strands, green strands and ghosts, and only
  allow isotopies that preserve these conditions.  Here is an example
  of such a diagram:
  \[
\begin{tikzpicture}[baseline,very thick,xscale=2.5,yscale=1.4]
  \draw (-.5,-1) to[out=90,in=-90] node[below,at start]{$u_2$}  (-.9,0) to[out=90,in=-90](.5,1);
  \draw[dashed] (-1.25,-1) to[out=90,in=-90] (-1.65,0) to[out=90,in=-90](-.25,1);
  \draw (.5,-1) to[out=90,in=-90]node[below,at start]{$u_4$} (1,1);
    \draw[dashed] (-.25,-1) to[out=90,in=-90] (.25,1);
  \draw  (1,-1) to[out=90,in=-90] node[below,at start]{$u_5$}
  node[circle, midway,fill=black,inner sep=3pt]{} (0,1);
    \draw[dashed]  (.25,-1) to[out=90,in=-90] (-.75,1);
  \draw (-1, -1) to[out=90,in=-90] node[below,at start]{$u_1$} (-.5,0) to[out=90,in=-90] (-1,1);
  \draw (-.1,-1) to[out=90,in=-90] node[below,at start]{$u_3$} (-.5,1);
  \draw[dashed] (-1.75, -1) to[out=90,in=-90] (-1.25,0) to[out=90,in=-90] (-1.75,1);
  \draw[dashed] (-.85,-1) to[out=90,in=-90] (-1.25,1);
   \draw[wei](.125,-1) -- node[below,at start]{$\la_2$} (.125,1); \draw[wei](-.8,-1) --node[below,at start]{$\la_1$} (-.8,1);
    \draw[wei](1.325,-1) --node[below,at start]{$\la_3$} (1.325,1);
\end{tikzpicture}\]

The rank $n$  type WF KLR algebra $\gls{tdalg}^\vartheta (h,\Bz)$ is the algebra
generated by these diagrams over $\K[h,\Bz]$ modulo the local
relations (\ref{dots-1}--\ref{eq:KLRtriple-point2},\ref{cost}--\ref{red-triple}) and \begin{equation}\label{KLR-dumb-red-triple}
   \begin{tikzpicture}[very thick,baseline=-2pt,yscale=.8,]
    \draw [wei]  (0,-1) -- (0,1);
       \draw[dashed](.5,-1) to[out=90,in=-30] (-.5,1);
       \draw(-.5,-1) to[out=30,in=-90] (.5,1);
  \end{tikzpicture}= \begin{tikzpicture}[very thick,baseline=-2pt,yscale=.8,]
    \draw [wei]  (0,-1) -- (0,1);
       \draw[dashed](.5,-1) to[out=150,in=-90] (-.5,1);
       \draw(-.5,-1) to[out=90,in=-150] (.5,1);
  \end{tikzpicture}\qquad \qquad   \begin{tikzpicture}[very thick,baseline=-2pt,yscale=.8,]
    \draw [wei]  (0,-1) -- (0,1);
       \draw(.5,-1) to[out=90,in=-30] (-.5,1);
       \draw[dashed](-.5,-1) to[out=30,in=-90] (.5,1);
  \end{tikzpicture}= \begin{tikzpicture}[very thick,baseline=-2pt,yscale=.8]
    \draw [wei]  (0,-1) -- (0,1);
       \draw(.5,-1) to[out=150,in=-90] (-.5,1);
       \draw[dashed](-.5,-1) to[out=90,in=-150] (.5,1);
  \end{tikzpicture}\qquad \qquad   \begin{tikzpicture}[very thick,yscale=.8,baseline=-2pt]
    \draw [wei]  (0,-1) -- (0,1);
       \draw[dashed](.5,-1) to[out=90,in=-30] (-.5,1);
       \draw[dashed](-.5,-1) to[out=30,in=-90] (.5,1);
  \end{tikzpicture}= \begin{tikzpicture}[very thick,yscale=.8,baseline=-2pt]
    \draw [wei]  (0,-1) -- (0,1);
       \draw[dashed](.5,-1) to[out=150,in=-90] (-.5,1);
       \draw[dashed](-.5,-1) to[out=90,in=-150] (.5,1);
  \end{tikzpicture}.
\end{equation} This is a reduced
weighted KLR algebra for the Crawley-Boevey graph of $\gls{U}$ for the
highest weight $\la$.

The steadied quotient of $  \gls{cdalg}^\vartheta (h,\Bz)$ is the quotient of
$\gls{tdalg}^\vartheta (h,\Bz)$ by the 2-sided ideal generated by all
unsteady idempotents.
\end{definition}

As with the other algebras we've introduced, the algebra
$\gls{tdalg}^\vartheta (h,\Bz)$ has a natural polynomial
representation $P^\vartheta$,
defined in \cite[Prop. 2.7]{WebwKLR}.  It also has a
grading, with the degrees of diagrams given by \[
  \deg\tikz[baseline,very thick,scale=1.5]{\draw (.2,.3) --
    (-.2,-.1) node[at end,below, scale=.8]{$u$}; \draw (.2,-.1) --
    (-.2,.3) node[at start,below,scale=.8]{$v$};}
  =-2\delta_{u,v} \qquad \deg\tikz[baseline,very
  thick,scale=1.5]{\draw (0,.3) -- (0,-.1) node[at
    end,below,scale=.8]{$u$} node[midway,circle,fill=black,inner
    sep=2pt]{};}=2 \qquad\deg \tikz[baseline,very thick,scale=1.5]{\draw[densely dashed]
    (-.2,-.1)-- (.2,.3)  node[at start,below, scale=.8]{$u$}; \draw
    (.2,-.1) -- (-.2,.3)  node[at start,below, scale=.8]{$v$};} =\deg \tikz[baseline,very thick,scale=1.5]{\draw
    (-.2,-.1)-- (.2,.3)  node[at start,below, scale=.8]{$u$}; \draw[densely dashed]
    (.2,-.1) -- (-.2,.3)  node[at start,below, scale=.8]{$v$};}=\delta_{u,q^{\pm 1}v}
\]\[
  \deg\tikz[baseline,very thick,scale=1.5]{\draw[wei] (.2,.3) --
    (-.2,-.1) node[at end,below, scale=.8]{$\la$}; \draw (.2,-.1) --
    (-.2,.3) node[at start,below,scale=.8]{$u$};}
  =\deg\tikz[baseline,very thick,scale=1.5]{\draw (.2,.3) --
    (-.2,-.1) node[at end,below, scale=.8]{$u$}; \draw[wei] (.2,-.1) --
    (-.2,.3) node[at start,below,scale=.8]{$\la$};}=\al_u^{\vee}(\la),\]

\subsection{Isomorphisms}
\label{sec:isomorphisms-1}

As in all previous sections, we can show that both of the algebras we
have introduced have polynomial style representations (graded in the
case of $\gls{tdalg}^\vartheta (h,\Bz)  $) where
\begin{equation*}
\mathbb{K}=\K[[h]]\qquad A=\gls{tWF}^\vartheta (\pq,\PQ_{\bullet})\qquad
  B=\K[[h]][X^{\pm}]\qquad I=Bh+B\prod_{u\in \gls{U}}(X-u) \qquad P=\cP^{\vartheta}
\end{equation*}
\begin{equation*}
\mathbb{K}=\K[h]\qquad A=\gls{tdalg}^\vartheta (h,\Bz) \qquad
  B=\K[h,y]\qquad I=Bh+By \qquad P=P^{\vartheta}.
\end{equation*}
and thus have suitable completions $\widehat{\gls{tWF}}^\vartheta
(\pq,\PQ_{\bullet})$ and $\widehat{\gls{tdalg}}^\vartheta (h,\Bz) $.

Now, let $\Bu$ be a
loading on a set $D\in \mathscr{D}$, that is, a map $D\to \gls{U}$.  Let
 $u_1,\dots,u_n$ be the values of $\Bu$ read from left to right.  
Attached to such data, we have an idempotent $e_{\Bu}$ in $
 \gls{tdalg}^\vartheta_{\mathscr{D}} (h,\Bz)$ and another $\epsilon_{\Bu}$
 in
 $\widehat{\gls{tWF}}^\vartheta_{\mathscr{D}}(\pq,\PQ_{\bullet})
 $ given by projection to the stable kernel of $X_r-u_r$ for all $r$. 
\begin{theorem}\label{thm:Hecke-KLR}
  We have isomorphisms of $\K[h,\Bz]$-algebras
\[\gls{cWF}^\vartheta_{\mathscr{D}}(\pq,\PQ_{\bullet})\cong
  \gls{cdalg}^\vartheta_{\mathscr{D}} (h,\Bz)\qquad \widehat{\gls{tWF}}^\vartheta_{\mathscr{D}}(\pq,\PQ_{\bullet})\cong
  \widehat{\gls{tdalg}}^\vartheta_{\mathscr{D}} (h,\Bz)\] which send
\[\epsilon_{\Bu}\mapsto e_{\Bu}\qquad\qquad X_r\mapsto\sum_{\Bu}u_rb(y_r)e_{\Bu}\qquad \qquad \tikz[baseline,very thick,scale=1.5, green!50!black]{\draw (.2,.3) --
  (-.2,-.1); \draw [wei]
  (.2,-.1) -- (-.2,.3);} \mapsto\tikz[baseline,very
  thick,scale=1.5]{\draw (.2,.3) -- (-.2,-.1); \draw [wei] (.2,-.1) --
    (-.2,.3);}\]
 \vspace{-3mm}

\[\tikz[baseline,very thick,scale=1.5, green!50!black]{\draw [wei] (-.2,.3) --
  (.2,-.1); \draw 
  (-.2,-.1) -- (.2,.3);}  \epsilon_{\Bu} \mapsto
\frac{\displaystyle\prod_{\vartheta_s=o_k}(u_rb({y_r-z_s})-\PQ_s)}{\displaystyle\prod_{s\in
  S_{u_r,j}} (y_r-z_s)}\, \tikz[baseline,very thick,scale=1.5]{\draw [wei] (-.2,.3) --
  (.2,-.1); \draw 
  (-.2,-.1) -- (.2,.3);} e_{\Bu} 
\]
\[\subeqn\label{crossing-match2}
\tikz[baseline,very thick,scale=1.5, green!50!black]{\draw (.2,.3) --
  (-.2,-.1); \draw
  (.2,-.1) -- (-.2,.3);} \epsilon_{\Bu}\mapsto
\begin{cases}
\displaystyle\frac{1}{u_{r+1}b(y_{r+1})-u_rb(y_{r})}(\psi_r-1) e_{\Bu} & u_r\neq u_{r+1}\\
 \displaystyle \frac{y_{r+1}-y_r}{u_{r+1}(b(y_{r+1})-b(y_{r}))}\psi_r e_{\Bu}& u_r=u_{r+1}
\end{cases}
\]
\[\subeqn\label{ghost-match2}
\tikz[baseline,very thick,scale=1.5, green!50!black]{\draw[densely dashed] 
  (-.2,-.1)-- (.2,.3); \draw
  (.2,-.1) -- (-.2,.3);}\epsilon_{\Bu} \mapsto
\begin{cases}
\displaystyle u_rb(y_r)-\pq u_sb(y_s)\tikz[baseline,very thick,scale=1.5]{\draw[densely dashed]
    (-.2,-.1)-- (.2,.3); \draw
    (.2,-.1) -- (-.2,.3);} e_{\Bu}&
  u_r\neq qu_s\\ 
 \displaystyle \frac {u_rb(y_r)-\pq u_sb(y_s)}{y_{s}-y_{r}}\tikz[baseline,very thick,scale=1.5]{\draw[densely dashed]
    (-.2,-.1)-- (.2,.3); \draw
    (.2,-.1) -- (-.2,.3);}e_{\Bu}& u_r=qu_s, d(h)=1\\
 \displaystyle \frac {u_rb(y_r)-\pq u_sb(y_s)}{y_{s}-y_{r}+h}\tikz[baseline,very thick,scale=1.5]{\draw[densely dashed]
    (-.2,-.1)-- (.2,.3); \draw
    (.2,-.1) -- (-.2,.3);}e_{\Bu}& u_r=qu_s,d(h)=e^h
\end{cases}
\qquad \qquad\tikz[baseline,very thick,scale=1.5, green!50!black]{\draw (.2,.3) --
  (-.2,-.1); \draw [densely dashed]
  (.2,-.1) -- (-.2,.3);} \mapsto \tikz[baseline,very thick,scale=1.5]{\draw (.2,.3) --
  (-.2,-.1); \draw [densely dashed]
  (.2,-.1) -- (-.2,.3);} \] 
where the solid strand shown is the $r$th (and $r+1$st in the first line),
and the ghost is associated to the $s$th from the left.
\end{theorem}
\begin{proof}
  That this map sends unsteady idempotents to unsteady idempotents is
  clear, so we need only show that we have an isomorphism
  $\gls{tWF}^\vartheta_{\mathscr{D}}(\pq,\PQ_{\bullet})\cong
  \gls{tdalg}^\vartheta_{\mathscr{D}} (h,\Bz)$.
  As in the proofs of Theorems \ref{O-isomorphism},
  \ref{W-isomorphism}, and \ref{F-isomorphism}, we check this by
  comparing polynomial representations.  The comparison for diagrams
  involving no red strands is covered by the isomorphism of Theorem
  \ref{W-isomorphism} and for crossings with red strands is checked in
  Theorem \ref{F-isomorphism}.
\end{proof}
Just as in Section \ref{sec:weight-gener}, this isomorphism does not
immediately grade the cyclotomic $q$-Schur algebra, since the
idempotent from Theorem \ref{cqs-morita} does not have homogeneous
image.  One can, however,
define a homogenous idempotent $e''$ with isomorphic image.  As
before, $e''$ will be a sum over $\ell$-ordered lists of multi-subsets
of $\gls{U}$
whose size gives a multi-composition in $\Lambda$.  Each of these
contributes the idempotent where the points connected to the part
$\mu_i^{(s)}$ are labeled with the multi-subset, in increasing order,
with a primitive idempotent in the nilHecke algebra acting on the
groups with the same label.  
 
Note that in the level one case, a graded version of the $q$-Schur
algebra was defined by Ariki \cite{Arikiq}.  This grading was uniquely
determined by its compatibility with the Brundan-Kleshchev grading on
the Hecke algebra, so our algebra must match up to graded Morita
equivalence with that of \cite[3.17]{Arikiq} (just as we saw with the
closely related quiver Schur algebra in \cite[Th. 7.9]{SWschur}).

\glsclearpage
\printglossaries
\bibliography{./gen}

\def\cprime{$'$} \def\cprime{$'$}
\providecommand{\bysame}{\leavevmode\hbox to3em{\hrulefill}\thinspace}
\providecommand{\MR}{\relax\ifhmode\unskip\space\fi MR }
\providecommand{\MRhref}[2]{%
  \href{http://www.ams.org/mathscinet-getitem?mr=#1}{#2}
}
\providecommand{\href}[2]{#2}
\begin{thebibliography}{GGOR03}

\bibitem[AK94]{AK}
Susumu Ariki and Kazuhiko Koike, \emph{A {H}ecke algebra of {$({\mathbb
  Z}/r{\mathbb Z})\wr{\mathfrak S}\sb n$} and construction of its irreducible
  representations}, Adv. Math. \textbf{106} (1994), no.~2, 216--243.
  \MR{MR1279219 (95h:20006)}

\bibitem[Ari09]{Arikiq}
S.~Ariki, \emph{Graded $q$-{S}chur algebras}, 2009, \arxiv{0903.3453}.

\bibitem[BK09]{BKKL}
Jonathan Brundan and Alexander Kleshchev, \emph{{Blocks of cyclotomic Hecke
  algebras and Khovanov-Lauda algebras}}, Invent. Math. \textbf{178} (2009),
  451--484.

\bibitem[DJM98]{DJM}
Richard Dipper, Gordon James, and Andrew Mathas, \emph{Cyclotomic {$q$}-{S}chur
  algebras}, Math. Z. \textbf{229} (1998), no.~3, 385--416. \MR{MR1658581
  (2000a:20033)}

\bibitem[DM02]{DMmorita}
Richard Dipper and Andrew Mathas, \emph{Morita equivalences of {A}riki-{K}oike
  algebras}, Math. Z. \textbf{240} (2002), no.~3, 579--610. \MR{1924022
  (2003h:20011)}

\bibitem[GGOR03]{GGOR}
Victor Ginzburg, Nicolas Guay, Eric Opdam, and Rapha{\"e}l Rouquier, \emph{On
  the category {$\mathcal O$} for rational {C}herednik algebras}, Invent. Math.
  \textbf{154} (2003), no.~3, 617--651.

\bibitem[Gre99]{GreenSchur}
R.~M. Green, \emph{The affine {$q$}-{S}chur algebra}, J. Algebra \textbf{215}
  (1999), no.~2, 379--411. \MR{1686197 (2000b:20011)}

\bibitem[GTL13]{GTL}
Sachin Gautam and Valerio Toledano~Laredo, \emph{Yangians and quantum loop
  algebras}, Selecta Math. (N.S.) \textbf{19} (2013), no.~2, 271--336.

\bibitem[HM16]{HMsemi}
Jun Hu and Andrew Mathas, \emph{Seminormal forms and cyclotomic quiver {H}ecke
  algebras of type {$A$}}, Math. Ann. \textbf{364} (2016), no.~3-4, 1189--1254.
  \MR{3466865}

\bibitem[Jim86]{Jimbo}
Michio Jimbo, \emph{A {$q$}-analogue of {$U(\mathfrak{gl}(N+1))$}, {H}ecke
  algebra, and the {Y}ang-{B}axter equation}, Lett. Math. Phys. \textbf{11}
  (1986), no.~3, 247--252. \MR{841713}

\bibitem[KL09]{KLI}
Mikhail Khovanov and Aaron~D. Lauda, \emph{A diagrammatic approach to
  categorification of quantum groups. {I}}, Represent. Theory \textbf{13}
  (2009), 309--347.

\bibitem[KL11]{KLII}
\bysame, \emph{A diagrammatic approach to categorification of quantum groups
  {II}}, Trans. Amer. Math. Soc. \textbf{363} (2011), no.~5, 2685--2700.
  \MR{2763732 (2012a:17021)}

\bibitem[Lus89]{Lusgraded}
George Lusztig, \emph{Affine {H}ecke algebras and their graded version}, J.
  Amer. Math. Soc. \textbf{2} (1989), no.~3, 599--635.

\bibitem[Mac96]{MacBourbaki}
I.~G. Macdonald, \emph{Affine {H}ecke algebras and orthogonal polynomials},
  Ast\'{e}risque (1996), no.~237, Exp. No. 797, 4, 189--207, S\'{e}minaire
  Bourbaki, Vol. 1994/95. \MR{1423624}

\bibitem[Mac03]{MacHecke}
\bysame, \emph{Affine {H}ecke algebras and orthogonal polynomials}, Cambridge
  Tracts in Mathematics, vol. 157, Cambridge University Press, Cambridge, 2003.
  \MR{1976581}

\bibitem[MS]{MS18}
Ruslan Maksimau and Catharina Stroppel, \emph{{Higher level affine Schur and
  Hecke algebras}}, \arxiv{1805.02425}.

\bibitem[MS19]{MiSt}
Vanessa Miemietz and Catharina Stroppel, \emph{Affine quiver {S}chur algebras
  and {$p$}-adic {$GL_n$}}, Selecta Math. (N.S.) \textbf{25} (2019), no.~2,
  25:32. \MR{3948934}

\bibitem[Rou]{Rou2KM}
Raphael Rouquier, \emph{2-{K}ac-{M}oody algebras}, \arxiv{0812.5023}.

\bibitem[SW]{SWschur}
Catharina Stroppel and Ben Webster, \emph{Quiver {S}chur algebras and
  $q$-{F}ock space}, \arxiv{1110.1115}.

\bibitem[Ugl00]{Uglov}
D.~Uglov, \emph{Canonical bases of higher-level {$q$}-deformed {F}ock spaces
  and {K}azhdan-{L}usztig polynomials}, Physical combinatorics ({K}yoto, 1999),
  Progr. Math., vol. 191, Birkh\"auser, 2000, pp.~249--299.

\bibitem[Weba]{Webalt}
Ben Webster, \emph{{Representation theory of the cyclotomic Cherednik algebra
  via the Dunkl-Opdam subalgebra}}, \arxiv{1609.05494}.

\bibitem[Webb]{WebwKLR}
\bysame, \emph{Weighted {K}hovanov-{L}auda-{R}ouquier algebras}, {to appear in
  Documenta Mathematica, \arxiv{1209.2463}}.

\bibitem[Web17a]{Webmerged}
\bysame, \emph{Knot invariants and higher representation theory}, Mem. Amer.
  Math. Soc. \textbf{250} (2017), no.~1191, 141.

\bibitem[Web17b]{WebRou}
\bysame, \emph{Rouquier's conjecture and diagrammatic algebra}, Forum Math.
  Sigma \textbf{5} (2017), e27, 71. \MR{3732238}

\bibitem[Wei13]{Kbook}
Charles~A. Weibel, \emph{The {$K$}-book}, Graduate Studies in Mathematics, vol.
  145, American Mathematical Society, Providence, RI, 2013, An introduction to
  algebraic $K$-theory. \MR{3076731}

\end{thebibliography}
\bibliographystyle{amsalpha}

\end{document}